\documentclass[11pt]{amsart}
\usepackage[utf8]{inputenc}
\usepackage{amsmath}
\usepackage{amssymb}
\usepackage{amsthm}
\usepackage{thmtools}
\usepackage{thm-restate}
\usepackage{mathtools}
\usepackage{comment}
\usepackage[most]{tcolorbox}
\usepackage[autopunct=true]{csquotes}
\usepackage{xcolor}

\usepackage[export]{adjustbox}

\calclayout                           

\usepackage[hyperfootnotes=false, hypertexnames=false]{hyperref}
\hypersetup{
    colorlinks=true,
    linkcolor=blue,
    urlcolor=blue,
    citecolor=blue,
}

\usepackage{enumerate}
\usepackage{graphicx}
\usepackage{tikz-cd}
\usepackage[labelfont=bf]{caption}
\usepackage{accents}
\usepackage{color}

\usepackage{todonotes}

\usepackage{tcolorbox}
\usepackage[capitalise]{cleveref}

\usepackage[shortlabels]{enumitem}
\usepackage[normalem]{ulem}

\usetikzlibrary{arrows.meta, bending, positioning}

\DeclareMathOperator{\Hull}{Hull}

\DeclareMathOperator{\Aut}{Aut}

\DeclareMathOperator{\diam}{diam}

\DeclareMathOperator{\tr}{tr}

\DeclareMathOperator{\Span}{Span}

\DeclareMathOperator{\rank}{rk}
\DeclareMathOperator{\Stab}{Stab}
\DeclareMathOperator{\Haus}{H}
\DeclareMathOperator{\CH}{CH}

\DeclareMathOperator{\relint}{rel-int}
\DeclareMathOperator{\rpd}{\Pb(\Rb^d)}
\DeclareMathOperator{\RP}{\Pb(\Rb^d)}

\DeclareMathOperator{\SL}{SL}

\DeclareMathOperator{\SO}{SO}

\DeclareMathOperator{\PSL}{PSL}
\DeclareMathOperator{\PGL}{PGL}
\DeclareMathOperator{\PO}{PO}

\DeclareMathOperator{\Cc}{\mathcal{C}}

\DeclareMathOperator{\Fc}{\mathcal{F}}

\DeclareMathOperator{\Ic}{\mathcal{I}}
\DeclareMathOperator{\Kc}{\mathcal{K}}

\DeclareMathOperator{\Nc}{\mathcal{N}}
\DeclareMathOperator{\Oc}{\mathcal{O}}
\DeclareMathOperator{\Pc}{\mathcal{P}}

\DeclareMathOperator{\Sc}{\mathcal{S}}

\DeclareMathOperator{\Wc}{\mathcal{W}}
\DeclareMathOperator{\Vc}{\mathcal{V}}

\DeclareMathOperator{\Bb}{\mathbb{B}}

\DeclareMathOperator{\Hb}{\mathbb{H}}

\DeclareMathOperator{\Nb}{\mathbb{N}}
\DeclareMathOperator{\Pb}{\mathbb{P}}
\DeclareMathOperator{\Rb}{\mathbb{R}}

\DeclareMathOperator{\Zb}{\mathbb{Z}}

\DeclareMathOperator{\asrk}{asrk}

\newcommand{\abs}[1]{\left|#1\right|}

\newcommand{\norm}[1]{\left\|#1\right\|}
\newcommand{\wt}[1]{\widetilde{#1}}
\newcommand{\wh}[1]{\widehat{#1}}

\DeclareMathOperator{\com}{CoM}

\DeclareMathOperator{\bdry}{\partial \Omega}
\DeclareMathOperator{\hil}{d_{\Omega}}

\newcommand{\Isom}{\mathrm{Isom}}

\newcommand{\set}[1]{\left\{#1 \right\}}
\newcommand{\seq}[1]{\left(#1 \right)}

\renewcommand{\emptyset}{\varnothing}

\newcommand{\ctrlf}{\Lambda}
\DeclareMathOperator{\vrtx}{Vert}
\DeclareMathOperator{\rF}{r_{\rm PES}}
\DeclareMathOperator{\rEuc}{r_{\rm Euc}}
\DeclareMathOperator{\rk}{rk_{\Rb}}
\DeclareMathOperator{\omegalim}{\omega\text{-}\lim}

 \definecolor{orange-red}{rgb}{1.0, 0.27, 0.0}

\newcommand{\Seq}[1]{$(#1)_{n\in\mathbb{N}}$}
\newcommand{\wrt}{with respect to}
\newcommand{\pcd}{properly convex domain}
\newcommand{\lht}{local Hausdorff topology}

\newcommand{\omegaGen}[1]{\Omega^{(#1)}_{(\delta,D)}}

\newcommand{\ngh}[2]{\mathcal{N}_{#2}\left(#1\right)}
\newcommand{\clngh}[2]{\mathcal{N}_{\le#2}\left(#1\right)}

\numberwithin{equation}{section}

\theoremstyle{plain}
\newtheorem{proposition}{Proposition}[section]
\newtheorem{theorem}[proposition]{Theorem}
\newtheorem{lemma}[proposition]{Lemma}
\newtheorem{claim}{Claim}[proposition]
\newtheorem{corollary}[proposition]{Corollary}
\newtheorem{fact}[proposition]{Fact}

\newtheorem{question}[proposition]{Question}
\newtheorem{example}[proposition]{Example}
\newtheorem{definition}[proposition]{Definition}
\newtheorem{notation}[proposition]{Notation}
\newtheorem*{notation*}{Notation}
\newtheorem{observation}[proposition]{Observation}

\theoremstyle{remark}
\newtheorem{remark}[proposition]{Remark}
\newtheorem*{remark*}{Remark}

\usepackage[style=alphabetic,maxbibnames=20,safeinputenc,backend=biber,backref=true]{biblatex}
\title[Coarse higher medians and convex projective geometry]{Coarse higher medians, symmetric spaces, and convex projective geometry}
\author{Mitul Islam and Grazia Rago}
\date{}
\address{School of Mathematics, Tata Institute of Fundamental Research, Mumbai 400005}
\email{mitul@math.tifr.res.in}
\address{Dipartimento di Matematica, Universit\`{a} di Bologna, Bologna 40126}
\email{grazia.rago2@unibo.it}

\begin{document}

\begin{abstract} We introduce a notion of coarse $r$-median spaces that is a higher-rank analog of Bowditch's coarse median spaces. Our notion is stable under quasi-isometries and recovers Bowditch's coarse medians when $r$ equals 1. We prove that several families of higher rank-symmetric spaces admit coarse $r$-medians with $r$ being the rank of the symmetric space. In particular, our list includes plenty of examples that are known to not admit Bowditch's coarse medians. Our main tools come from convex projective geometry, and we prove the existence of coarse higher medians on all divisible as well as quasi-homogeneous properly convex domains.  
\end{abstract}

\maketitle


\section{Introduction}

Starting with the pioneering work of Gromov \cite{Gromov_1987}, hyperbolicity has played a central role in coarse geometry and geometric group theory. Even beyond the strictly hyperbolic setting, the presence of hyperbolic-like directions has driven much of the field over last several decades, through notions such as rank-one or  Morse geodesics \cite{WB1995,HS2014,MC2017}, acylindrical hyperbolicity \cite{DS2016}, and hierarchical hyperbolicity \cite{BHS2017,BHS2019}, to name a few. 

In contrast, many natural spaces of interest—most notably higher-rank symmetric spaces of non-compact type—exhibit a complete absence of hyperbolic-like directions. These spaces are characterized by an abundance of higher-dimensional flats: every geodesic is contained in a totally geodesic copy of $\Rb^k$  for  $k\geq 2$; hence the term \emph{higher-rank}. This feature often eludes the hyperbolicity-inspired frameworks mentioned above. Some higher-rank symmetric spaces like $\Hb^2\times \Hb^2$ (i.e. products of hyperbolic spaces) have been studied using the theory of coarse medians introduced by Bowditch (\cite{BowditchCoarse}, see  \cref{sec:equiv_with_bowditch_median}). 
But most of the higher-rank symmetric spaces, e.g. $\SL_d(\Rb)/\SO(d)$ for any $d\geq 3$, lie beyond the ambit of Bowditch's theory of coarse medians (\cite{Haettel}, see \cref{thm:haettel_no_coarse_1_median}).

In this paper, we introduce a notion of a \emph{coarse $r$-median structure}, which provides a natural higher-rank analogue of Bowditch's theory of coarse median spaces \cite{BowditchCoarse}. Our definition is stable under quasi-isometries and recovers Bowditch's coarse medians when $r=1$, thereby extending a central tool of coarse geometry beyond the setting of rank-one (and their products).

Bowditch's theory of coarse median provides a unified viewpoint on a large class of non-positively curved spaces -- Gromov hyperbolic spaces, mapping class groups of surfaces, finite-dimensional CAT(0) cube complexes, and products of such spaces. Bowditch's medians often have interesting representation-theoretic consequences for a group, e.g. obstruction to  \emph{Property (T)} \cite{CDH2010}. Behrstock-Minsky \cite{BM2011} used such medians to establish \emph{Rapid Decay (RD) property} for mapping class groups. More generally, any (finite rank) coarse median group  has \emph{property (RD)} \cite[Theorem 9.1]{bowditchEmbedding}.  In this context, a famous conjecture is that property (RD) holds for uniform lattices in $\SL_d(\Rb)$ (more generally, any semisimple Lie group) \cite{valette2002introduction}. 
Although \cite{Lafforgue} gave a positive answer for $\SL_3(\Rb)$  (see \cite{IC2016} for the other cases that are known), the conjecture remains open for $\SL_d(\Rb)$ for any $d\geq 4$. Our  quest for ``higher-median" structures on $\SL_d(\Rb)/\SO(d)$ for $d\geq 3$ is partly motivated by such questions.

Our main results show that these coarse $r$-medians arise naturally in a broad range of higher-rank spaces. In particular, we prove that all divisible as well as quasi-homogeneous properly convex domains admit canonical coarse $r$-median structures, yielding a large and geometrically rich class of examples beyond the scope of existing median-type theories (\cref{thm:main_coarse_r_median_on_pcd}). This includes many higher-rank symmetric spaces where classical coarse medians are known not to exist, e.g. $\SL_d(\Rb)/\SO(d)$ for $d\geq 3$ (see Theorems \ref{thm:coarse_median_irred_symm_sp} and \ref{thm:coarse_median_red_symm_sp}). Taken together, these results suggest that coarse $r$-medians capture an essential feature of higher-rank geometry, analogous to the role of coarse medians in the hyperbolic setting. 

Previously, Bader-Lazarovich \cite{BaderLazarovich2023} had investigated CAT(0) 2-complexes and observed properties analogous to 2-medians. They had wondered whether there is a theory of coarse 2-median spaces and if rank-2 symmetric spaces fits into this framework. We answer this in the affirmative. However, the original Bader-Lazarovich construction doesn't exactly fit into our framework, because they lack a certain kind of `weak convexity' (see \cref{sec:bader_lazarovich_examples}).

It is worth remarking that often, our notion of coarse $r$-medians capture finer information than rank of the symmetric space. For example, let  us compare two rank-2 symmetric spaces: the reducible $\Hb^2 \times \Hb^2$ and the irreducible $\SL_3(\Rb)/\SO(3)$. Let $r_0(X)$ denote the minimal $r$ for which $X$ admits a coarse $r$-median structure. Then, we show $r_0(\Hb^2 \times \Hb^2)=1$ while $r_0(\SL_3(\Rb)/\SO(3))=2$ (see \cref{sec:rank_vs_coarse_median}). This  discrepancy between $r_0$ and the rank of the symmetric space indicates that the coarse $r$-median structure can distinguish between the reducible rank-2 space $\Hb^2\times \Hb^2$ and the irreducible rank-2 space $\SL_3(\Rb)/\SO(3)$.

 In another direction, our work may also be interpreted as identifying a certain kind of higher-rank hyperbolicity phenomena. Previously, Kleiner-Lang \cite{kleinerlang} had studied several kinds of higher-rank hyperbolicity phenomena in symmetric spaces. Although both these works are guided by the same philosophy that hyperbolicity should manifest above the dimension of the maximal flats, but  our goals and methods are significantly  different. Our approach is geometric -- we  use the Hilbert metric geometry of symmetric spaces to build certain special convex simplices, and exploit the slimness of such simplices. On the other hand, \cite{kleinerlang} uses analytic methods to develop a general theory of higher rank hyperbolicity akin to Gromov's theory of hyperbolicity, e.g. stability of maximal quasi-flats, Morse lemma, slimness of quasi-simplices, etc. Finally, while our work is spiritually similar to the combinatorial higher-rank hyperbolicity of \cite{JorgensenLang}, the precise relationship between the two frameworks remain unclear and invites  further study (see \cref{sec:jorgensen-lang_hyp}).

\subsection*{Summary of the main results} Our approach to coarse $r$-medians is spiritually closer to Niblo-Wright-Zhang's perspective \cite{NWZ2021} on Bowditch's coarse medians, rather than the original  perspective of Bowditch in \cite{BowditchCoarse}. In this paper, our main tools will come from convex projective geometry -- properly convex domains, Hilbert metrics, and their convexity. 

\subsection{Convex projective geometry and symmetric spaces} 

A properly convex domain is an open set $\Omega \subset \RP$ such that $\overline{\Omega}$ (closure of $\Omega$ in $\RP$) projects to a compact convex domain in some affine chart. All our properly convex domains in the paper will have dimension at least one.
The projective symmetries of a properly convex domain $\Omega$ are encoded in its automorphism group  $\Aut(\Omega):=\{g \in \PGL_d(\Rb): g \Omega = \Omega \}$.
We call a properly convex domain \emph{divisible} if there exists a discrete subgroup $\Gamma < \Aut(\Omega)$ such that $\Omega/\Gamma$ is compact. 
We will work with a natural $\Aut(\Omega)$-invariant distance $\hil$ on $\Omega$ called the Hilbert metric. The Hilbert metrics are often not geodesically  unique   and the metric geometry of $(\Omega,\hil)$ is rarely CAT(0). But a major advantage of the Hilbert metric is that it plays well with the convexity that $\Omega$ inherits from the affine chart. For example, projective lines in $\Omega$ are geodesics for the Hilbert metric.

A prominent example of a properly convex domain is the projective ball $\Bb$ in $\Pb(\Rb^{d})$. In this case, $(\Bb,d_{\Bb})$ is the well-known Beltrami-Klein model of the real hyperbolic space $\Hb^{d-1}$. This is an example of a symmetric domain, i.e. the automorphism group, conjugate to $\PO(1,d-1)$, is reductive and acts transitively on $\Bb$.  Several other families of symmetric spaces, including $\SL_d(\Rb)/\SO(d)$ for $d\geq 3$, can be modeled using symmetric domains; see Sections \ref{sec:intro_appl_to_symm_space} \& \ref{sec:symmetric_domains} and \cref{table:projective_symmetric_domains}.  For $\SL_d(\Rb)/\SO(d)$, this symmetric domain has  a simple description. It is the projectivization of the cone of $d\times d$ real symmetric positive definite matrices. 

Due to the existence of uniform lattices in semi-simple Lie groups, all the symmetric domains are divisible. However, not all divisible convex domains are symmetric. In fact, non-symmetric non-hyperbolic domains exist in every dimension $n \geq 2$ \cite{YB2006,BDL2018,BV2025}. In this paper, we will work with a slight weakening of the divisibility condition. Following Zimmer \cite{Z2018}, we will call a properly convex domain $\Omega$ \emph{quasi-homogeneous} if there is a compact set $K\subset \Omega$ such that $\Aut(\Omega) \cdot K=\Omega$.

\subsection{Slim simplices in properly convex domains}
\label{sec:slim_simplices_hilbert_geom}

A properly convex domain $\Omega$ inherits a natural notion of \emph{convexity} from the affine chart that contains $\overline{\Omega}$. A set $C \subset \Omega$ convex in such an affine chart also satisfies a \emph{convexity property for the Hilbert metric}: for any $x,y \in C$, there exists a Hilbert metric geodesic (namely, the projective line in $\Omega$ joining $x$ and $y$) that is contained in $C$. 
To any non-empty set $X \subset \overline{\Omega}$, we can associate its convex hull 
\begin{align*}
    \CH_\Omega(X)=\bigcap \left\{ D : D \subset \overline{\Omega} \text{ is a convex set containing } X\right\}.
\end{align*}
We will use $\CH_{\Omega}(a_0,\dots,a_k)$ to denote the convex hull of the set $\{a_0,\dots,a_k\} \subset \overline{\Omega}$ (also see \cref{notn:CH}).

We will primarily focus on the study of \emph{compact $k$-simplices $S$ in $\Omega$}. They are convex hulls of $(k+1)$ points  that lie in $\Omega$. These $(k+1)$ points are called the \emph{vertices} of the simplex and we denote the corresponding ordered tuple $(x_0,\dots,x_{k+1})$ by $\vrtx(S)$. Depending on the relative position of these points, the dimension of the simplex would either equal $k$ or be strictly less than $k$. In the former case, we will call the simplex  \emph{non-degenerate} and in the latter case, we will call the simplex \emph{degenerate}. If a simplex $S$ has  $\vrtx(S)=(x_0,\dots,x_k)$, then the $i$-\emph{facet} $F^i(S)$ is the convex hull of 
all but the $i$-th vertex, i.e. $F^i(S):=\CH_\Omega(x_0,\dots,\hat{x}_i,\dots,x_k)$. For more details, see \cref{sec:defn_of_simplices} (also see \cref{fig:simplex_defn_example_1}). 

For compact simplices in $\Omega$, we introduce a slimness condition that generalizes the slim triangle condition from Gromov hyperbolic metric spaces \cite{GH2013}. Below $\ngh{C}{\delta}$ is the tubular $\delta$-neighborhood (in $\hil$) of the set $C$. We adopt the convention that $\ngh{C}{0}=C$.  
\begin{definition}
    \label{defn:slim_simplices}
    Suppose $\Omega$ is a properly convex domain and $\delta \geq 0$. We will say that a compact $k$-simplex $S$ is $\delta$-slim, if for all $0 \leq i \leq k$ we have $$
    F^i(S) \subset \bigcup_{\substack{j=0 \\ j\neq i}}^{k} \Nc_{\delta}(F^j(S)).
    $$
\end{definition}

Our first result is to relate the slimness of simplices in $\Omega$ to certain special convex subsets of $\Omega$ called properly embedded non-degenerate simplex, henceforth abbreviated as PES. A PES is the convex projective analogue of a totally geodesic flat in CAT(0) geometry. The precise definition is as follows: the convex hull $S:=\CH_{\Omega}(x_0,\dots,x_k)$ of $(k+1)$ points $x_0,\dots,x_k  \in \partial \Omega$ is called \emph{$k$-dimensional PES in $\Omega$} provided $S\cap \Omega \neq \emptyset$, all facets are contained in $\bdry$ (i.e. $\cup_{i=0}^kF^i(S) \subset \partial \Omega$), and $S$ is non-degenerate (i.e. $\dim\Pb(\Span S)=k$). For a $k$-dimensional PES $S$, we write $\dim(S)=k$. Observe that since the vertices (and facets) of a PES lie in $\partial \Omega$, a PES \emph{is not} a compact $k$-simplex in $\Omega$. Rather it is only a $k$-simplex in $\Omega$.  See \cref{sec:PES} and \cref{fig:simplex_defn_example_1} for details.

 As a PES is analogous to a CAT(0) flat, we define the projective simplex rank of a domain in analogy with the Euclidean rank of a CAT(0) space. 
\begin{definition}
     The \emph{projective simplex rank} of a properly convex domain $\Omega \subset \RP$, denoted by $\rF(\Omega)$, is the maximum dimension of any PES in $\Omega$, i.e. 
    \begin{equation*}
        \rF(\Omega)\coloneqq \sup \set{\dim(S): S \text{ is a PES in } \Omega}.
    \end{equation*}
\end{definition}
Note that $\rF(\Omega) \geq 1$, since any properly convex domain of dimension at least one contains bi-infinite geodesics. Furthermore, $\rF(\Omega)\le d-1$, with equality if and only if $\Omega$ is  $(d-1)$-projective simplex domain (\cref{defn:projective_simplex_domain}).  
On the other hand, the first author had introduced a notion of `rank one' convex projective manifolds $\Omega/\Gamma$ in \cite{I2024}. This rank $\rF(\Omega)$ is distinct from that notion; there are examples of `rank one' manifolds $\Omega/\Gamma$ where $\rF(\Omega)>2$ (see \cref{rem:isolated_simplices_rank_one}).

\begin{theorem}[\cref{proof:slim_simplex_above_pes_rank}]
\label{thm:slim_simplex_above_pes_rank}
    Suppose $\Omega$ is a quasi-homogeneous properly convex domain. Then there exists $\delta_0=\delta_0(\Omega)> 0$ such that: if $m \geq  \rF(\Omega)+1$ and $S$ is a compact $m$-simplex in $\Omega$, then $S$ is $\delta_0$-slim. 
\end{theorem}

\begin{remark}\label{rmk:slim_simplices_symmetric}
    Suppose $\Omega$ is the projective model  of the symmetric space $X_d:=\SL_d(\Rb)/\SO(d)$. Then $\rF(\Omega)=d-1$, the real rank of $\SL_d(\Rb)$ (see \cref{lem: flat rank for symmetric spaces}).  Let $D$ be the Riemannian distance on $\Omega$ associated to a  $\SL_d(\Rb)$-invariant Riemannian distance on $X_d$ (see \cref{sec:intro_appl_to_symm_space}). Then \cref{thm:slim_simplex_above_pes_rank} holds for this distance $D$ as well. That is, there exists $\wh{\delta}_0$ so that for any compact $m$-simplex $S$ in $\Omega$ with $m\geq d$,  $D(F^i(S),\cup_{j\neq i}F^j(S)) <\wh{\delta}_0$ for all $i=0,\dots,m$. This is  true for any symmetric space admitting a convex projective model, not just $X_d$. For the proof of this remark, see \cref{proof:slim_simplices_symmetric}. 
\end{remark}

This theorem should be interpreted as a higher-rank hyperbolicity result for the Hilbert metric: isometrically embedded $m$-simplices are slim whenever $m$ is strictly bigger than $\rF(\Omega)$. Indeed, $\rF(\Omega)=1$ is precisely the case where $(\Omega,\hil)$ is Gromov hyperbolic. In the paper, we actually prove a stronger equivalence result (\cref{prop: duality between properly embedded simplices and slim simplices}) which implies \cref{thm:slim_simplex_above_pes_rank}. Moreover, we prove a boundary characterization of slimnness of simplices (\cref{prop:boundary_char_of_slimness}). As a corollary of these results, we can recover a result of Benoist relating Gromov hyperbolicity of $(\Omega,\hil)$ to the topological boundary  $\partial \Omega$ of $\Omega$; see \cref{sec:proof_of_benoist_corollary}.
\begin{corollary}[{part of \cite[Theorem 1.1]{B2004}}]
\label{cor:benoist_hyperbolicity_result}
    Suppose $\Omega$ is a quasi-homogeneous properly convex domain. Then the following are equivalent:
    \begin{enumerate}
    \item $\rF(\Omega)=1$;
        \item $(\Omega,\hil)$ is a Gromov hyperbolic metric space; and
        \item there are no non-trivial projective line segments in $\partial \Omega$.
    \end{enumerate}
\end{corollary}

\subsection{Coarse $r$-medians in convex projective geometry}
\label{sec:coarse_medians_hilbert_geom}

By virtue of the convexity of the Hilbert metric, the higher-rank hyperbolicity flavour of \cref{thm:slim_simplex_above_pes_rank} could be massively strengthened. This strengthening is precisely the  notion of a coarse $r$-median on the domain $\Omega$. When $r=1$, our notion of coarse $r$-median coincides with the classical notion of coarse medians due to Bowditch (\cref{defn:bowditch_defn-coarse_median}). Hence, to motivate the notion of coarse $r$-medians, we start with the case of coarse 1-median (i.e. Bowditch's coarse median) in $\Hb^2$. 

Let $x_0,x_1,x_2 \in \Hb^2$ and let $T$ be the geodesic triangle in $\Hb^2$ spanned by these three points. It is a classical fact that $T$ is slim and most of its area is concentrated near the `centroid' of $T$. One way of capturing this `centroid' is by evaluating $\mu(x_0,x_1,x_2)$ where $\mu$ is Bowditch's median map (\cref{defn:bowditch_defn-coarse_median}). But a more geometric approach is to fix a $\delta >0$ sufficiently large and then take the intersection of the $\delta$-neighborhoods of the edges of $T$. This intersection coarsely approximates $\mu(x_0,x_1,x_2)$.

\subsubsection*{Coarse centroid} Motivated by this intuition and \cref{thm:slim_simplex_above_pes_rank}, we introduce an analogous notion of (coarse) `centroid' for simplices in $\Omega$: given a simplex $S$ spanned by $(k+1)$ points, we consider the intersection of $\delta$-neighborhoods of its facets $F^i(S)$. 

\begin{definition}
\label{defn:coarse_centroid}
      Suppose $\Omega \subset \RP$ is a properly convex domain, $\delta\geq 0$, and $k\in \Nb$. We define the $\delta$-centroid of $(k+1)$ points $x_0,\dots,x_k\in \Omega$ as the set  
    \begin{equation*}
         C_\delta(x_0,\dots,x_k):=\bigcap_{i=0}^k\Nc_{\delta}\left( \CH_\Omega(x_0,\dots,\wh{x}_i,\dots,x_k)\right).
    \end{equation*}
   The $\delta$-centroid of a compact $k$-simplex $S$ in $\Omega$ is defined to be the $\delta$-centroid of $\vrtx(S)$ and it will be naturally denoted by $C_\delta(S)$. 
\end{definition}

We prove in \cref{thm:slimness_gives_centroid} below that for $\delta>0$ sufficiently large, any compact $k$-simplex has a non-empty $\delta$-centroid. Our proof combines \cref{thm:slim_simplex_above_pes_rank} with certain other convexity properties that are specific to convex projective geometry. The main technical ingredient is a classical lemma from convex geometry (\cref{lem:KKM}); see \cref{sec:slim_simplices_non_empty_centroid}.

\begin{theorem}[\cref{sec:proof_of_slimness_gives_centroid}]
\label{thm:slimness_gives_centroid}
    Suppose $\Omega \subset \RP$ is a quasi-homogeneous properly convex domain, and $m\geq \rF(\Omega)+1$. There exists $\delta'_0>0$ such that, for any $(m+1)$ points $x_0,\dots,x_m \in \Omega$ and any $\delta \geq \delta_0'$,  $$C_{\frac{\delta}{2}}(x_0,\dots,x_m) \neq \emptyset. $$ 
    In particular, we may choose $\delta_0':=2\delta_0$, where $\delta_0$ is as in \cref{thm:slim_simplex_above_pes_rank}.
\end{theorem}

We must emphasize that for $m \geq 3$, going from \cref{thm:slim_simplex_above_pes_rank} to \cref{thm:slimness_gives_centroid} is a major upgrade. We do not see any reason, a priori, to expect the analogue of \cref{thm:slimness_gives_centroid} in a general metric space with `slim simplices' property when $m \geq 3$. However, we must also remark on the very special case of $m=2$. For any quasi-homogeneous properly convex domain $\Omega$, $\delta$-slimness of $2$-simplices is equivalent to the non-emptiness of $\frac{\delta}{2}$-centroid for $3$-tuples (see \cref{lem:non-empty centroid implies slimness}).

Now consider $r=m-1$ so that $r\geq \rF(\Omega)$. Fix any $\delta \geq \delta_0'$ as in \cref{thm:slimness_gives_centroid}. Then given any $(r+2)$-tuple $(x_0,\dots,x_{r+1})\in \Omega^{r+2}$, we would like to think of the existence of the `coarse centroid' $C_{\frac{\delta}{2}}(x_0,\dots,x_{r+2})$ as the \emph{coarse $r$-median structure} on $(\Omega,\hil)$. Moreover, we would like to think of 
\begin{equation}
\label{eqn:informal_coarse_r_median_map}
     \Omega^{r+2} \ni(x_0,\dots,x_{r+2}) \mapsto C_{\frac{\delta}{2}}(x_0,\dots,x_{r+2}) \in \Pc(\Omega) \setminus \{\emptyset,\Omega\}
\end{equation} as the \emph{coarse $r$-median map}, in analogy with Bowditch's median map. However there is a subtlety that arises here, already in the classical case $r=1$. Indeed, Bowditch's median map should take values in $\Omega$ while ours takes values in $\Pc(\Omega)$. In order to resolve this, we must prove a uniform bound the diameter of $C_{\frac{\delta}{2}}(x_0,x_1,x_2)$, where the bound depends only on $\delta$. As it turns out, this uniform bound does exist  when $r=1$. But the case $r\geq 2$ is a lot more complicated as we will now explain.

 \begin{figure}[h]
        \centering
        \includegraphics[scale=0.2]{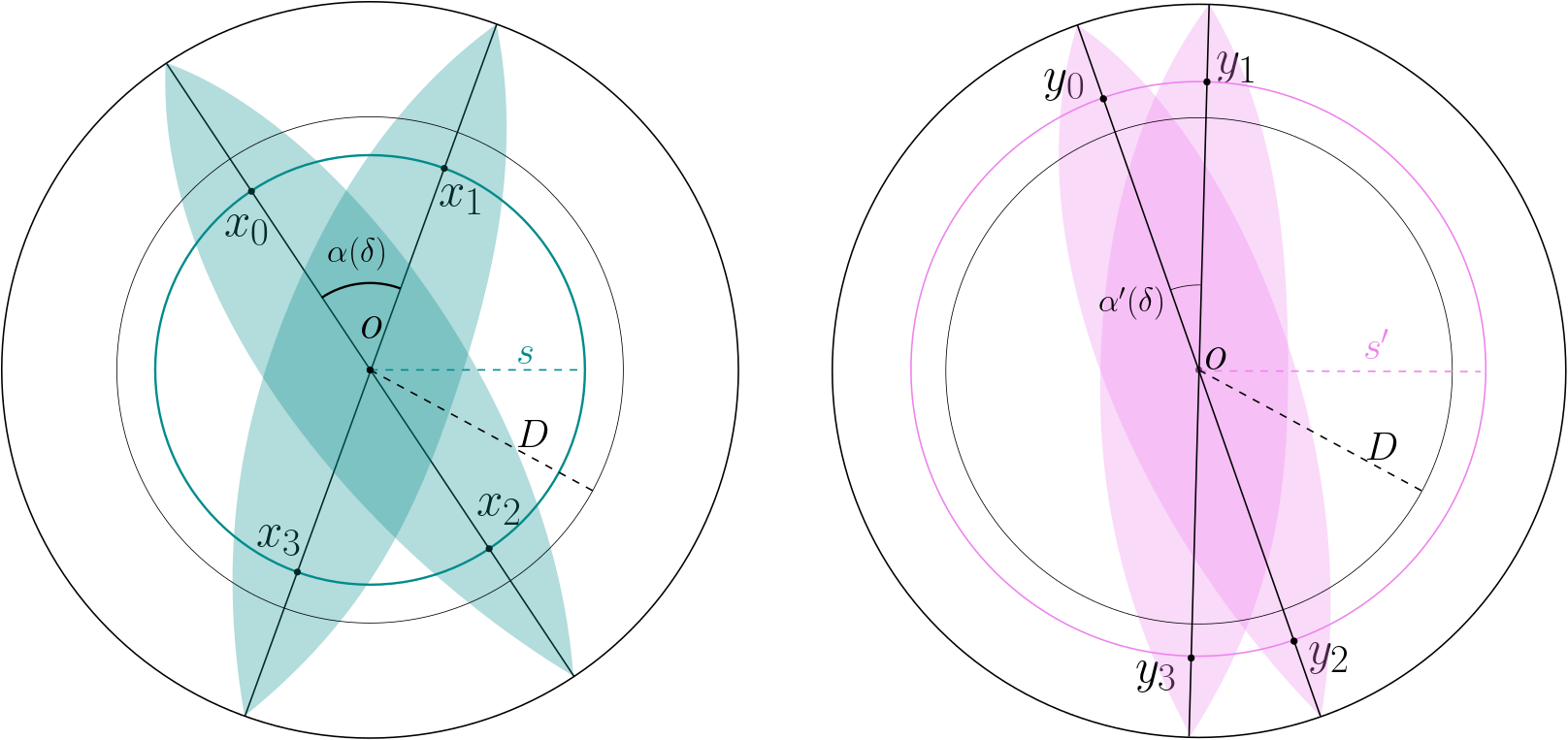}
        \caption{Two examples of 4-tuples in $\Hb^2$, all of them lying on a hyperbolic circle around $o$: (left) generic 4-tuple; (right) non-generic 4-tuple. See \cref{example: alg_non_generic} for details.}
        \label{fig: GenericityExample1}
\end{figure}

\subsubsection*{Generic points} Consider the projective model of $\Hb^2$ -- a unit disc centered at the origin $o$ in an affine chart of $\Rb\Pb^2$.
Let $r=2$ and $\delta>0$ be sufficiently large. Consider a 4-tuple $\mathbf{x}=(x_0,\dots,x_3)$ in $(\Hb^2)^4$ as in \cref{fig: GenericityExample1}. Here, the points $x_0,\dots,x_3$ lie on a hyperbolic circle around $o$ of radius $s(\mathbf{x})$. As we can observe from \cref{fig: GenericityExample1}, $\diam(C_{\frac{\delta}{2}}(x_0,\dots,x_3))$ depends on the angle $\alpha(\mathbf{x})$ at $o$ between $[x_0,x_2]$ and $[x_1,x_3]$ as well as the radius $s(\mathbf{x})$. As the picture on the right of \cref{fig: GenericityExample1} shows, we can move the points $x_0,\dots,x_3$ and make $\diam(C_{\frac{\delta}{2}}(x_0,\dots,x_3)) \to \infty$ while sending $\alpha(\mathbf{x}) \to 0$ and $s(\mathbf{x}) \to \infty$. However, in this case, the simplex spanned by $x_0,\dots,x_3$ increasingly resembles a bi-infinite geodesic in $\Hb^2$, that is, a strictly lower dimensional simplex. 
We will call such 4-tuples \emph{non-generic}, from the viewpoint of the Hilbert metric. On the other hand, the picture on the left of \cref{fig: GenericityExample1} shows that if $\alpha(\mathbf{x})$ is bounded away from $0$, then we can indeed get an upper bound on $\diam(C_{\frac{\delta}{2}}(x_0,\dots,x_3))$,  that depends only on $\delta$. We call such 4-tuples \emph{generic}. The precise definition of generic and non-generic points is below, see \cref{def: generic_points_in_Hilbert_new}. See \cref{sec: generic_pts_and_diam_bounds} for further examples and discussion. Below we use the notation: if $v_0,\dots,v_{r+2}\in \Omega$ and $I=\{j_1 < \dots < j_k\} \subset \{0,\dots,r+2\}$, then $\CH_{\Omega}(v_i:i\in I):=\CH_{\Omega}(v_{j_1},\dots,v_{j_k})$. 

\begin{definition}\label{def: generic_points_in_Hilbert_new}
    Let $\Omega\subseteq\RP$ be a properly convex domain and $\delta,D\ge0$.   
    Then, we say that the $(r+2)$-tuple  $(v_0,\dots,v_{r+1})\in\Omega^{r+2}$ is \emph{$(\delta,D)$-generic} provided the following holds: if $I_0,\dots,I_{q} \subset \{0,\dots,r+1\}$ are non-empty proper subsets  such that  
    \begin{equation*}
       \bigcap_{i=0}^{q}I_i=\emptyset ~\text{ and }~ \diam\left(\bigcap_{i=0}^{q}\Nc_{\delta} \bigg( \CH_\Omega\left( v_{j} : j\in I_i\right) \bigg) \right) >D, 
    \end{equation*}
    then we must have $|I_0|=\dots=|I_q|=r+1.$
    Otherwise, we say that the tuple $(v_0,\dots,v_{r+1})$ is \emph{$(\delta,D)$-non-generic}.
\end{definition}

In \cref{def: generic_points_in_Hilbert_new}, we have to talk about a quantitative notion of genericity, namely  $(\delta,D)$-genericity, since this is the way to formalize what we meant by $\alpha \to 0$ and $s \to \infty$ in the example above. Here $\delta$ is the parameter that we use to define $C_{\frac{\delta}{2}}(\cdot)$,  while $D$ determines coarse scale, beyond which $C_{\frac{\delta}{2}}(\cdot)$ stops resembling a ``point" to us. On the other hand, the example in \cref{fig: GenericityExample1} also indicates that the notion of genericity is closed related to ideal simplices in $\Omega$ and their degenerations. For example, the picture on the right in \cref{fig: GenericityExample1} could be seen as an ideal 4-simplex in $\Hb^2$ that is degenerating towards a bi-infinite geodesic, as the angle $\alpha \to 0$. For $r\geq 2$, the combinatorial data of such degenerations could be a lot more complex. One should think of the subsets $I_1,\dots,I_q$ of $\{0,\dots,r+1\}$ in \cref{def: generic_points_in_Hilbert_new} as the combinatorial data that keeps track of this. Finally we remark that there are plenty of generic points whenever $X$ is an infinite set: there are infinitely many $(\delta,D)$-generic tuples when $\delta>0$ and $D$ is sufficiently large (\cref{rmk:existence_generic_tuples_pcd}).

 The notion of $(\delta,D)$ genericity is important because the coarse $r$-median map of \cref{eqn:informal_coarse_r_median_map}, when restricted to this set of $(\delta,D)$-generic tuples, is essentially a map taking values in $\Omega$ (more precisely, it is within a bounded distance from  a $\Omega$-valued map, see \cref{rmk:relation_coarse_median_map_on_generic}(3)). The precise result is as follows. The main tools in its proof are Lemmas \ref{newlem: slim simplices have bounded centroid} and \ref{lem: slim simplices have bounded centroid_uniform_version}. 

\begin{proposition}[\cref{sec:proof_of_CM_2_for_pcd}]
\label{prop:CM_2_for_pcd}
    Suppose $\Omega\subset \RP$ is a quasi-homogeneous properly convex domain and $r \geq \rF(\Omega)$. There exist $\delta'_0 > 0$ 
    and a function $\Psi:[\delta'_0,\infty) \times [\delta'_0,\infty) \to [0,\infty)$,  
    non-decreasing in both coordinates, such that the following holds: if $\delta,D \geq \delta'_0$, and $(x_0,\dots,x_{r+1}) \in \Omega^{r+2}$ is a $(\delta,D)$-generic tuple, then $C_{\frac{\delta}{2}}(x_0,\dots,x_{r+1})\neq\emptyset$, and $$\diam_{\hil}\left(C_{\frac{\delta}{2}}(x_0,\dots,x_{r+1})\right) \leq \Psi(\delta,D).$$
\end{proposition}

\subsubsection*{Coarse $r$-median structure} With all the ingredients at our disposal, we will now sum up the above discussion. This summary will serve as an explanation for why the metric space $(\Omega,\hil)$ carries  a coarse $r$-median structure. In fact, our discussion here will serve as the guiding principle for the general definition of a coarse $r$-median structure on any metric space (\cref{defn:coarse_r_median}). For the rest of this discussion, fix a quasi-homogeneous properly convex domain $\Omega \subset\RP$ and $r= \rF(\Omega)$.

\begin{enumerate}[leftmargin=*,label=(\alph*)]
\item \emph{Any $(k+1)$ points $x_0,\dots,x_k \in \Omega$ can be `filled' by certain filling maps $F_k:\Omega^{k+1}\to \Pc(\Omega)\setminus\{\emptyset,\Omega\}$ defined via
$$F_k(x_0,\dots,x_k):=\CH_{\Omega}(x_0,\dots,x_k).$$
Taking these filling maps together, we think of them as the family $$F:=\{F_i:\Omega^{i+1} \to \Pc(\Omega)\setminus\{\emptyset,\Omega\}\}_{i=0,\dots,r}.$$
Moreover, these filling maps have several properties that one naturally expects: $F_0(x_0)=\{x_0\}$; $F_k(x_0,\dots,x_{k}) \subseteq F_{k+1}(x_0,\dots,x_{k+1})$ with equality if $x_k=x_{k+1}$; $F_k$ is invariant under permutation of inputs;  and $F_k$ has convex image.}
\end{enumerate}
Indeed, the existence of these filling maps is a special structure and is not available in any metric space. In fact, we believe that such geometrically meaningful `fillings' are key to the existence of a coarse median structure.  These filling maps interact the metric geometry of $\hil$ in the following particularly interesting ways -- let $C_{\delta}(x_0,\dots,x_{r+1}):=\bigcap_{i=0}^{r+1}\ngh{F_r(x_0,\dots,\wh{x_i},\dots,x_{r+1})}{\delta},$ then:

\begin{enumerate}[leftmargin=*,label=(\alph*),resume]
\item \emph{\cref{thm:slimness_gives_centroid}: There is a $\delta_0>0$ such that the coarse centroid $C_{\frac{\delta_0}{2}}(x_0,\dots,x_{r+1})$ of any $(r+2)$-tuple is non-empty. }
\item \emph{\cref{prop:CM_2_for_pcd}: There is a uniform bound $\Psi(\delta,D)$ on $\diam \left( C_{\frac{\delta}{2}}(x_0,\dots,x_{r+1}) \right)$ for any $(\delta,D)$-generic $(r+2)$-tuple $(x_0,\dots,x_{r+1})$, whenever $\delta,D \geq \delta_0$.}  
\end{enumerate}
When we say that the metric space $(\Omega,\hil)$ has a \emph{coarse $r$-median structure}, we mean that \emph{$(\Omega,\hil)$ can be equipped with the pair  $(F,\Psi)$ -- where $F$ is as in (a)  and $\Psi$ is as in (c) respectively --  such that the pair satisfies the points (a), (b), and (c) above.} The precise statement will be \cref{thm:main_coarse_r_median_on_pcd}, which we will state after giving the general definition of coarse $r$-median space. Further we must remark that although our above discussion was only for $r=\rF(\Omega)$, our \cref{thm:main_coarse_r_median_on_pcd} is more general and we can put a coarse $r$-median structure on $(\Omega,\hil)$ for all $r \geq \rF(\Omega)$.

\subsubsection*{Coarse $r$-median map} The coarse $r$-median structure on $(\Omega,\hil)$ lets us define an \emph{$\Aut(\Omega)$-invariant $r$-median map} in the following way (details in \cref{sec: Coarse rF-median map}): for any $(\delta,D)$-generic tuple $(x_0,\dots,\allowbreak x_{r+1})$ with $\delta,D$ sufficiently large, set 
\begin{equation*}
    \mu_{\delta}(x_0,\dots,x_{r+1}):=\com\left(C_{\frac{\delta}{2}}(x_0,\dots,x_{r+1})\right).
\end{equation*}
Here, $\com(\cdot)$ is the center of mass for the Hilbert metric (\cref{sec:com}). When $\Omega \cong \SL_d(\Rb)/\SO(d)$, this $\mu_{\delta}$ is a $\SL_d(\Rb)/\SO(d)$-valued map that is $\SL_d(\Rb)$-invariant. We may interpret this as a $\SL_d(\Rb)$-invariant ``barycenter" map on $(r+2)$-tuples in $\SL_d(\Rb)/\SO(d)$. Recently, Bucher-Savini \cite{BucherSavini} have also constructed barycenter maps, but the relationship between their maps and our $\mu_{\delta}$ is unclear.

\subsection{Definition of coarse $r$-median structure on a general metric space}
\label{sec:coarse_medians_general_defn}
Our notion of a coarse $r$-median structure is essentially a coarsening of the conditions $(a), (b), \text{ and }(c)$ from above. We say a metric space $(X,d)$ has a \emph{coarse $r$-median structure} if one can find a family of coarse ``filling" functions $F=\{ F_i:X^{i+1}\to \Pc(X) \setminus\{\emptyset,X\}\}_{i=0,\dots,r}$ on $X$ such that $(a),(b) \text{ and }(c)$ from above are satisfied. However, the conditions  -- $(a),(b) \text{ and }(c)$ -- all must be reinterpreted in a coarse sense for this purpose. 

First, a general metric space doesn't have a useful notion of convex hull. So, the convexity of the image of $F_k$ in $(a)$ must be reinterpreted in a coarse geometric sense. This is done by introducing a control functions $\ctrlf:[0,\infty) \to [0,\infty)$. The same applies for the set containment conditions in $(a)$. The result of this coarsening process is the notion of a \emph{coarse $r$-filling}  $(F,\ctrlf)$ on $(X,d)$, see \cref{defn:coarse_filling_weak}. On a proper geodesic metric space $(X,d)$, $\ctrlf$ can be chosen to be affine (recall that affine functions on $\Rb$ are of the form $x \mapsto ax+b$). The notion of \emph{$(\delta,D)$-generic tuples} can then be easily generalized; see \cref{def: generic_tuples_general}.

Next, in $(b)$ and $(c)$, we need to fix a non-constant affine function $\lambda:[\delta_0,\infty) \to [0,\infty)$ and look at $C_{\lambda(\delta)}(\cdot)$ instead of $C_{\delta}(\cdot)$. The extra flexibility offered by this affine function $\lambda$ is crucial in making the coarse $r$-median structure  quasi-isometry invariant. These new function $\ctrlf$ and $\lambda$ are the ingredients needed for the general definition. 
The \emph{coarse $r$-median structure on $(X,d)$} is then defined to be $(F,\ctrlf,\delta_0,\lambda,\Psi)$ where $(F,\ctrlf)$ is a coarse $r$-filling; $\delta_0 >0$; $\lambda$ is a non-constant affine function; and $\Psi$ is the diameter-bounding function such that the `coarsened' versions of $(a),(b) \text{ and } (c)$ are satisfied. See \cref{defn:coarse_r_median} for the precise definition.

With the definition in place, we can now state our main result about coarse $r$-medians on properly convex domains.

\begin{theorem}[\cref{sec: proof_of_pcd_are_coarse_median}]\label{thm:main_coarse_r_median_on_pcd}
Suppose $\Omega$ is a quasi-homogeneous properly convex domain. Then $\Omega$ admits a coarse $r$-median, for any $r \ge \rF(\Omega)$.
\end{theorem}

The coarse $r$-median structure that we construct is invariant under the action of $\Aut(\Omega)<\Isom(\Omega,\hil)$. However, we emphasize that, in general, a coarse $r$-median structure has no reason to be invariant under  isometries of the metric.

\subsubsection{Quasi-isometry invariance} \label{sec:intro_qi}The coarse $r$-median structure on a metric space is a quasi-isometry invariant, i.e. if two spaces $(X,d)$ and $(X',d')$ are quasi-isometric, then one of them admits a coarse $r$-median if and only if the other one does as well. In particular, this implies that to prove the existence of a coarse median on $(X,d)$, we may work with any other metric $d_1$ in the quasi-isometry class of $d$. See \cref{prop: median quasi-isometry}.

\subsubsection{Generalization of Bowditch's coarse medians}(\cref{sec:equiv_with_bowditch_median}) Our notion truly generalizes Bowditch's classical notion of coarse medians. Indeed, coarse 1-medians in our sense, with affine parameters, exactly coincides with Bowditch's notion (\cref{prop:equiv_coarse_1_median}). One should also compare it with the coarse interval structure of \cite{NWZ2021} (\cref{defn:nwz_coarse_interval}). Moreover, we discover higher medians on spaces that do not admit Bowditch's coarse medians, e.g. a coarse 2-median on $\SL_3(\Rb)/\SO(3)$ (compare \cref{thm:haettel_no_coarse_1_median}). 

\subsubsection{Products} \label{sec:intro_prod} Product of coarse $r$-median spaces are coarse $r$-median (\cref{prop: product of coarse medians}).

\subsubsection{Applications to symmetric spaces}\label{sec:intro_appl_to_symm_space} The quasi-isometry invariance leads us to a corollary of \cref{thm:main_coarse_r_median_on_pcd} regarding coarse $r$-medians on irreducible symmetric spaces. Recall that an irreducible symmetric space of non-compact type is a space of the form $G/K$ where $G$ is a connected simple Lie group with trivial center and $K$ is a maximal compact subgroup of $G$ (see for instance \cite[Chapter 2]{E1996}). We say that such a $G/K$ \emph{admits a convex projective model} if $G/K$ is $G$-equivariantly diffeomorphic to a  properly convex domain $\Omega$, i.e. there is an isomorphism $f: G \to \Aut(\Omega)^0$  and a diffeomorphism $\phi: G/K \to \Omega$ such that $\phi (gx)=f(g)\phi(x)$ for all $x \in G/K$ and $g\in G$. Here $\Aut(\Omega)^0$ is the identity component of  $\Aut(\Omega)$. Such symmetric spaces essentially correspond to the irreducible symmetric domains (\cref{rem:irred_sym_sp_and_domain}). They are completely classified; see \cref{table:projective_symmetric_domains} for the full list. 

Suppose $\Omega$ is the convex projective model for $G/K$. As $G/K$ is an irreducible symmetric space, it carries a $G$-invariant Riemannian distance $d$ that is unique up to scaling. This induces an $\Aut(\Omega)$-invariant Riemannian distance $d_f$ on $\Omega$ defined by: $d_f(x,y):=d(f^{-1}(x),f^{-1}(y))$. We call $d_f$ \emph{`the' associated $G$-invariant Riemannian distance on $\Omega$} (also unique up to scaling). The Hilbert metric on $\Omega$ is in the same quasi-isometry class as this associated distance $d_f$ (\cref{lem:riem_QI_hilbert}). Hence:

\begin{theorem}[\cref{sec:appl_to_symm_spaces}]\label{thm:coarse_median_irred_symm_sp}
Suppose $G/K$ is an irreducible symmetric space of non-compact type with $d$ denoting its (unique up to scaling) $G$-invariant Riemannian distance function. In addition, suppose $G/K$ admits a convex projective model (see \cref{table:projective_symmetric_domains} for the complete list). Then $(G/K,d)$ admits a coarse $r$-median for any $r \geq \rank_{\Rb}(G)$, where $\rank_{\Rb}(G)$ denotes the real rank of $G$.  
    
    In particular, this implies that $\SL_n(\Rb)/\SO(n)$ admits a coarse $(n-1)$-median for all $n \geq 2$.
\end{theorem}

In the case of reducible symmetric spaces where each factor admits a convex projective model, we have the following.

\begin{theorem}[\cref{sec:appl_to_symm_spaces}]
    \label{thm:coarse_median_red_symm_sp}
    Consider the symmetric space of non-compact type $Y=\prod_{i=1}^n\left( G_i/K_i \right)$ such that each $G_i/K_i$ satisfies the assumptions of \cref{thm:coarse_median_irred_symm_sp}. Let $d$ be a $(\prod_{i=1}^n G_i)$-invariant Riemannian distance function on $Y$. Then $(Y,d)$ admits a coarse $r$-median for any $r \geq r_*$, where $ r_*:= \max_{1 \leq i \leq n}\rank_{\Rb}(G_i)$.
\end{theorem}

\subsubsection{Asymptotic cones} \label{sec:intro_asym_cone}
Many affine buildings arise as asymptotic cones of higher-rank symmetric spaces. Given our Theorems \ref{thm:coarse_median_irred_symm_sp} and \ref{thm:coarse_median_red_symm_sp}, it is natural to investigate the existence of coarse $r$-medians on affine buildings. We prove that the coarse $r$-median structure can be passed onto asymptotic cones, under a certain technical condition on the function $\Psi(\delta,D)$.  
\begin{proposition}[\cref{cor:coarse_median_on_asymp_cone}]\label{prop:median_on_asym_cone_intro}
    Suppose $(X,d)$ is a geodesic metric space equipped with a coarse $r$-median $(F,\ctrlf,\delta_0,\lambda,\Psi)$ that is $\Psi$-affine (i.e. the function $(\delta,D) \mapsto \Psi(\delta,D)$ is affine in both coordinates). Then any asymptotic cone of $X$ admits an $r$-median structure.
\end{proposition}

Higher-rank symmetric spaces are geodesic spaces and our \cref{thm:coarse_median_irred_symm_sp} produces coarse $r$-medians on many of them. Then, by virtue of \cref{prop:median_on_asym_cone_intro}, it suffices to verify the $\Psi$-affine condition to get a $r$-median structure on the corresponding affine buildings. However this turns out be challenging and we could not verify the $\Psi$-affine condition.  So, for example, we do not know whether $\tilde{A}_2$ buildings admit a coarse $2$-median. However, we must remark that ours is only a sufficient condition -- there could be other independent constructions that do not require this  assumption.

\subsection{Rank of a symmetric space and coarse higher medians}
\label{sec:rank_vs_coarse_median}
The coarse $r$-median on a symmetric space often records finer information than the rank of the symmetric space. To illustrate this, let  us compare and contrast two examples of rank-2 symmetric spaces: the reducible example $\Hb^2 \times \Hb^2$ and the irreducible example $\SL_3(\Rb)/\SO(3)$. It is a classical fact that $\Hb^2 \times \Hb^2$ admits a coarse $1$-median (\cite{BowditchCoarse} and \cref{prop:equiv_coarse_1_median}). On the other hand, by \cref{thm:coarse_median_irred_symm_sp}, $\SL_3(\Rb)/\SO(3)$ admits a coarse $2$-median. However, $\SL_3(\Rb)/\SO(3)$ cannot admit a coarse $1$-median. This non-existence of coarse $1$-median follows from the following theorem of Haettel.
\begin{theorem}[\cite{Haettel}]
\label{thm:haettel_no_coarse_1_median}
    Let $X$ be a symmetric space of non-compact type and $d$ be an invariant Riemannian distance on $X$. There exists a coarse $1$-median on $(X,d)$ if and only if $X=\prod_{i=1}^n X_i$ where each $X_i$ is a rank one symmetric space of non-compact type and $n \geq 1$.
\end{theorem}

To a metric space $(X,d)$, we now associate constant $r_0(X,d)$ which is the minimal value of $r$ for which $(X,d)$ admits a coarse $r$-median. Then, by the above discussion, $$r_0(\Hb^2\times\Hb^2,d_{\Hb^2\times\Hb^2})=1 \text{ and } r_0(\SL_3(\Rb)/\SO(3),d_{\SL_3(\Rb)/\SO(3)})=2,$$ where $d_X$ denotes the natural invariant Riemannian distance on $X$ in both examples. In the reducible case of $\Hb^2 \times \Hb^2$, rank equals 2 while $r_0=1$. This maybe interpreted as the coarse median structure detecting the finer information that $\Hb^2\times \Hb^2$ is a product of two underlying rank-1 spaces. In the irreducible case of $\SL_3(\Rb)/\SO(3)$, $r_0$ coincides with rank.

The equality of $r_0$ with rank in the irreducible cases of $\SL_2(\Rb)$ and $\SL_3(\Rb)$ case inspires the following question whose answer is unknown to the authors.
\begin{question}
\label{ques:r0_vs_rPES}
    Is $r_0(\SL_d(\Rb)/\SO(d),d_{\SL_d(\Rb)/\SO(d)})=\rF(\SL_d(\Rb)/\SO(d))=d-1$ for all $d\geq 2$?
\end{question}
Indeed, by \cref{thm:coarse_median_irred_symm_sp}, we know that $r_0 \geq d-1$. So the missing ingredient is a theorem on non-existence of coarse $r$-medians when $r<d-1$, akin to Haettel's result for coarse $1$-medians. Such a non-existence result, if true, would be very special to the symmetric space case. Indeed, the analogue of \cref{ques:r0_vs_rPES} (i.e. the equality of $r_0$ and $\rF$) for a general divisible domain has a negative answer. There are divisible domains where $r_0(\Omega,\hil)=1$ but $\rF(\Omega) \geq 2$; see \cref{sec:coarse_1_median_in_conv_proj_geom}.

\subsection{Non-existence of coarse $r$-medians} 
\label{sec:bridson_discussion}
While discussing non-existence of coarse $r$-medians, it is extremely important to remember the quasi-isometry class of the metric in question.  Note that for any $d\geq 3$, there exists a diffeomorphism $\Theta_d: \SL_d(\Rb)/\SO(d) \to \Rb^{\frac{d^2+d-2}{2}}$. Moreover $(\Rb^k,\ell_1)$ always admits a coarse $1$-median \cite{BowditchCoarse}. Then, with the pull-back metric $d_1=\Theta_d^* \circ \ell_1$, $(\SL_d(\Rb)/\SO(d), d_1)$ admits a coarse $1$-median. However, the catch is that $\Theta_d$ is an arbitrary homeomorphism and is not $\SL_d(\Rb)$-invariant. In absence of this $\SL_d(\Rb)$-invariance,  $d_1$ is no longer in the quasi-isometry class of the invariant Riemannian distance on $\SL_d(\Rb)/\SO(d)$. Hence, this coarse 1-median is not of interest to us. In fact, \cref{thm:haettel_no_coarse_1_median} implies that such these abstract diffeomorphisms $\Theta_d$ will never be $\SL_d(\Rb)$ invariant. We thank Martin Bridson for the conversation that led us to \cref{sec:bridson_discussion}.

\subsection{Factor rank} In \cite{BowditchCoarse}, Bowditch discusses the notion of `rank' for coarse 1-median spaces. To avoid any confusion with other notions of rank in this paper, we will call this the `factor rank'. The name `factor rank' is motivated by the following example: both $\Hb^2$ and $\Hb^2 \times \Hb^2$ are coarse 1-medians spaces, but $\Hb^2$ has rank 1 while $\Hb^2 \times \Hb^2$ has rank 2. In the case of coarse $r$-medians as well, it is natural to expect an analogue. For example, let $Z:=\SL_4(\Rb)/\SO(4)$ and note that both $Z$ and $Z \times Z$  are coarse 3-median spaces. There should be a notion of `factor rank' that takes the value 1 for $Z$, and 2 for $Z \times Z$. However, this `factor rank' will be distinct from the projective simplex rank, e.g. $\rF(Z)=3$ while its factor rank should be 1. However, we do not yet have a definition of this `factor rank' for coarse $r$-median spaces. We hope to revisit this in a later paper.

\subsection{Bader-Lazarovich examples} \label{sec:bader_lazarovich_examples}
 In \cite{BaderLazarovich2023}, Bader-Lazarovich  give a well-defined notion of 2-median for CAT(0) polygonal complexes. Similar to our approach, they introduce 2-fillings that play well with the ambient CAT(0) metric. In particular, for any 4-tuple, they show that the coarse centroid is either an interval (`non-generic' configuration in our parlance) or a single point (`generic' configuration in our parlance). In spite of these similarities, their work doesn't exactly fit into our framework. To be precise, the 2-fillings they construct, fail to satisfy our weak convexity condition $(F4)$ of \cref{defn:coarse_filling_weak}; see \cite[Example 4.2]{BaderLazarovich2023}.

\subsection{J\o{}rgensen-Lang higher-rank hyperbolicity} \label{sec:jorgensen-lang_hyp}
J\o{}rgensen-Lang \cite{JorgensenLang} introduced a combinatorial notion of higher-rank $(n,*)$- hyperbolicity. The relationship between this notion and our coarse $r$-medians becomes nebulous, once we venture beyond the Gromov hyperbolic world.
For instance, consider a quasi-homogeneous properly convex domain $(\Omega,\hil)$ and let $r=\rF(\Omega)$. Here, the property of having slim $(r+1)$-simplices may intuitively suggest that the space is $(r,*)$-hyperbolic. However, this expectation is obstructed by the geometry of PES, as we now explain.

By \cite[Prop. 2.6]{JorgensenLang}, a normed vector space is $(n,*)$-hyperbolic if and only if it is finite-dimensional with a polyhedral norm, in which case the minimal such $n$ equals the number of pairs of opposite facets of the unit ball. Then a 2d PES in $\Omega$ is $(3,*)$-hyperbolic, but not $(2,*)$-hyperbolic. This is because the Hilbert metric on a 2d PES is induced by a hexagonal norm.  Now consider a divisible properly convex domain $\Omega$ with $\rF(\Omega)=2$ \cite{YB2006}. Then $(\Omega,\hil)$ cannot be $(2,*)$-hyperbolic (see above) but it admits a coarse 2-median (Th. \ref{thm:main_coarse_r_median_on_pcd}).

\subsection*{Acknowledgements} MI thanks Martin Bridson for enlightening conversations, especially about \cref{sec:bridson_discussion}. GR thanks Stefano Francaviglia and Beatrice Pozzetti for many illuminating discussions. The authors thank Thomas Haettel, Graham Niblo, Jiawen Zhang, Pierre-Louis Blayac, and Elia Fioravanti for their comments. MI and GR acknowledge support from DFG  338644254 and INdAM-GNSAGA respectively. Both authors thank Anna Wienhard and MPI-MIS (Max Planck Institute for Mathematics in the Sciences) for facilitating this collaboration. This work started when both authors were at MPI-MIS. MI also thanks University of Bologna and the Conference on Geometry \& Dynamics 2025 (TIFR-Mumbai and London Mathematical Society global meeting). 

\subsection*{Notational conventions} 
\hypertarget{link:notn_conv}
Suppose $(X,d)$ is a metric space. 
\begin{itemize}
    \item If $x \in X$ and $r>0$, we write $B_d^X(x,r):=\{ w \in X \mid d(x,y)<r\}$.
    \item  If $\emptyset \neq F \subset X$ and $r>0$, $\Nc^X_r(F):=\{ w \in X \mid d(w,F)<r\}$. 
    \item We set the convention that $\Nc^X_0(F)=F$ for any $\emptyset \neq F \subset X$.
    \item If $\emptyset \neq A \subset X$, $\diam_d(A):= \sup \{ d(x,y): x,y\in A \}$, and $\diam_d(\emptyset)=0$. 

\end{itemize}   
If $(X,d)$ is clear from context, we will write $B(x,r)$ (or $B_d(x,r)$), $\Nc_r(F)$, and $\diam(A)$.

\section{Preliminaries on Convex Projective Geometry}

\subsection{Properly convex domains and convexity}  

Recall that a \emph{properly convex domain} is a non-empty open set $\Omega\subset\RP$ that is a bounded convex domain in some affine chart of $\RP$. We equip $\Omega$ with the subspace topology from $\RP$ and denote by $\overline{\Omega}$ its closure in $\RP$. Further $\partial \Omega:=\overline{\Omega}-\Omega$. 

Given any distinct $x,y \in \RP$, there are two connected components of $\Pb(\Span\{x,y\})\setminus\{x,y\}$, so that the notion of a `projective line segment joining $x$ and $y$' is not well-defined. However, in presence of a properly convex domain $\Omega$, we can make a well-defined choice. 
For any $x,y \in \overline{\Omega}$, there is exactly one connected component of $\Pb(\Span\{x,y\})-\{x,y\}$ which intersects $\overline{\Omega}$. Let $[x,y]$ denote the closure of this unique connected component that intersects $\overline{\Omega}$. We call $[x,y]$ \emph{the} `projective segment joining $x$ and $y$'.  Further, we call $(x,y):=[x,y]-\{x,y\}$ \emph{the} `open projective segment joining $x$ and $y$'. By convention, we set $[x,x]=\{x\}$ and $(x,x)=\emptyset$.

 Using this notion of projective segments joining points in $\overline{\Omega}$, we have a well-defined notion of convex hull of points in $\overline{\Omega}$. 
\begin{definition}\label{def: convex_hull_pcd}
    For any non-empty subset $X$ of a properly convex domain $\Omega$, we denote by $\CH_\Omega(X)$ the intersection of all convex subsets of $\overline{\Omega}$ that contains $X$. 
\end{definition}
In general $\CH_\Omega(X)$ is a subset of $\overline{\Omega}$ and it may lie entirely in $\partial \Omega$ (see \cref{fig:simplex_defn_example_1}(a)). If we need to extract the part that lies in $\Omega$, we will explicitly write $\CH_\Omega(X) \cap \Omega$.
\begin{notation}
\label{notn:CH}
    For the sake of cleaner notation, we will at times abuse the convex hull notation in the following ways:
    \begin{itemize}
        \item for a $(k+1)$-tuple $(x_0,\dots,x_k)\in \overline{\Omega}^{k+1}$, we will write $\CH_{\Omega}(x_0,\dots,x_k)$ to denote $\CH_{\Omega}(\{x_0,\dots,x_k\})$.
        \item if $A,B \subset \overline{\Omega}$, we will write $\CH_{\Omega}(A,B)$ to denote $\CH_{\Omega}(A \cup B)$. In particular, if $A=\{a\}$, $\CH_{\Omega}(a,B):=\CH_{\Omega}(\{a\} \cup B)$.   
    \end{itemize}
\end{notation}

\subsection{Hilbert metric}
 Given $x,y\in\Omega$, let $x_{\infty},y_{\infty} \in \partial \Omega$ be such that $\Pb(\Span\{x,y\}) \cap \overline{\Omega}=[x_{\infty},y_{\infty}]$ and, along the projective segment $[x_\infty,y_\infty]$, the points are ordered as $x_\infty,x,y,y_\infty$. Then the distance between $x$ and $y$ in the \emph{Hilbert metric} $\hil$ is defined as
$$
\hil(x,y)=\dfrac{1}{2}\log{[x_\infty,x,y,y_\infty]},
$$
where $[x_\infty,x,y,y_\infty]=\dfrac{\abs{y-x_\infty}}{\abs{x-x_\infty}}\dfrac{\abs{x-y_\infty}}{\abs{y-y_\infty}}$ is the cross-ratio of the four points lying on a copy of $\Rb\Pb^1$. Note that the metric depends only on a choice of cross-ratio $[\cdot,\cdot,\cdot,\cdot]$ on $\Rb\Pb^1$ and we fix this choice once and for all in this paper.

 Moreover, the metric space $(\Omega,\hil)$ (sometimes called a Hilbert geometry) is a complete and proper metric space whose induced topology coincides with the subspace topology from $\RP$. Since projective transformations preserve the cross-ratio of four points on a $\Rb\Pb^1$, the automorphism group $\Aut(\Omega)$ acts by isometries on $(\Omega,\hil)$.

For the Hilbert metric, the projective segments $[x,y]$ are geodesics joining $x$ and $y$. However, in general, there could be many more geodesics for $\hil$ that join $x$ and $y$. The uniqueness of geodesics depends on the geometry of the boundary $\partial\Omega$.

\begin{proposition}[{\cite[Proposition 2]{deLaHarpe}}]\label{prop: uniquely geodesic}
    Let $\Omega\subset\rpd$ be a \pcd\ and $x \neq y\in\Omega$. Let $x_\infty,y_\infty\in\partial\Omega$ be such that $[x_\infty,y_\infty]=\Pb(\Span\{x,y\}) \cap \overline{\Omega}$. Then $[x,y]$ is the unique $\hil$-geodesic joining  $x$ and $y$ if and only if $x_\infty$ and $y_\infty$ are not contained in coplanar open projective segments in $\partial\Omega$.
\end{proposition}

Furthermore, we have the following lemma about extendability of projective geodesics. 
\begin{lemma}
\label{lem:top_char_bdry_points}
    Suppose $\Omega\subset \RP$ is a properly convex domain, $z \in \Omega$, and $x \in \overline{\Omega}$. Then $x \in \Omega$
    if and only if there exists $y \in \overline{\Omega}$ such that $x \in [z,y)$.
\end{lemma}
\begin{proof}
    Without loss of generality, we may assume that $z \neq x$ so that $\Omega \cap \Pb(\Span\{z,x\})$ is a  1-dimensional properly convex domain. Any 1-dimensional properly convex domain  is projectively equivalent to $\{[t:1]| 0<t<1\}$ in $\Pb(\Rb^2)$. Hence the result.
\end{proof}

\subsection{Closure of a set in a properly convex domain} Fix a properly convex domain $\Omega\subset \RP$. If $X \subset \overline{\Omega}$, then there are two natural ways of taking its closure --  either  in $\overline{\Omega}$ or in $\RP$. We observe that they coincide. 

\begin{observation}
\label{obs:closure_of_subset_equivalent}
    If $X \subset \overline{\Omega}$, then the closure of $X$ in $\overline{\Omega}$ coincides with the closure of $X$ in $\RP$. We will denote by $\overline{X}$ the closure of $X$ in $\RP$. 
\end{observation}
\begin{proof}
    Ley $Y$ denote the closure of $X$ in $\RP$. Indeed, it suffices to show that $Y\subset \overline{\Omega}$. As $X \subset \overline{\Omega}$ and $\overline{\Omega}$ is a closed subset of $\RP$, it is obvious that $Y \subset \overline{\Omega}$. 
\end{proof}
\begin{definition}
\label{defn:closure_of_X_in_clOm} We say that $X\subset \overline{\Omega}$ is \emph{closed in }$\overline{\Omega}$ provided $\overline{X}=X$.
\end{definition}
If $Y\subset \Omega$, there is yet another natural way of taking its closure -- in the subspace topology of $\Omega$. We will not use this notion much, but we record it for completeness.

\begin{definition} 
\label{defn:closure_of_X_in_Om}
We say that $Y\subset \Omega$ is \emph{closed in $\Omega$} provided $\Omega \setminus Y$ is open. In fact, $Y \subset \Omega$ is closed in $\Omega$ if and only if $Y=\overline{Y}\cap \Omega$. 
\end{definition}

\subsection{Faces in the boundary}

As we already mentioned, in a \pcd\ $\Omega\subset\RP$, the geometry of the boundary $\partial\Omega$ heavily influences the Hilbert metric $\hil$. To analyze this geometry, it is useful to decompose the boundary into relatively open convex subsets. For $x\in\overline{\Omega}$, we denote the \emph{open face of $x$} as 
$$
F_\Omega(x)=\{x\}\cup\{y\in\overline{\Omega}\mid x\text{ and }y\text{ are contained in an open segment of }\overline{\Omega}\}.
$$
For a non-empty set $X\subset \overline{\Omega}$, we will use the notation $F_{\Omega}(X)=\cup_{x \in X}F_{\Omega}(x)$. 

These open faces satisfy the following properties.
\begin{fact}\label{fact: behaviour of faces of Omega}
Let $\Omega\subset\RP$ be a \pcd. Then, for any $x,y\in\overline{\Omega}$:
\begin{enumerate}
    \item $y\in F_\Omega(x)$ if and only if $F_\Omega(y)=F_\Omega(x)$.
    \item  If $y\in\partial F_\Omega(x)$, then $F_\Omega(y)\subset \partial F_\Omega(x)$.
    \item If $z\in(x,y)$, then $(p,q)\subset F_\Omega(z)$ for all $p\in F_\Omega(x)$ and $q\in F_\Omega(y)$.
\end{enumerate}
\end{fact}
These properties allow us to define an equivalence relation on $\overline{\Omega}$. Given $x,y\in\overline{\Omega}$, we say that $x\sim_\Omega y$ if and only if $F_\Omega(x)=F_\Omega(y)$. 
Then, an \emph{open face} of $\Omega$ is the equivalence class of a point \wrt\ $\sim_\Omega$. Hence, we get a partition of $\overline{\Omega}$ into the union of its open faces. In particular, $F_{\Omega}(x)=\Omega$ if and only if $x \in \Omega$. When $x \in \partial \Omega$, $F_{\Omega}(x) \subset \partial \Omega$ and we will call these the \emph{open faces of $\Omega$ in the boundary}. 

Note that any open face $F_{\Omega}(x)$ is a properly convex domain in $\Pb(\Span F_{\Omega}(x))$. So it can be equipped with its own Hilbert metric $d_{F_{\Omega}(x)}$. Using these, we can define an extended distance on $\overline{\Omega}$ as follows: given $x, y \in \overline{\Omega}$,
$$
\operatorname{d}_{\overline{\Omega}}(x,y)=\begin{cases}
    \operatorname{d}_{F_\Omega(x)}(x,y),&\text{if } F_\Omega(x)=F_\Omega(y),\\
    +\infty, &\text{otherwise}.
\end{cases}
$$
We now list some facts about this extended Hilbert distance. 

\begin{fact}[{\cite[Lemma 2.11]{IslamWeisman}}]\label{fact: lower semicontinuity of extended distance}
Let $\Omega\subset\RP$ be a \pcd. Let \Seq{x_n},\Seq{y_n}$\subset\Omega$ be two sequences of points. Suppose that \Seq{x_n} and \Seq{y_n} converge, respectively, to $x_\infty\in\partial\Omega$ and $y_\infty\in\partial\Omega$. Then,
$$
\operatorname{d}_{\overline{\Omega}}(x_\infty,y_\infty)\le\liminf_{n\to\infty}\hil(x_n,y_n).
$$
In particular, if $\liminf_{n\to\infty}\hil(x_n,y_n) <\infty$, then $x_{\infty} \in F_{\Omega}(y_\infty)$. 
\end{fact}

\begin{fact}[{\cite[Lemma 2.12]{IslamWeisman}}]\label{fact: maximum principle for extended distance}
Let $\Omega\subset\RP$ be a \pcd, $x_1,x_2\in\overline{\Omega}$, $y_1\in F_\Omega(x_1)$, and $y_2\in F_\Omega(x_2)$. Then,
$$
{\operatorname{d}_{\overline{\Omega}}}^{\Haus}\left((x_1,x_2),(y_1,y_2)\right)\le\max\left\{ {\operatorname{d}}_{F_\Omega(x_1)}(x_1,y_1),{\operatorname{d}}_{F_\Omega(x_2)}(x_2,y_2) \right\}.
$$
\end{fact}

\subsection{Relative interior and boundary} Next we discuss topological and metric properties of convex subset of $\overline{\Omega}$.

\begin{definition}
  Suppose $\Omega \subset \RP$ is a properly convex domain and $C\subset \overline{\Omega}$ is a convex subset. 
  \begin{enumerate}
      \item The \emph{relative interior of $C$}, denoted by $\relint(C)$, is the set of all points in $C$ that are open in $\Pb(\Span C)$. That is, $x \in \relint(C)$ if and only if there is an open set $U \subset \RP$ containing $x$ such that $U \cap \Pb(\Span C) \subset C$.
      \item The set $C$ is called \emph{relatively open} if $\relint(C)=C$.
      \item The \emph{boundary} of $C$ is $\partial C:=\overline{C}-\relint(C).$
  \end{enumerate}   
\end{definition}

\begin{lemma}
\label{lem:openness_of_relint}
    Suppose $C \subset \overline{\Omega}$ is a convex set and $x , y$ are two distinct points in $\relint(C)$. Then there is an open line segment in $C$ containing both $x$ and $y$.
\end{lemma}
\begin{proof}
    By convexity, $[x,y] \subset C$. We will show that we can extend this projective segment at both ends while staying inside $C$. As $x \in \relint(C)$,  there is an open set $U_x \subset \RP$ containing $x$ such that $U_x \cap \Pb(\Span \{x,y\}) \subset C$. Hence we can extend $x'\in C$ such that $[x,y] \subsetneq (x',y]$ and $(x',y] \subset C$. Thus we can extend $[x,y]$ beyond $x$ inside $C$. Extending beyond $y$ works by a similar reasoning. 
\end{proof}

The relative interior of a convex set has an interesting dichotomoy. 
 
\begin{lemma}
\label{lem:relint_dichotomy}
   If $C\subset \overline{\Omega}$ is a closed convex set, then either $C \subset \partial\Omega$ or $\relint(C) \subset \Omega$.
\end{lemma}
\begin{proof}
    Suppose $C \not \subset \partial \Omega$ and let $c_0 \in C \cap \Omega$. Let $x \in \relint(C)$. If $x=c_0$, then $x \in \Omega$. So assume that $x \neq c_0$. By definition of the relative interior, there is an open set $U \subset \RP$ containing $x$ such that $U \cap \Pb(\Span C) \subset C$. Then $U \cap \Pb(\Span\{c_0,x\})=(x',x'')$ for some $x',x'' \in C \subset \overline{\Omega}$. As $x \in U$, $x \in (x',x'')$. As $x',x'' \in \Pb(\Span\{c_0,x\})$, we can switch the labels $x'$ and $x''$ if necessary and assume that the three distinct points $c_0,x,x''$ appear in this order as we travel from $c_0$ towards $x$. Thus we have found $x'' \in \overline{\Omega}$ such that $x\in [c_0,x'')$. Then \cref{lem:top_char_bdry_points} implies that $x \in \Omega$. This finishes the proof.
\end{proof}

We now analyze the case where $\relint(C) \subset \bdry$.
\begin{lemma}
\label{lem:relint_and_faces}
    Suppose $C\subset \overline{\Omega}$ is a closed convex set such that $\relint(C)\cap F_{\Omega}(x) \neq \emptyset$ for some $x \in \bdry$. Then $\relint(C) \subset F_{\Omega}(x)$. 
\end{lemma}
\begin{proof}
    Let $c \in \relint(C)\cap F_{\Omega}(x)$.  If $\relint(C)=\{c\}$, then we are done. If $c' \in \relint(C)-\{c\}$, then \cref{lem:openness_of_relint} implies that there is an open projective segment in $C \subset \overline{\Omega}$ containing $c$ and $c'$. As $c \in F_{\Omega}(x)$, \cref{fact: behaviour of faces of Omega} implies that $c' \in F_{\Omega}(x)$. Hence the proof.
\end{proof}

\subsection{Center of mass}
\label{sec:com}
In a properly convex domain, one can associate a \textquote{center of mass} to any compact subset as follows. Denote by ${\Kc}_d$ the set of pairs $(\Omega,K)$ where $\Omega\subset \RP$ is a properly convex domain and $K\subset \Omega$ is a compact subset. Then, there exists a map $\com \colon \Kc_d \rightarrow \RP$
\begin{align*}
    (\Omega,K) \longmapsto \com_\Omega(K)
\end{align*}
that satisfies: for every $(\Omega,K)\in \Kc_d$, 
\begin{enumerate}
    \item $\com_\Omega(K)\in\CH_\Omega(K)$,
    \item $\com_\Omega(K)=\com_\Omega\big(\CH_\Omega(K)\big)$, and
    \item $\gamma\com_\Omega(K)=\com_{\gamma\Omega}(\gamma K)$ for any $\gamma\in\PGL_d(\mathbb{R})$.
\end{enumerate}
There are various constructions of this map, see \cite[Proposition II.31]{IZ21} or \cite[Lemma 4.2]{Marquis}.

\subsection{Maximum principle} Recall our \hyperlink{link:notn_conv}{Notational Convention}. If  $r > 0$ and $C$ is a convex subset of a properly convex domain $\Omega$, define $\ngh{C}{r}:=\{x \in \Omega : \hil(x,C)<r\}$.  Moreover $\ngh{C}{0}=C$.

Although the distance between two projective segments  in the Hilbert metric is not convex, we have a `maximum principle'  which can be used to prove the convexity of $\Nc_r(C)$.

\begin{lemma}[{\cite[Corollary 1.9]{CLT}}]\label{lem: maximum principle}
    Suppose $\Omega$ is a \pcd ~and  $C\subseteq\Omega$ is a non-empty convex subset. Then the function $\Omega \ni x \mapsto \hil(x,C) \in \Rb$ satisfies the maximum principle, i.e., for any compact set $K\subseteq\Omega$, the restriction of this function to $K$ attains its maximum at an extreme point of $K$.
\end{lemma}

\begin{corollary}\label{cor: convexity of neighborhoods}
    Let $F\subseteq\Omega$ be a non-empty convex subset and $\delta\geq 0$. Then, both $\ngh{F}{\delta}$ and $(\overline{\Nc_{\delta}(F)}\cap \Omega)$ are convex subsets of $\Omega$.
\end{corollary}
\begin{proof}
    It suffices to show that if $a,b\in\partial\ngh{F}{\delta}$ then $[a,b]\subseteq\overline{\ngh{F}{\delta}} \cap \Omega$. By definition, there exist $c,d\in F$ such that $\delta=\hil(a,c)=\hil(b,d)$. From Lemma \ref{lem: maximum principle} we have that any point of $[a,b]$ is within distance $\delta$ from $[c,d]$. Since $F$ is convex, $[c,d]\subseteq F$ and thus $[a,b]\subseteq\overline{\ngh{F}{\delta}} \cap \Omega$.
\end{proof}

\subsection{Hausdorff topology} 
We recall the classical definition of Hausdorff distance; see for example \cite[Definition 5.30]{BridsonHaefliger}. 
\begin{definition}
\label{defn:haus_dist}
    Let $(X, D)$ be a metric space and let $\mathcal{C}(X)$ be the set of all non-empty closed subsets of $X$.
  The Hausdorff distance on $\mathcal{C}(X)$ (induced by $D$) is defined by: for any $A,B \in \Cc(X)$, 
\begin{align*}
    d^{\Haus}_D(A,B):=\max \left\{ \sup_{a \in A }d(a,B), \sup_{b \in B} d(b,A) \right\}.
\end{align*}
If $(X,D)$ is compact, then $(\Cc(X),d^H_D)$ is a compact metric space.  
\end{definition}

Let $d_{\Pb}$ denote a distance on $\RP$ induced by a(ny) Riemannian metric on $\RP$. We equip $\Cc(\RP)$ with the metric topology induced by $d^H_{d_{\Pb}}$ and  $(\Cc(\RP),d^H_{d_{\Pb}})$ is compact.  
Given a properly convex domain $\Omega$, let 
\begin{equation}
\label{eqn:defn_of_C_Omega}
C_{\overline{\Omega}}=\{X \subset \overline{\Omega} \mid X \text{ is closed in } \Omega, X \neq \emptyset\}.
\end{equation}
Recall the notion of being closed in $\overline{\Omega}$ from \cref{defn:closure_of_X_in_clOm}. By \cref{obs:closure_of_subset_equivalent}, $C_{\overline{\Omega}} \subset \Cc(\RP)$. We equip $C_{\overline{\Omega}}$ with the subspace topology from $\Cc(\RP)$.

\begin{corollary}[{Corollary of \cref{fact: lower semicontinuity of extended distance}}]
\label{cor:tracking_sets_using_points}
    Suppose $\seq{C_n}$ is a sequence of closed  subsets of $\overline{\Omega}$ and $\seq{x_n}$ is a sequence in $\Omega$ such that:
    \begin{enumerate}
        \item $C_n$ converges to a closed  set $C \subset \overline{\Omega}$ and $x_n \to x \in \overline{\Omega}$,
        \item there is a constant $R>0$ such that $\left( \sup_{n\in \Nb }\hil(x_n,C_n) \right) \leq  R$. 
    \end{enumerate}
    Then $d_{F_{\Omega}(x)}(x,C) \leq R$ and, in particular, $F_{\Omega}(x) \cap C \neq \emptyset$.
\end{corollary}
\begin{proof}
    For each $x_n$, pick $c_n \in C_n \cap \Omega$ such that $\hil(x_n,c_n)=\hil(x_n,C_n)< R$. Pass to a subsequence and assume that $c_n\to c \in \overline{\Omega}$. As $C_n\to C$, $c \in C$. Then \cref{fact: lower semicontinuity of extended distance} implies that $c \in F_{\Omega}(x)$ and $d_{F_{\Omega}(x)}(x,c)<R$.
\end{proof}

\subsection{Local Hausdorff topology induced by the Hilbert metric} Fix a properly convex domain $\Omega \subset \RP$. For a point $x \in \Omega$ and a radius $r > 0$, we write $B(x, r) = \{y \in \Omega \mid \hil(x, y) < r\}$. We will define a certain subset of $C_{\overline{\Omega}}$ (see \cref{eqn:defn_of_C_Omega}). Recall the notion of being closed in $\Omega$ from \cref{defn:closure_of_X_in_Om}. Now define 
\begin{align*}
    C^{o}_{\overline{\Omega}}&:=\{ Y \subset \Omega \mid Y \text{ is closed in } \Omega, Y\neq \emptyset \}\\
    &= \{ X \cap \Omega : X \in C_{\overline{\Omega}} \}\setminus \{\emptyset\}.
\end{align*}

\begin{definition}\label{def:local_hausdorff}
For $C_0 \in C^{o}_{\overline{\Omega}}$, $x_0 \in \Omega$, and $r_0, \epsilon_0 > 0$, let
$$U(C_0, x_0, r_0, \epsilon_0)\coloneqq \{C\in C^{o}_{\overline{\Omega}} \mid d_{\Omega}^{\Haus} \left( C_0 \cap B(x_0, r_0), C \cap B(x_0, r_0) \right) < \epsilon_0\}.$$
The \emph{local Hausdorff convergence topology} on $C^{o}_{\overline{\Omega}}$ induced by $\hil$ is the topology generated by the family of sets $\{U(C_0, x_0, r_0, \epsilon_0)\mid C_0\in C^{o}_{\overline{\Omega}},\ x_0\in \Omega, \ r_0,\ \epsilon_0 > 0\}$.
\end{definition}
For the sake of brevity, we will call this the \emph{local Hausdorff topology} (on $C^{o}_{\overline{\Omega}}$). As the name suggests, one  can think of this topology as a localization of the Hausdorff topology.  This has been used widely in CAT(0) geometry \cite{HK05}, in convex projective geometry \cite{IZRelaHyp, IZ24OneDimFaces}, and also in several complex variables \cite{Fra89}.

Since $(\Omega,\hil)$ is proper, \cite[Section 2]{IZ24OneDimFaces} implies that the local Hausdorff topology is second countable and we have the following: 

\begin{lemma}
\label{lem: characterization lht convergence}
A sequence  \Seq{C_n} in $C^{o}_{\overline{\Omega}}$ converges to $C \in C^{o}_{\overline{\Omega}}$ as $n\to \infty$ in the local Hausdorff topology if and only if
$$
\lim\limits_{n\to\infty} d^{\Haus}_{\Omega}(C_n\cap B(x_0,r), C\cap B(x_0,r))=0
$$
for all $x_0\in \Omega$ and $r>0$ such that $C\cap B(x_0,r)\neq \emptyset$.
\end{lemma}

We have the following convergence results for convex hulls of finite sets.

\begin{lemma}\label{lem: convergence of simplices}
    Fix $k\in\mathbb{N}$. For each $i\in \{0,\dots,k\}$, let $\seq{v^i_n}$ be a sequence in $\overline{\Omega}$ such that $v^i_n\xrightarrow[n\to\infty]{} v^i\in\overline{\Omega}$. Set $\wh{S}_n\coloneqq \CH_\Omega(v^0_n,\dots,v^k_n)$  and $\wh{S}\coloneqq \CH_\Omega(v^0,\dots,v^k).$   Assume that $(\wh{S}_n\cap \Omega)\neq \emptyset$ for any  $n\in\mathbb{N}$ and $(\wh{S}\cap \Omega)\neq \emptyset$. Then, $$\left( \wh{S}_n\cap \Omega \right)  \xrightarrow[n\to\infty]{} \left( \wh{S}\cap \Omega   \right) \text{ in the \lht.}$$
\end{lemma}
\begin{proof}
    Set $S_n:=\wh{S}_n\cap \Omega$ for all $n\in \Nb$ and $S:=\wh{S}\cap \Omega$.\\
    Fix $x_0\in \Omega$ and $r>0$ such that $S\cap B(x_0,r)\neq \emptyset$. Set $B_r:=B(x_0,r)$. By \cref{lem: characterization lht convergence}, it suffices to show that 
    \begin{equation}\label{eqn: convergence simplices lht 0}
        \lim_{n\to\infty}{\hil}^{\Haus}(S_n\cap B_r, S\cap B_r)=0.
    \end{equation}
    For each $n\in\mathbb{N}$, let $p_n\in \overline{S\cap B_r}$ such that
    $
    \hil\left(p_n,S_n\cap B_r\right)=\sup\limits_{p\in S\cap B_r} \hil(p,S_n\cap B_r).
    $
    As $S\subseteq\CH_\Omega(v^0,\dots,v^k)$,
    then for each $n$ there exist $a_n^0,\dots,a_n^k\in[0,1]$ such that $\sum_{i=0}^ka_n^i=1$ and $p_n=\sum_{i=0}^ka_n^iv^i$.
    Now consider the points $q_n=\sum_{i=0}^ka_n^iv_n^i$ that lie in $\overline{S_n}$. Since $v^i_n$ converges to  $v^i$ for all $i$, then $\hil\left(p_n,q_n\right)\to0$. Moreover $\lim_{n\to\infty}\hil(q_n,B_r)=0$, as   $\hil(q_n,B_r) \leq \hil(q_n,p_n)$. Thus $\hil(q_n,\overline{S_n \cap B_r})\to 0$.
    Then, 
    \begin{align}
    \label{eqn:simplex_convergence_1}
       0 &\leq \lim_{n\to\infty}\left(\sup\limits_{p\in S\cap B_r} \hil(p,S_n\cap B_r)\right)= \lim_{n\to\infty}\hil\left(p_n,S_n\cap B_r\right)\\
        &\le\lim_{n\to\infty}\hil\left(p_n,q_n\right)+\hil\left(q_n,S_n\cap B_r\right)=0. \nonumber
    \end{align}

    On the other hand, for each $n\in\mathbb{N}$, consider a point $p'_n\in\overline{S_n\cap {B_r}}$ such that $\hil(p_n',S\cap B_r)=\sup \limits_{p'\in S_n\cap B_r} \hil(p',S\cap B_r)$.    
    Suppose there is a subsequence $\seq{p_{k_n}'}$ such that $\lim_{n\to\infty}\hil(p_{k_n}',S\cap B_r)>0.$
    For each $n$, we write $p'_n = \sum_{i=0}^k a^i_n v^i_n$ where  $a_n^i$ lies in $[0,1]$ and $\sum_{i=0}^ka_n^i=1$.  Then there is a subsequence $\seq{\ell_n}$ of $\seq{k_n}$ such that  $\lim_{n\to\infty}a^i_{\ell_n}=a^i\in[0,1]$ for each $i$.  Hence $\seq{p'_{\ell_n}}$ converges to $p_{\infty}' \coloneqq \sum_{i=0}^k a^i v^i.$ Note that $p_{\infty}' \in \overline{S}$ by convexity and  $p_{\infty}' \in \overline{B_r}$ as each $p_n\in B_r$.
    Therefore 
    \begin{align*}
        0< \lim_{n\to\infty}\hil(p_{k_n}',S\cap B_r)=\lim_{n\to\infty} \hil(p'_{\ell_n},S\cap B_r) \leq \hil(p'_{\ell_n},p_{\infty}')=0,
    \end{align*}
    a contradiction. This implies that 
    \begin{align}
        \label{eqn:simplex_convergence_2}
        \lim_{n\to\infty}\left(\sup\limits_{p'\in S_{n}\cap B_r} \hil(p',S\cap B_r) \right) =\lim_{n\to\infty} \hil(p'_n,S\cap B_r)=0.
    \end{align}
    Then \cref{eqn:simplex_convergence_1} and \cref{eqn:simplex_convergence_2} imply \cref{eqn: convergence simplices lht 0}. This finishes the proof.
\end{proof}

\begin{corollary} 
\label{cor: compactness of simplices in loc hausdorff top}
Fix $k \in \Nb$. Suppose $K \subset \Omega$ is compact, and  $\wh{S}_n\coloneqq \CH_\Omega(v^0_n,\dots,v^k_n)$ is such that $\wh{S}_n \cap K \neq \emptyset$  for all $n \in \Nb$. Then there is a subsequence $\seq{\wh{S}_{k_n}}$ and a closed convex set $\wh{S}$ of $\overline{\Omega}$ with  $\wh{S}\cap \Omega \neq \emptyset$ such that  $$(\wh{S}_{k_n}\cap \Omega)\xrightarrow[n \to \infty]{}  (\wh{S}\cap \Omega) \text{ in the local Hausdorff topology.}$$ 
In particular, $\wh{S}=\CH_{\Omega}(v^0,\dots,v^k)\cap \Omega$ where $v^i:=\lim_{n\to\infty}v^i_{k_n}$.
\end{corollary}
\begin{proof}
    First we pass to a subsequence $\seq{k_n}$ so that $\lim_{n\to\infty}v^i_{k_n}=v^i$ exists in $\overline{\Omega}$ for all $i =0,\dots,k$. Consider $\wh{S}=\CH_{\Omega}(v^0,\dots,v^k)$. Note that it suffices to prove $\wh{S} \cap \Omega \neq \emptyset$, because then the result is immediate from \cref{lem: convergence of simplices}.  To prove that $\wh{S}\cap \Omega \neq \emptyset$, pick $u_{k_n} \in \wh{S}_{k_n} \cap K$ for each $n.$ 
    Up to passing to another subsequence, we can assume that $\seq{u_{k_n}}$ converges to some $u\in\overline{\Omega}$. As $K\subset \Omega$, $u \in K \subset \Omega$. So we only need to show that $u \in \wh{S}$. This can be done by similar arguments as in the proof of \cref{lem: convergence of simplices}: write $u_{k_n}=\sum_{i=0}^k a^i_nv^i_n$ with $a^i_n\in [0,1]$ and $\sum_{i=0}^ka^i_n=1$; then note that $u_{k_n} \to u$ provided $a^i_n\to a^i$ for each $i=0,\dots,k$; and then conclude that $u=\sum_{i=0}^ka^iv^i \in \wh{S}$.  
\end{proof}

\subsection{Projective Simplex Domain}

\begin{definition}
\label{defn:projective_simplex_domain}
We say that a properly convex domain $\Omega \subset \RP$ is a \emph{projective simplex domain in $\RP$} if $\Omega=g \cdot \{ [x_1:\dots:x_d] | x_1>0,\dots,x_d>0 \}$ for some $g \in \PGL_d(\Rb)$.
\end{definition}
In the affine chart $\RP-g \cdot \{[x_1:\cdots:x_d]| x_1+\dots+x_d=0\}$, $\Omega$ projects to an open $(d-1)$-dimensional Euclidean simplex. Hence the name `projective simplex' for such domains. Note the difference between this definition (where the properly convex domain itself is a simplex) and the notion of a $k$-simplex in some arbitrary properly convex domain $\Omega$ (see \cref{sec:defn_of_simplices}).

\subsection{Symmetric  domains}
\label{sec:symmetric_domains}

\subsubsection{Projective model for the symmetric space of $\SL_n(\mathbb{R})$.}

Let $M_n(\Rb)$ denote the set of $n\times n$ real matrices and let $\Pi_n(\mathbb{R})$ denote its subset consisting of positive definite matrices. Then $\Pi_n(\mathbb{R})$ can be identified with an open convex cone in the vector space $\mathbb{R}^{\frac{n(n+1)}{2}}$. The projectivization  $\Pb(\Pi_n(\Rb))$ of this cone is a properly convex domain in $\Pb(\Rb^{\frac{n(n+1)}{2}})$. Indeed, the proper convexity follows as $\overline{\Pb(\Pi_n(\Rb))}$ is disjoint from the projective hyperplane $\Pb(\{ A \in M_n(\Rb): \tr(A)=-1\})$.

The properly convex domain $\Pb(\Pi_n(\Rb))$ is the projective model of the symmetric space $\SL_n(\mathbb{R})/\SO(n)$. To see this, consider the action of $\SL_n(\mathbb{R})$ on $\Pb(\Pi_n(\mathbb{R}))$ via $$ g \cdot [S] \mapsto [gSg^T].$$  This action is transitive as symmetric matrices are orthogonally diagonalizable. Finally, the stabilizer of $[I_n]$ (the projective class of the identity matrix) in $\SL_n(\Rb)$ is precisely the maximal compact subgroup $\SO(n)$. Hence the conclusion.

The $\SL_n(\Rb)$-invariant Hilbert distance between  $[A],[B]\in \Pb(\Pi_n(\mathbb{R}))$ can be expressed in terms of the eigenvalues of $A^{-1}B$: if $\lambda_1 \geq \dots \geq \lambda_n >0$ are the eigenvalues of the matrix $A^{-1}B$, then the Hilbert distance is 
\begin{equation*}
d_{\Pb(\Pi_n(\mathbb{R}))}([A], [B]) = \frac{1}{2} \log \left( \frac{\lambda_1}{\lambda_n} \right).
\end{equation*}

\begin{table}[h]
\centering
\renewcommand{\arraystretch}{1.8} 
\begin{tabular}{|p{8.5cm}|c|c|} 
\hline
\textbf{Projective Domain} $\Omega$ & $\mathbf{d}=\dim(\Omega)$ & \textbf{$H$} \\
\hline

{\small $\mathbb{H}^{n-1}_{\mathbb{R}}=\mathbb{P}\left(\{x\in \mathbb{R}^n \mid x_1^2 - x_2^2 - \dots - x_n^2 > 0\}\right)$; $n \geq 3$}
& {\small$ n - 1$ }
& {\small$\PO(1,n-1)$} \\
\hline

{\small $\mathbb{P}(\Pi_n(\mathbb{R}))=\mathbb{P}(\{\text{Pos. def. sym. } n \times n \text{ real matrices}\})$; $n \geq 3$} 
& {\small $ \dfrac{n(n+1)}{2} - 1$ }
& {\small $\PSL_{n}(\mathbb{R})$} \\
\hline

{\small $\mathbb{P}(\Pi_n(\mathbb{C}))=\mathbb{P}(\{\text{Pos. def. herm. } n \times n \text{ complex mat.}\})$}
& {\small$ n^2 - 1$ }
& {\small$\PSL_{n}(\mathbb{C})$ }\\
\hline

{\small $\mathbb{P}(\Pi_n(\mathbb{H}))=\mathbb{P}(\{\text{Pos. def. herm. } n \times n \text{ quaternion mat.}\})$} 
& {\small$ 2n^2 - n - 1$}
& {\small$\PSL_{n}(\mathbb{H})$} \\
\hline

{\small $\mathbb{P}(\Pi_3(\mathbb{O}))$}
& {\small$ 26$ }
& {\small$E_{6(-26)}$} \\
\hline
\end{tabular}
\caption{Irreducible properly convex symmetric domains.}
\label{table:projective_symmetric_domains}
\end{table}

\subsubsection{Classification of irreducible symmetric domains}
 A properly convex domain $\Omega$ is called \emph{symmetric} if $\Aut(\Omega)$ is a reductive Lie group that acts transitively on $\Omega$ \cite{Benoist_survey_2008}. A properly convex domain is said to be \emph{irreducible} if it cannot be decomposed as the projective join of two properly convex domains of strictly smaller dimension. Consequently, any properly convex domain decomposes uniquely into a join of irreducible factors.
So the classification of all symmetric domains reduces to the classification of irreducible symmetric domains.

Koecher \cite{Koecher} and Vinberg \cite{VinbergHomogeneous} completely classified the \emph{irreducible symmetric domains}. 
\cref{table:projective_symmetric_domains} provides the complete list of all irreducible symmetric domains $\Omega$ and the simple Lie groups $H$ locally isomorphic to the corresponding $\Aut(\Omega)$.
For each $H$, its real rank  -- denoted by $\rk(H)$ -- is the dimension of its maximal abelian $\Rb$-diagonalizable subgroup. We call a symmetric domain $\Omega$ \emph{higher rank} if $\rk(H) \geq 2$ for its corresponding $H$. Indeed, in \cref{table:projective_symmetric_domains}, the only rank one domains (i.e. $\rk(H)=1$) are $\Hb^{n-1}_{\Rb}$ while all the other domains are higher rank.  Any symmetric domain $\Omega$ is quasi-homogeneous because $\Aut(\Omega)^0$ (the identity component of $\Aut(\Omega)$) acts transitively on $\Omega$.

\begin{remark}
\label{rem:irred_sym_sp_and_domain}
    Recall from \cref{sec:intro_appl_to_symm_space} the notion of an irreducible symmetric space $G/K$ admitting a convex projective model $\Omega$. We note that such $\Omega$ exactly corresponds to the irreducible symmetric domains, i.e. one of the domains listed in \cref{table:projective_symmetric_domains}. Indeed, a domain $\Omega$ satisfies the definition in \cref{sec:intro_appl_to_symm_space} if and only if $\Aut(\Omega)^0$ is a connected Lie group with trivial center acting transitively on $\Omega$. Hence the conclusion is obvious. 
\end{remark}

 Each irreducible symmetric domain is, in fact, divisible. Indeed, the work of Borel \cite{Borel} implies that each in case of \cref{table:projective_symmetric_domains}, $\Aut(\Omega)$ contains a discrete subgroup $\Gamma$ such that $\Omega/\Gamma$ is compact.

\begin{lemma}
\label{lem:riem_QI_hilbert}
    Suppose $\Omega$ is a properly convex domain as in \cref{table:projective_symmetric_domains} so that $\Aut(\Omega)$ is locally isomorphic to the corresponding $H$ in the table. Let $d$ be any $\Aut(\Omega)^0$-invariant, proper distance on $\Omega$. Then $(\Omega,d)$ is quasi-isometric to $(\Omega,\hil)$, where $\hil$ is the Hilbert metric on $\Omega$. 
\end{lemma}
\begin{proof}
    By the fact stated right before the lemma, we can find a discrete subgroup $\Gamma<\Aut(\Omega)^0$ such that $\Gamma$ acts co-compactly on $\Omega$. By the hypotheses on $d$, $\Gamma$ also acts properly and isometrically on $(\Omega,d)$. Then, by the \v{S}varc-Milnor lemma, $\Gamma$  is quasi-isometric to both $(\Omega,d)$ and $(\Omega,\hil)$. Hence the result.
\end{proof}

\subsection{Rank of a symmetric domain}
\label{sec:rank_of_symm_domain}

\subsubsection{Asymptotic rank} Given a proper metric space $(X,d)$, we define its \emph{asymptotic rank} (denoted by $\asrk(X,d)$) as the supremum of all integers $k\ge0$ for which the following holds: there exist 
\begin{enumerate}
    \item a unit ball $B \subset \Rb^k$ \wrt\ some norm on $\mathbb{R}^k$,
    \item a sequence of subsets $(X_n)_{n\in\mathbb{N}}$ of $X$, and
    \item a diverging sequence $(r_n)_{n\in\mathbb{N}}\subseteq\mathbb{R}_{>0}$
\end{enumerate}
such that $(X_n,r_n^{-1}d)_{n\in\mathbb{N}}$ converges to $B$ \wrt\ the Gromov-Hausdorff topology.

\begin{remark}
\label{rem:normed_space_embedding_gives_lower_bound_on_asrk}
    If $(V,\norm{\cdot})$ is a normed vector space such that $V \hookrightarrow X$ is an isometric embedding into the metric space $(X,d)$, then $\asrk(X,d) \geq \dim V$.
\end{remark}

\begin{fact}[{\cite[Section 3]{Wenger2011}}]
\label{fact:asrk_QI_inv}
    If $X$ and $X'$ are quasi-isometric metric spaces, then $\asrk(X)=\asrk(X')$.
\end{fact}

\begin{proposition}
\label{lem:asrk_flat_rank_easy_direction}
    Suppose $\Omega$ is a properly convex domain. The asymptotic rank of $(\Omega,\hil)$ is at least $\rF(\Omega)$, that is $\asrk(\Omega,\hil)\geq\rF(\Omega)$.
\end{proposition}
\begin{proof}
    Suppose $S \subset \Omega$ is a PES. By convexity of $S$,  $d_S=\hil{\big|_{S \times S}}$. Recall that $(S,d_S)$ is isometric with $\Rb^{\dim S}$ equipped with a polyhedral norm. Thus, by \cref{rem:normed_space_embedding_gives_lower_bound_on_asrk}, $\asrk(\Omega,\hil)\geq\dim(S)$. Hence the result.  
\end{proof}

\subsubsection{Euclidean rank}

For a non-positively curved Riemannian manifold $Y$  (more generally, a complete CAT(0) space), there is a related notion of Euclidean rank. The Euclidean rank of $Y$, denoted by $\rEuc(Y)$,  is the maximal $n$ such that there is an isometric embedding of the Euclidean space $\Rb^n$ in $Y$. 

For Riemannian symmetric spaces, this notion of Euclidean rank has several characterizations. It is well known that for a Riemannian symmetric space of non-compact type $X$, $\rEuc(X)$ coincides with the real rank of the semi-simple Lie group $G=\Isom(X)^0$, the identity component of the isometry group of $X$. It also equals the asymptotic rank, as shown by the fact below.

\begin{fact}[{\cite[Theorem 3.4]{Wenger2011}}]
\label{fact:asrk_equals_reuc}
    If $(X,d)$ is a complete CAT(0) space that is proper and co-compact, then $\asrk(X,d)=\rEuc(X,d)$. In particular, the equality holds when $X$ is a Riemannian symmetric space of non-compact type.
\end{fact}

\subsubsection{Projective simplex rank}

Recall from \cref{sec:intro_appl_to_symm_space} the notion of an irreducible symmetric space $G/K$ admitting a convex projective model $\Omega$. Moreover,  such $\Omega$ are precisely the irreducible symmetric domains; see \cref{rem:irred_sym_sp_and_domain}. We will show that in this case, all notions of rank coincide with the PES-rank, by establishing the following inequalities:

\begin{equation*}
    \boxed{\rF(\Omega)
    \stackrel{\normalfont\mbox{(\ref{lem:asrk_flat_rank_easy_direction})}}{\le}
    \asrk(\Omega,\hil)
    \stackrel{\normalfont\mbox{(\ref{fact:asrk_equals_reuc})}}{=}
    \rEuc(G/K,d)
    \ =\
    \rk(G)
    \stackrel{\normalfont\mbox{(\ref{lem: flat rank for symmetric spaces})}}{\le}
    \rF(\Omega)}
\end{equation*}

\begin{lemma}
    \label{lem: flat rank for symmetric spaces}
    Suppose $G/K$ is an irreducible symmetric space with a convex projective model $\Omega$.  Then $\rF(\Omega)=\rk(G)$.
\end{lemma}
\begin{proof}
    Let $d$ be a $G$-invariant distance on $G/K$ induced by a non-positively curved Riemannian metric. By the hypothesis, there is an isomorphism $f:G \to \Aut(\Omega)^0$ and a diffeomorphism $\phi:G/K \to \Omega$ such that $\phi(g ~ \cdot )=f(g) \phi(\cdot)$. Define an $\Aut(\Omega)^0$-invariant distance $d_{\phi}$ on $\Omega$ given by $d_{\phi}(x,y)=d(\phi^{-1}(x),\phi^{-1}(y))$ for all $x,y \in \Omega$. Then $(G/K,d)$ is isometric to $(\Omega,d_{\phi})$ and, by \cref{lem:riem_QI_hilbert}, $(\Omega,d_{\phi})$ is quasi-isometric to $(\Omega,\hil)$. 

    Note that by virtue of \cref{fact:asrk_QI_inv}, $\asrk(G/K,d)=\asrk(\Omega,\hil)$. By \cref{fact:asrk_equals_reuc} and \cref{lem:asrk_flat_rank_easy_direction}, $\rF(\Omega) \leq \rEuc(G/K,d)$.  As we remarked, it is well-known that $\rEuc(G/K,d)=\rk(G)$. Thus, 
    $$\rF(\Omega) \leq \rk(G).$$

    Now it suffices to prove the reverse inequality. For that, first fix a uniform lattice $\Gamma$ in $G.$ Set $r:=\rk(G)$. Then, by \cite[Corollary 2.9]{PR1972}, $\Gamma$ contains a subgroup $A_0$ isomorphic to $\Zb^{r}$. Let $A$ be a maximal abelian subgroup of $\Gamma$ containing $A_0$. Set $A':=f(A)$

    By the flat torus theorem in convex projective geometry \cite{IZ21}, $A'$ preserves a PES $S'$ in $\Omega$, where $\dim(S')$ equals the free rank of the abelian group $A'$. Since $A'$ contains $f(A_0)$, we have that $\dim(S') \geq r$. Thus, $\rF(\Omega) \geq r=\rk(G)$.
    \end{proof}

\section{Preliminaries on simplices in properly convex domains}

\label{sec:defn_of_simplices}
\label{sec:definition_of_simplices}

\subsection{Definitions}

Suppose $\Omega$ is a properly convex domain. For us, a $k$-simplex will be the convex hull of $(k+1)$-points in $\overline{\Omega}$. The formal definition is as follows.
\begin{definition}
\label{defn:k_simplices_again} 

Let $k\geq 1$. A \emph{$k$-simplex} in $\overline{\Omega}$ is the convex hull in $\overline{\Omega}$ of $(k+1)$ points $v_0,\dots,v_k \in \overline{\Omega}$. The ordered $(k+1)$-tuple $(v_0,\dots,v_k)$ is the \emph{tuple of vertices of $S$}
(denoted by $\vrtx(S)$) and each $v_i$, $0\leq i\leq k$, is a  \emph{vertex of $S$}.\\
We call such a $k$-simplex 
 \begin{enumerate}
     \item \emph{non-degenerate}, if $\dim(\Pb(\Span\{v_0,\dots,v_k\}))=k,$ and 
     \item \emph{degenerate}, if $\dim(\Pb(\Span\{v_0,\dots,v_k\}))<k.$ 
 \end{enumerate}
\end{definition}

\begin{remark}
Often in convex projective geometry literature, when authors write `$k$-simplex', they are referring to a non-degenerate $k$-simplex in our sense. However, we will not do so in this paper. Our k-simplices may be degenerate. See \cref{sec:examples_simplices} below for examples.
\end{remark}

\begin{figure}[h]
    \centering
    \includegraphics[width=1\linewidth]{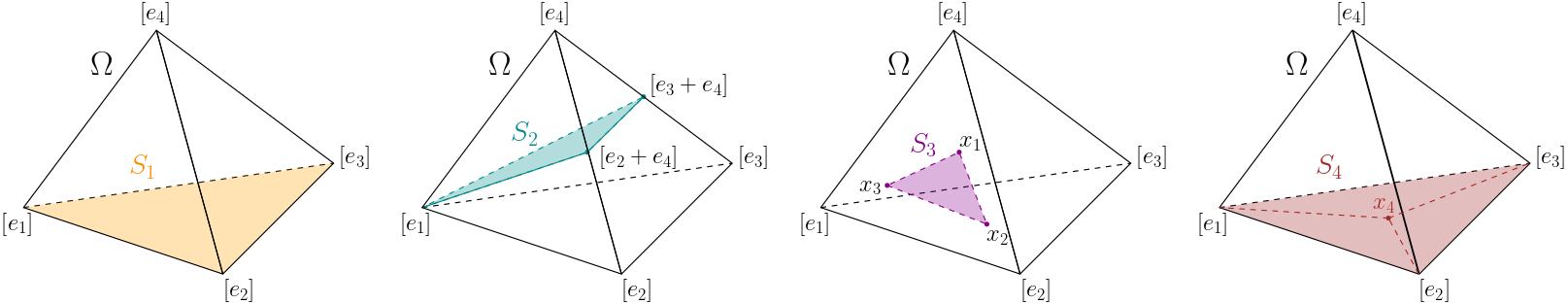}
    \caption{Different kinds of simplices in a projective simplex domain $\Omega$: (a) 2-simplex in $\overline{\Omega}$, but not in $\Omega$; (b) PES; (c) compact 2-simplex in $\Omega$; (d) degenerate 3-simplex in $\overline{\Omega}$. See Example  \ref{sec:examples_simplices} for details.}
    \label{fig:simplex_defn_example_1}
\end{figure}

Note that a $k$-simplex in $\overline{\Omega}$ is closed in $\overline{\Omega}$. However, a $k$-simplex of $\overline{\Omega}$ may very well be contained entirely in $\partial\Omega$ (see \cref{fig:simplex_defn_example_1}(a)). So we now make a definition that identifies the $k$-simplices that intersect $\Omega$ non-trivially.

\begin{definition}
    Let $S\coloneqq \CH_\Omega(v_0,\dots,v_k)$ be a $k$-simplex in $\overline{\Omega}$. We will call $S$ a \emph{$k$-simplex in $\Omega$} provided $S\cap \Omega \neq \emptyset$, or equivalently, $\relint(S) \subset \Omega$. 
\end{definition}
 An immediate application of \cref{lem:relint_dichotomy} is the following observation.
\begin{observation}
    Suppose $S$ is a $k$-simplex in $\overline{\Omega}$. Then, either $S$ is a $k$-simplex in $\Omega$, or $S\subset \partial \Omega$.  
\end{observation}

Even if $\relint(S) \subset \Omega$, $\partial S$ may still intersect $\partial \Omega$ non-trivially; see \cref{fig:simplex_defn_example_1} for examples. So we now idenitfy a special family consisting of $k$-simplices  that lie entirely in $\Omega$. Quite unimaginatively, we call them compact $k$-simplices in $\Omega$, as they are compact in the subspace topology on $\Omega$.

\begin{definition}
\label{defn:closed_simplex_in_Omega}
    We will say that $S$ is a \emph{compact  $k$-simplex in $\Omega$} if $S= S\cap\Omega$. 
    Equivalently, a $k$-simplex $S$  in  $\Omega$ is compact if and only if each vertex of $S$ lies in $\Omega$, i.e. $\vrtx(S) \in S^{k+1} \subset \Omega^{k+1}$. 
\end{definition}

\subsection{Example}(See \cref{fig:simplex_defn_example_1})
\label{sec:examples_simplices}
We consider the properly convex domain $\Omega:=\{ [x_0:x_1:x_2:x_3] | x_0,x_1,x_2,x_3>0 \}$. It is a 3-dimensional projective simplex domain.

We set $$[e_1]:=[1:0:0:0] \in \partial \Omega,\dots, [e_4]:=[0:0:0:1] \in \partial \Omega.$$ 
\begin{enumerate}[label=(\alph*)]
    \item $S_1:=\CH_{\Omega}([e_1],[e_2],[e_3])$ is a 2-simplex in $\overline{\Omega}$, but not a 2-simplex in $\Omega$ (since $S_1 \cap \Omega=\emptyset$).
    \item $S_2:=\CH_{\Omega}([e_1],[e_2+e_4],[e_3+e_4]).$ This is an example of a 2-simplex in $\Omega$. In particular, this is an example of a 2-dimensional PES in $\Omega$ -- a notion that we will discuss below in \cref{defn:PES}. However, $S_2$ isn't a compact 2-simplex in $\Omega$ as the vertices of $S_2$ lie in $\partial \Omega$. But it is worth noting that $S_2\cap \Omega$ is a closed subset of $\Omega$.
    \item Let $S_3:=\CH_{\Omega}(x_1,x_2,x_3)$ where $x_1=[1:1:1:1],x_2=[1:2:1:1],x_3=[2:1:1:1]$. Then $S_3$ is a non-degenerate compact $2$-simplex in $\Omega$. 
    \item All three simplices $S_1, S_2,S_3$ above are non-degenerate 2-simplices. On the other hand, $S':=\CH_{\Omega}([e_1],[e_2],[e_3],[e_1+e_2+e_3])$ is a degenerate 3-simplex in $\overline{\Omega}$. 
\end{enumerate}

\subsection{Faces of simplices}

We now introduce the notation of faces for a $k$-simplex $S$ with vertices $v_0,\dots,v_k$. This generalizes the face notation for non-degenerate simplices.  For any $I \subseteq \set{0,\dots,k}$, we call 
\begin{align*}
    F^I(S)\coloneqq \CH_\Omega(v_i\mid i\not\in I)
\end{align*}
the \emph{face opposite to the vertices indexed by $I$}. 

\begin{definition}
    A \emph{facet} of a $k$-simplex $S$ is any face of $S$ that is opposite to a single vertex. That is, the facets of $S$ are $F^{\{i\}}(S)$ for $i=0,\dots,k$ and are often denoted by $F^i(S)$. We set $\Fc(S)\coloneqq \{F^i(S) :i=0,\dots,k\}$ to be the set of all facets of $S$.
\end{definition}
\begin{observation}
\label{obs:face_of_face_in_simplex}
    Suppose $S$ is a $k$-simplex in $\overline{\Omega}$. If $0\leq i\neq j\leq k$, then $F^j(F^i(S))=F^j(S) \cap F^i(S)=F^{\{i,j\}}(S)$.
\end{observation}

For a non-degenerate simplex $S$, each facet is a codimension-1 face of $S$. The following elementary result about the faces of a non-degenerate simplex can be proven by drawing one in an affine chart. 
\begin{observation}
\label{obs:about_non_deg_simplices}
    Suppose $S$ is a non-degenerate $k$-simplex in $\overline{\Omega}$ with $k\geq 2$ and $i \neq j \in \{0,\dots,k\}$. Then
    \begin{enumerate}
         \item $F^i(S) \cap F^j(S) =F^{\{i,j\}}(S) \subset \partial F^i(S) \cap \partial F^j(S)$.
         \item $\relint(F^i(S)) \cap F^j(S) =\emptyset$.
         \item {$\partial S =\cup_{i=0}^k F^i(S)$.}
    \end{enumerate}
\end{observation}

\subsection{Properly Embedded Non-degenerate Simplex}
\label{sec:PES}

In the definition below, PES is an acronym for \emph{properly embedded non-degenerate simplex}. These are convex projective analogues of totally geodesic flats in CAT(0) geometry. 
\begin{definition}\label{defn:PES}
    Suppose $\Omega \subseteq \RP$ is a properly convex domain and $2 \leq k \leq (d-1)$. We say that a convex subset $S$ of $\Omega$ is a \emph{$k$-dimensional PES} in $\Omega$ if $S$ is a non-degenerate $k$-simplex in $\Omega$ and $\partial S=\cup_{i=0}^k F^i(S)\subset\partial\Omega$.
\end{definition}
As the reader may easily guess,  the name `properly embedded simplex' originates from the fact that the map $(\relint(S))\hookrightarrow \Omega$ is a proper embedding.


\begin{lemma}
\label{lemma:faces_of_PES}
    Let $S$ be a $k$-dimensional PES in $\Omega \subset \RP$ for some $ 1 \leq k \leq d-1$, and let $i\in\{0,\dots,k\}$. Then:
    \begin{enumerate}
        \item $F^i(S) \subset \partial \Omega$.
        \item $F^i(S)$ is a non-degenerate $(k-1)$-simplex in $\overline{\Omega}$.
        \item If $u_i \in \relint(F^i(S))$, then $\relint(F^i(S)) \subset F_{\Omega}(u_i)$ and $\partial F^i(S) \subset \partial F_{\Omega}(u_i)$. In particular, $F^i(S)$ is a $(k-1)$-dimensional PES in $F_{\Omega}(u_i)$.
        \item $F_{\Omega}(F^i(S)) \subset \bdry$.
    \end{enumerate}
\end{lemma}
\begin{proof} Let $\vrtx(S):=(v_0,\dots,v_k).$ Then (1) is immediate from the  definition of PES.

    \noindent (2) By definition of $F^i(S)$, it is a $(k-1)$ simplex in $\overline{\Omega}$. That it is non-degenerate is clear, since $S$ is a non-degenerate $k$-simplex and $S=\CH_{\Omega}(v_i,F^i(S)).$
    
    \noindent (3) First, we apply \cref{lem:top_char_bdry_points} to $\relint(F^i(S))$ which is a properly convex domain in $\Pb(\Span F^i(S))$. Then, if $x \in \relint(F^i(S))$, \cref{lem:top_char_bdry_points} implies that there is an open projective segment in $\relint(F^i(S))$ that contains both $x$ and $u_i$. Thus $\relint(F^i(S)) \subset F_{\Omega}(u_i)$ by definition of $F_{\Omega}(\cdot)$. 

    Since $S$ is non-degenerate, $\partial F^i(S)=\cup_{j\neq i} F^{\{i,j\}}(S)$ (\cref{obs:face_of_face_in_simplex}). So it suffices to show that $F^{\{i,j\}}(S) \subset \partial F_{\Omega}(u_i)$ for any $j\neq i$. 
    Let $x' \in F^{\{i,j\}}(S)$ for some $j\neq i$. Then $$[x',v_i] \subset F^j(S) \subset \partial S \subset \partial \Omega.$$ Moreover, $x' \in \overline{F_{\Omega}(u_i)}$ because of the previous paragraph. 
    Now suppose, if possible, that $x' \in F_{\Omega}(u_i).$ Note that $(u_i,v_i) \subset \relint(S) \subset \Omega$, since $u_i$ lies in the face opposite to $v_i$. Then, as $x'\in F_{\Omega}(u_i)$, \cref{fact: behaviour of faces of Omega}(3) implies that $(x',v_i) \subset \Omega$, a contradiction. Hence $x' \in \partial F_{\Omega}(u_i).$

    \noindent (4) This is immediate from (1) above.
\end{proof}

\subsection{Slimness of degenerate simplices}

     We say that \emph{a $k$-simplex in $\Omega$ is $0$-slim} provided $F^i(S) \subset \bigcup_{j=0, j\neq i}^k F^j(S)$ for all $i\in \{0,\dots,k\}$. This definition coincides with our definition of $\delta$-slim simplices with $\delta=0$ because of our \hyperlink{link:notn_conv}{convention} about $0$-neighborhoods. 

    We first prove that a $0$-slim simplex is necessarily degenerate.

\begin{lemma}
\label{lem:non_deg_simplices_not_0_slim}
    Suppose $S$ is a non-degenerate $k$-simplex in $\Omega$ with $k\geq 2$. Then $S$ cannot be $0$-slim.
\end{lemma}
\begin{proof}
    Suppose, if possible, that $S$ is $0$-slim. Then $F^0(S) \subset \bigcup_{j=1}^kF^j(S)$. But by \cref{obs:about_non_deg_simplices}, $\relint(F^0(S)) \cap F^j(S) =\emptyset$ for any $j\in\{1,\dots,k\}$. As $k\geq 2$, $\relint(F^0(S))$ is non-empty. Hence we have a contradiction.
\end{proof}

Now we will show that all degenerate simplices are in fact $0$-slim.

\begin{lemma}\label{lem: facets of degenerate simplices}
    Suppose $S$ is a degenerate $k$-simplex in $\overline{\Omega}$ for some $k \geq 2$. Then, for any $i\in\{0,\dots,k\}$, 
    \begin{align*}
        F^i(S)\subseteq \bigcup_{\substack{j=0\\j\neq i}}^k F^j(S).
    \end{align*}
\end{lemma}
\begin{proof}
    Let $\vrtx(S):=(v_0,\dots,v_k)$. Since we are free to relabel the vertices of $S$, it suffices to show that 
    \begin{align*}
        F^k(S)=\CH_\Omega(v_0,\dots,v_{k-1}) \subseteq \bigcup_{i=0}^{k-1} F^i(S).
    \end{align*}
    We will prove this by induction on $k$. 
    
    First we verify the base case for $k=2$. A degenerate 2-simplex is a closed projective line segment $\ell$. Then, either $v_2$ is an endpoint of $\ell$ or $v_2$ lies in the relative interior of $\ell$. In the latter case, $\ell=[v_0,v_1]$ with $v_2$ in its interior so that $[v_0,v_1]\subseteq [v_0,v_2] \cup [v_2,v_1]$ is verified. In the former case, $\ell=[v_0,v_2]$ or $\ell=[v_1,v_2]$. Thus $[v_0,v_1]\subseteq \ell  \subseteq [v_0,v_2]\cup [v_1,v_2]$. This proves the base case.

    Now suppose our claim is true for some $k\ge2$. We now prove that the statement holds for $(k+1)$. 
    Consider a degenerate $(k+1)$-simplex $S'\coloneqq \CH_\Omega(v_0,\dots,v_{k+1})$ and $F'\coloneqq F^{k+1}(S')=\CH_\Omega(v_0,\dots,v_k)$. Let 
    \begin{align*}
      d   &\coloneqq \dim(\Pb(\Span\set{v_0,\dots,v_{k+1}}))-\dim(\Pb(\Span\set{v_0,\dots,v_k})) \\
          &<   (k+1) - \dim(\Pb(\Span\set{v_0,\dots,v_k})). 
    \end{align*} 
    Then $d$ is either 0 or 1.

    \noindent \underline{Case 1.} If $d=1$, then $F'$ is a degenerate $k$-simplex, as $\dim(\Pb(\Span\{v_0,\dots,v_k\}))< k+1-d< k.$ Then, by the induction hypothesis,
    $$F'':=\CH_{\Omega}(v_0,\dots,v_{k-1}) \subset \bigcup_{i=0}^{k} \CH_\Omega(v_0,\dots,\hat{v}_i,\dots,v_k).$$
    As $F'=\CH_{\Omega}(v_k,F'')$, we have 
    $F'\subseteq \bigcup_{i=0}^{k} \CH_\Omega(v_0,\dots,\hat{v}_i,\dots,v_k).$
    Finally note that $\CH_\Omega(v_0,\dots,\hat{v}_i,\dots,v_k) \subseteq F^i(S')$. Hence, $$F^{k+1}(S')=F'\subseteq \cup_{i=0}^{k}F^i(S').$$ 

    \noindent \underline{Case 2.} If $d=0$, then $v_{k+1} \in \Pb\Span\set{v_0,\dots,v_k}$. 
    Then, we claim that 
    \begin{equation}
    \label{eqn:degenerate_but_not_in_CH}
       \CH_\Omega(v_0,\dots,v_{k+1})
       \subset
     \bigcup_{i=0}^k\CH_\Omega(v_0,\dots,\hat{v}_i,\dots,v_{k+1}). 
    \end{equation} 
    Before proving the claim, we note that the claim finishes the proof since $F'$ is contained in $\CH_\Omega(v_0,\dots,v_{k+1})$, that is,  the left hand side of \cref{eqn:degenerate_but_not_in_CH}, while the right hand side is precisely $\bigcup_{i=0}^kF^i(S')$. 

    To prove the claimed \cref{eqn:degenerate_but_not_in_CH}, note that 
    \begin{align*}
    \CH_\Omega(v_0,\dots,v_{k+1}) &=\CH_\Omega\left(v_{k+1},\CH_\Omega(v_0,\dots,v_k)\right)=\CH_\Omega(v_{k+1},F').
    \end{align*}
    Since $d=0$, $\dim(\Span\set{v_0,\dots,v_k})\leq k$. If this dimension is strictly less than $k$, then $F'=\CH_\Omega(v_0,\dots,v_k)$ is a degenerate $k$-simplex and the same argument as in Case 1 can be repeated to prove that $F'\subseteq \bigcup_{i=0}^k \CH_\Omega(v_0,\dots,\hat{v}_i,\dots,v_k)$. Thus \cref{eqn:degenerate_but_not_in_CH} is immediate. 
    Otherwise, if $\dim(\Span\set{v_0,\dots,v_k})=k$, then $F'$ is a non-degenerate $k$-simplex. Thus, if we pick any $p\in \relint(F')$, then there exists a unique $p' \in \partial F'$ such that $p\in (v_{k+1},p')$. Consequently, $p\in \CH_\Omega(v_0,\dots,\hat{v}_i,\dots, v_{k+1})$ for some $i=0,\dots,k$. This finishes the proof of the claim.
\end{proof}

As an immediate consequence of \cref{lem: facets of degenerate simplices}, we have the following corollaries.

\begin{corollary}\label{cor: degenerate are 0-slim}
     Suppose $S$ is a degenerate $k$-simplex in $\Omega$ for some $k\geq2$. Then $S$ is $0$-slim.
\end{corollary}

\begin{corollary} \label{cor:degenerate_iff_0_slim}
  Let $S$ be a $k$-simplex in $\Omega$ with $k\geq 2$. Then $S$ is $0$-slim if and only if $S$ is a degenerate $k$-simplex.
\end{corollary}
\begin{proof}
    It is a consequence of \cref{cor: degenerate are 0-slim} and \cref{lem:non_deg_simplices_not_0_slim}.
\end{proof}

\begin{corollary}\label{cor: degenerate simplices are the union of all facets}
    Suppose $S$ is a degenerate $k$-simplex in $\overline{\Omega}$ for some $k\geq2$. Then $S=\bigcup_{i=0}^k F^i(S).$
\end{corollary}
\begin{proof}
    Let $\vrtx(S):=(v_0,\dots,v_k)$. For any $i=0,\dots,k$, denote the facet of $S$ opposite to $v_i$ by $F^i$. As $S=\CH_\Omega(v_0,F^0)$, we have $S=\bigcup_{x\in F^0}[x,v_0].$ 
    Let $x\in F^0$. Lemma \ref{lem: facets of degenerate simplices} implies that $F^0\subseteq\bigcup_{i=1}^k F^i.$  Then there exists $i\in\{1,\dots,k\}$ such that $x\in F^i$. As $v_0\in F^i$, $[x,v_0]\in F^i$. Hence $S\subset \cup_{i=0}^k F^i.$
\end{proof}

Next, we will establish that in a degenerate $k$-simplex (which we have proven to be $0$-slim), the intersection of all the facets is non-trivial. To establish this result, we will rely on a classical result called the Helly's Theorem. See \cite{Helly} for the original proof, and \cite{BK1999, BF2009} for further references and generalizations.

\begin{theorem}[Helly's Theorem]\label{thm: Helly theorem}
    Let $n,m\in\mathbb{N}$ be positive integers such that $n\geq m+1$. Let $\mathcal{G}$ $= \{C_1,  \dots, C_n\}$ be a finite family of convex sets in $\mathbb{R}^m$.
    If the intersection of every sub-collection of $m+1$ sets from $\mathcal{G}$ is non-empty, then the intersection of all sets in $\mathcal{G}$ is non-empty.
\end{theorem}
\begin{corollary}\label{cor: degenerate simplices centroid}
    Let $\Omega\subset \RP$ be a properly convex domain. Suppose $S$ is a degenerate $k$-simplex in $\overline{\Omega}$ for some $k\geq2$. Then 
    $\left(\bigcap_{i=0}^kF^i(S) \right)\neq \emptyset.$
\end{corollary}
\begin{proof}
     Let $\vrtx(S):=(v_0,\dots,v_k)$. Since $S$ is degenerate, $\dim\left(\Pb\Span\{v_0,\dots,v_k\}\right)\le k-1$. Since $\Omega$ is properly convex, the domain $\Omega'\coloneqq\Pb\Span\{v_0,\dots,v_k\}\cap\Omega$ is also a \pcd. Consequently, we can work in an affine chart of $\Pb(\Rb^{d-1})$ where $\Omega'$ is bounded. Here, we can apply \cref{thm: Helly theorem} to the collection $\{F^i(S):i=0,\dots,k\}$. The conclusion then follows from \cref{thm: Helly theorem}.
\end{proof}

\subsection{Fat simplices in projective simplex domains}

Recall the definition of a projective simplex domain (\cref{defn:projective_simplex_domain}).

\begin{corollary}\label{cor: PES contain arbitrarily fat}
    Let $\Omega$ be a projective simplex domain in $\RP$ and $1\le k\le d-1$. Then, for any $n\in\mathbb{N}$, there exists a non-degenerate compact $k$-simplex in $\Omega$ that is \emph{not} $n$-slim. 
\end{corollary}
\begin{proof}
By assumption, there is an affine chart where $\Omega$ is an Euclidean $d$-simplex. We prove the result by induction on $d\ge2$.
Assume that $d=2$, then $\Omega$ is a segment. In this case, it is clear that for any $n\in\mathbb{N}$ there exist $x_n,y_n\in\Omega$ so that $\hil(x_n,y_n)>n$. Hence, $S=\CH_\Omega(x_n,y_n)$ is a 1-simplex that is not $n$-slim.

Now assume that the thesis holds for a fixed $d\ge2$. Consider a projective simplex domain $\Omega$ in $\Pb(\Rb^{d+1})$. Fix an affine chart where $\Omega$ is an Euclidean $d$-simplex.

Fix $k=1,\dots,d$.
When $k\le d-1$, we can consider a $k$-plane $H_k$ that intersects $\Omega$ and is parallel to some $k$-dimensional face of $\Omega$ in the fixed affine chart. Then, we consider the \pcd
$$
\Omega'\coloneqq\Pb\Span\{H_k\}\cap\Omega.
$$
By construction $\Omega'$ is a projective simplex domain. Since $k\le d-1$,can apply the inductive hypothesis to $\Omega'$. Therefore, we conclude that $\Omega'$, and hence $\Omega$, has a simplex that is not $n$-slim for any $n\in\mathbb{N}$. 

When $k=d$, we denote by $x_0,\dots,x_d$ the vertices of $\Omega$. Then, we consider, for any $i=0,\dots,d$, a sequence $\seq{v_n^i}$ in $\Omega$ converging to $x_i$. We claim that, up to extracting a subsequence, the simplex $S_n\coloneqq\CH_\Omega(v_n^0,\dots,v_n^d)$ is not $n$-slim. This will conclude the proof. To prove the claim, fix a point $p\in \relint(\CH_{\Omega}\{x_1,\dots,x_d\})$, the open face of $\Omega$ opposite to $x_0$. Then, consider a sequence $p_n\in F^0(S_n) \subset \Omega$ that converges to $p$. Then for any $i\in \{1,\dots,d\}$, $\liminf_{n\to\infty}\hil(p_n,F^i(S_n))=\infty$. Otherwise \cref{lem:relint_and_faces} would imply that $p \in F_{\Omega}(\CH_{\Omega}(x_0,\dots,\wh{x}_i,\dots,x_d))=\CH_{\Omega}(x_0,\dots,\wh{x}_i,\dots,x_d)$. This contradicts that $p\in \relint(\CH_{\Omega}\{x_1,\dots,x_d\})$.  
\end{proof}

\subsection{Non-emptiness of coarse centroid implies slimness}

Recall the definition of coarse centroid for a simplex from \cref{defn:coarse_centroid}. 
\begin{observation}
\label{obs: convexity of centroids}
    For any compact $k$-simplex $S$ in $\Omega$, $C_{\delta}(S)$ is a convex set. 
\end{observation}
\begin{proof}Obvious from \cref{cor: convexity of neighborhoods} and the fact that finite intersection of convex sets is convex. \end{proof}

Now we will show that if the coarse centroid of $S$ is non-empty, then $S$ must be  a slim simplex.  

\begin{lemma}
\label{lem:non-empty centroid implies slimness}
    Suppose $\Omega \subset \RP$ is a properly convex domain and $S$ is a compact $k$-simplex in $\Omega$. Then for any $\delta \geq 0$, if $C_{\delta}(S)$ is non-empty, then $S$ is $2\delta$-slim. 
\end{lemma}
\begin{remark}
    This lemma is elementary and follows from the convexity of the Hilbert metric. However, it's converse is a delicate issue. 
    A primary case of the converse has already been proven in \cref{cor: degenerate simplices centroid}: for a $0$-slim simplex $S$, $C_0(S) \neq \emptyset$. The general case of the converse requires more sophisticated machinery and we will establish this in \cref{prop: center of a simplex}.
\end{remark}
\begin{proof}We have to show that for every $i\in\{0,\dots,k\}$, we have
    \begin{equation*}
        F^i(S)\subseteq\bigcup\limits_{j\neq i}\ngh{F^j(S)}{2\delta}.
    \end{equation*}

    Take a point $p\in C_\delta(S)$. By definition, $\hil(p,F^j(S))<\delta$ for any $j=0,\dots,k$. So for each $j$, we can pick $p^j\in\relint(F^j(S))$ such that $\hil(p,p^j)<\delta$. Then $\hil(p^i,p^j)<2\delta$ for all $j\neq i$.
    
    We decompose each $F^i(S)$ into sub-simplices having $p^i$ as a vertex. More precisely, for any $j\neq i$, define the $(k-1)$-simplex in $\Omega$ 
    $$
    F^i_j\coloneqq\CH_\Omega(p^i,F^{\{i,j\}}(S)).
    $$
    Then for every $i=0,\dots,k$, we have $F^i(S)=\bigcup\limits_{j\neq i}F^i_j.$   
    Since $\hil(p^i,p^j)<2\delta$, Lemma \ref{lem: maximum principle} implies that 
    $
    F^i_j=\CH_\Omega(p^i,F^{\{i,j\}}(S))\subseteq\ngh{\CH_\Omega(p^j,F^{\{i,j\}}(S))}{2\delta}=\ngh{F^j_i}{2\delta}.$
    Note that $\ngh{F^j_i}{2\delta}\subseteq\ngh{F^j(S)}{2\delta}$ for any $i\neq j$. Hence, for any $i=0,\dots,k$, 
   \begin{equation*}
    F^i(S)=\bigcup\limits_{j\neq i}F^i_j\subseteq\bigcup\limits_{j\neq i}\ngh{F^j(S)}{2\delta}. \qedhere
    \end{equation*}
\end{proof}

\section{Slim slimpices in properly convex domains}
\label{sec: slim_simplices_properly_convex}
In this section, we relate slimness of simplices to the projective simplex rank $\rF(\Omega)$ and prove a stronger version of \cref{thm:slim_simplex_above_pes_rank}. 

\subsection{KKM Lemma}

We begin by recalling a classical result in convex geometry that we will use crucially in this section. This result, often called the KKM lemma, was first proven by Knaster, Kuratowski, and Mazurkiewicz in 1929 \cite{KKM} (also see \cite{AliprantisBorder} for further details).  For the sake of completeness, we include an elementary proof using Brouwer's fixed point theorem. The proof is routine and we do not claim any originality.  We will use the following notation for Euclidean convex hulls: given any $(k+1)$ points $x_0,\dots,x_k \in \Rb^n$, let us denote by $$\Hull(\{x_0,\dots,x_k\})=\left\{\sum_{i=0}^kt_ix_i: \sum_{i=0}^kt_i=1, t_0\geq 0, \dots,t_k\geq 0 \right\}.$$

\begin{lemma}[Knaster-Kuratowski-Mazurkiewicz]
\label{lem:KKM}
    Let $\Delta \subset \mathbb{R}^n$ be an $n$-dimensional Euclidean simplex with vertices $\{v_0, \dots, v_n\}$, i.e. $$\dim\Span\{v_0,\dots,v_n\}=n \text{ and } \Delta=\Hull(\{v_0,\dots,v_n\}).$$ 
    Let $C_0, \dots, C_n$ be a collection of closed subsets of $\Delta$. Assume that for every non-empty subset of indices $I \subseteq \{0, \dots, n\}$, $$\Hull(\{v_i: i \in I\})\subset \bigcup_{i\in I}C_i.$$
    Then, $\bigcap_{i=0}^n C_i \neq \emptyset.$
\end{lemma}
\begin{proof}
    Let $d_{\Rb^n}$ be the Euclidean distance. Assume by contradiction that $\bigcap_{i=0}^n C_i = \emptyset$. Then, for any $i=0,\dots,n$, we consider the continuous function $f_i:\Delta\to \mathbb{R}_{\ge0}$ defined by 
    $$
    f_i(x)\coloneqq \operatorname{d}_{\mathbb{R}^n}(x,C_i). 
    $$
    Since we assumed that $\bigcap_{i=0}^n C_i =\emptyset$, for any $x\in\Delta$, there exists some $i'\in\{0,\dots,n\}$ such that $f_{i'}(x)>0$. Thus, the  function $f:\Delta\to\Delta$ defined by
    $$
    x \mapsto f(x)\coloneqq\sum\limits_{i=0}^n\dfrac{f_i(x)}{\sum\limits_{j=0}^n f_j(x)}v_i,
    $$
    is a well-defined continuous function on $\Delta$. 
    By assumption, $\Delta \subset \cup_{i=0}^nC_i$. So, for any $x\in\Delta$, there exists some $i''\in\{0,\dots,n\}$ such that $f_{i''}(x)=0$. This implies that $$f(\Delta) \subset \partial \Delta=\bigcup_{i=0}^n \Hull(\{v_0,\dots,\wh{v}_i,\dots,v_n\}).$$
    Then, by Brouwer's fixed point theorem, we have that $f$ has a fixed point $x_0\in\partial \Delta$.

    Now write $x_0$ as the convex combination of the vertices of $\Delta$. Indeed, there is a unique $I \subset \{0,\dots,n\}$ and a unique such expression $x_0=\sum_{i\in I}a_iv_i$ where each $a_i$ in the sum is positive. As $x_0 \in \Hull(\{v_i:i\in I\}) \subset \bigcup_{i\in I}C_i$, there exists some $\ell \in I$ such that $f_{\ell}(x_0)=0$.  Then, if we write $f(x_0)=\sum_{i=0}^nc_iv_i$, the coefficient $c_{\ell}=0$. Since $f(x_0)=x_0$, this implies that $\ell \not\in I$ because of the way we defined $I$. This is a contradiction. 
\end{proof}

\subsubsection{A first application of KKM lemma: Quantitative fatness}

\begin{lemma}\label{lem: non-n-slim simplices}
    Let $\Omega\subset\rpd$ be a \pcd. Fix $n,k\in\mathbb{N}$ with $k\geq 2$. Then, for any compact $k$-simplex $S$ in $\Omega$ that is \emph{not} $n$-slim, there exists a point $c\in \relint (S)\subset\Omega$  such that
    \begin{equation*}
        \hil(c, F^0(S)\cup\dots\cup F^k(S))> \frac{n}{3}.
    \end{equation*}  
\end{lemma}
\begin{remark*}
    Note that this result is about non-degenerate  $k$-simplices in $\Omega$. Indeed, $S$ is non-degenerate, because if $S$ was degenerate, then it would be $0$-slim (\cref{lem: facets of degenerate simplices}) and hence, $n$-slim.
\end{remark*}
\begin{proof}
    We first choose an affine chart so that we can view $S$ as a compact Euclidean simplex in this affine chart. Moreover, since by assumption $S$ is not $n$-slim, by \cref{cor: degenerate are 0-slim}, we have that $S$ is a non-degenerate simplex in the affine chart. Hence we may now apply \cref{lem:KKM} to $S$. 
    To prove the result, suppose by contradiction that for any $c\in\relint(S)$ there exists a facet $F^j(S)$ for some $j=j(c)\in\{0,\dots,k\}$ such that $\hil(c,F^j(S))\le\dfrac{n}{3}$. This implies that $\relint(S)\subset\bigcup_{i=0}^k\overline{\ngh{F^i(S)}{\frac{n}{3}}}$. As $\partial S=\bigcup_{i=0}^k F^i(S)$, we have
    $
    S\subset\bigcup\limits_{i=0}^k\overline{\ngh{F^i(S)}{\frac{n}{3}}}.
    $
    Set $V_i\coloneqq \overline{\ngh{F^{i+1}(S)}{\frac{n}{3}}}\cap S$ for all $i=0,\dots,k-1$, and $V_k\coloneqq \overline{\ngh{F^{0}(S)}{\frac{n}{3}}}\cap S$.
    Then \cref{lem:KKM} implies that there exists $p\in\bigcap\limits_{i=0}^kV_i=\bigcap\limits_{i=0}^k\overline{\ngh{F^i(S)}{\frac{n}{3}}}\cap S$. 
    So $ p \in C_{\frac{n}{2}}(S)\cap S$. Then, by \cref{lem:non-empty centroid implies slimness}, $S$ is $n$-slim, which is a  contradiction.
\end{proof}

\subsection{Characterizing slimness using projective simplex rank}
\label{sec:slimness_above_PES_rank}

In this section, we prove a stronger version of \cref{thm:slim_simplex_above_pes_rank}. The main idea of the proof is to follow a re-centering procedure, inspired by Benoist's proof of \cite[Proposition 2.5]{B2004} Indeed, if there are simplices that get progressively `fatter', then we bring the `fat' part of each such simplex a fixed compact set, and then look at the limit of these `fat' simplices. It turns out that the limiting object is a PES of dimension greater than the projective simplex rank $\rF(\Omega)$, thus yielding a contradiction.

\begin{proposition}\label{prop: duality between properly embedded simplices and slim simplices}
    Let $\Omega\subseteq \RP$ be a quasi-homogeneous properly convex domain and $1 \leq m\le d-1$. 
    Then the following are equivalent:
    \begin{enumerate}
        \item $\rF(\Omega) \leq m$.
        \item There exists a positive constant $\delta=\delta(\Omega)$ such that every compact $k$-simplex  in  $\Omega$ with $k\ge m+1$  is $\delta$-slim. 
    \end{enumerate} 
\end{proposition}
\begin{remark}
    Note that the second condition in the result above is really a condition on the non-degenerate compact simplices in $\Omega$. This is because any degenerate compact simplex of dimension at least 2 is $0$-slim -- see \cref{lem: facets of degenerate simplices}.
\end{remark}
\begin{proof} Suppose there exists a constant $\delta >0 $ such that any non-degenerate $k$-simplex $S$ with $k\ge m+1$ in $\Omega$ is $\delta$-slim. We need to show that any PES in $\Omega$ has dimension at most $m$.
    Assume by contradiction that for some $k\ge m+1$, there exists a $k$-dimensional PES in $\Omega$ that we will label $S$. Then $S$ is a projective simplex domain in $\Pb(\Span(S))$. Applying Corollary \ref{cor: PES contain arbitrarily fat} to $S$ with every $n\in \Nb$, we get a sequence $\seq{S_n}_{n\in \Nb}$ of  non-degenerate compact $k$-simplices in $S$ where $S_n$ is \emph{not} $n$-slim. Now note that $\vrtx(S_n) \in S^{k+1} \subset \Omega^{k+1}$.
    Thus each $S_n$ is a compact $k$-simplex in $\Omega$ that is \emph{not} $n$-slim.  Choosing $n>\delta$, we get compact $k$-simplices in $\Omega$ that are \emph{not} $\delta$-slim,  a contradiction.

    Conversely, let $1 \leq m\le d-1$ be such that $m \geq \rF(\Omega)$, i.e. any PES in $\Omega$ has dimension at most $m$. 
    Assume by contradiction that for any $n\in\mathbb{N}$ there exists a natural number $k_n\geq (m+1)\geq 2$ and a compact $k_n$-simplex $S_n$ in $\Omega$ such that $S_n$ is \emph{not} $n$-slim. Note that each $S_n$ is necessarily a non-degenerate simplex, since any degenerate $k'$-simplex with $k' \geq 2$ is $0$-slim (\cref{lem: facets of degenerate simplices}). As each $S_n$ is a non-degenerate simplex in $\Omega \subseteq \RP$, $k_n \leq (d-1)$. Then, up to passing to a subsequence, we may assume that there exists $k\in\mathbb{N}$ with $2\leq m+1 \leq k \leq d-1$ such that $k_n=k$ for all $n\in\mathbb{N}$.

    Let $v^0_n,\dots,v^k_n\in\Omega$ be the vertices of $S_n$. For $i=0,\dots,k$, let $F^i_n:=F^i(S_n)$ denote the facet of $S_n$ opposite to the vertex $v^i_n$.  
    Since $S_n$ is not $n$-slim, by \cref{lem: non-n-slim simplices}, there exists a point $c_n\in \relint(S_n)\subset\Omega$ such that 
    \begin{equation}\label{eqn: Sn is not n-slim}
        \hil(c_n, F^0_n\cup\dots\cup F^k_n)> \frac{n}{3}.
    \end{equation}

    Since $\Omega$ is quasi-homogeneous, there exists a compact set $K\subseteq\Omega$ such that $\Aut(\Omega) \cdot K=\Omega$. Then, up to translating by elements of $\Aut(\Omega)$, we may assume that  $c_n\in K$ for all $n\in\mathbb{N}$. Next, up to passing to a subsequence, we may assume that there exist $v^0,\dots,v^k,c \in \overline{\Omega}$ such that $v^{i}_n \to v^i$ for each $i$, and $c_n \to c$ as $n\to\infty$. Then $c \in K \subseteq \Omega$ as $K$ is compact. 

    Let $S\coloneqq\CH_\Omega\{v^0,\dots,v^k\}$. Indeed, $S\cap\Omega \neq \emptyset$, as $c \in S$. Then, applying \cref{lem:relint_dichotomy} to $S$, we get that  $\relint(S) \subset \Omega$. Thus, $S$ is a $k$-simplex in $\Omega$. Moreover,  by \cref{lem: convergence of simplices}, $S_n \to S\cap\Omega$ in the local Hausdorff topology. 
    
    For each $i=0,\dots,k$, denote by $F^i:=F^i(S)$ the facet of $S$ opposite to the vertex $v^i$. Each $F^i$ is a closed convex subset of $\overline{\Omega}$. We now claim that for each $F^i$, $F^i \subset \bdry$. To prove the claim, note that by \cref{lem:relint_dichotomy}, either $F^i\subset\partial\Omega$ or $\relint(F^i)\subset\Omega$ for each $i=0,\dots,k$. Suppose, if possible, the latter case holds for some $i$. Then, $F^i \cap \Omega \neq \emptyset$ and \cref{lem: convergence of simplices} implies that
    $$
    F^i_n\to F^i \cap \Omega,
    $$
    as $n\to\infty$, in the local Hausdorff topology induced by $\hil$. As $c \in \Omega$ and $F^i \cap \Omega \neq \emptyset$, $\hil(c,F^i)<\infty$. On the other hand, as $c_n\to c$ and $F^i_n \to F^i\cap\Omega$ (in the local Hausdorff topology) as $n\to \infty$, we can pass to a subsequence and assume that 
    $$\hil(c,F^i)=\lim_{n\to\infty}\hil(c,F^i_n)=\lim_{n\to\infty}\hil(c_n,F^i_n).$$
    However, by \cref{eqn: Sn is not n-slim}, $\hil(c_n,F^i_n) > \frac{n}{3}$ so that $\hil(c,F^i)>\lim_{n\to \infty} \frac{n}{3}=\infty$. This yields a contradiction. This contradiction proves the claim that $F^i\subset \bdry$ for $i=0,\dots,k$. 
     
    Thus $S$ is a $k$-simplex with each of its facets contained in $\bdry$ while $\relint(S) \subset \Omega$ and $2 \leq m+1 \leq  k \leq d-1$. 
    Then, to prove that $S$ is a PES, it suffices to show that $S$ is a non-degenerate $k$-simplex. Suppose, if possible, that $S$ is a degenerate $k$-simplex. Then \cref{cor: degenerate simplices are the union of all facets} implies that $S =\cup_{i=0}^kF^i$. Since each $F^i \subset \partial \Omega$, thus $S \subseteq \bdry$, which is a contradiction. Thus $S$ is $k$-dimensional PES where $k \geq (m+1)$. This is a contradiction.
\end{proof}

\begin{proof}[Proof of \cref{thm:slim_simplex_above_pes_rank}]
\label{proof:slim_simplex_above_pes_rank}
    \cref{prop: duality between properly embedded simplices and slim simplices} directly implies that there exists a constant $\delta_0>0$ such that if $m\ge\rF(\Omega)+1$ and $S$ is a compact $m$-simplex, then $S$ is $\delta_0$-slim.
\end{proof}

\begin{remark*}
       Suppose $\Omega$ is a projective simplex domain (\cref{defn:projective_simplex_domain}). Then $\rF(\Omega)=\dim\Omega$. Then, for any $m \geq  \rF(\Omega)+1$, a compact $m$-simplex $S$ is a degenerate simplex, and hence $S$ is $0$-slim (see \cref{lem: facets of degenerate simplices}). Thus \cref{thm:slim_simplex_above_pes_rank} holds with any $\delta_0>0$.
\end{remark*}

\begin{proof}[Proof of \cref{rmk:slim_simplices_symmetric}]
\label{proof:slim_simplices_symmetric}
    Suppose $\Omega$ is the projective model of an irreducible symmetric space $X=G/K$ and let $D$ be the associated Riemannian distance on $\Omega$. Then, \cref{lem: flat rank for symmetric spaces} implies that $\rF(\Omega)=\rEuc(X)=\rk(G)$, the real rank of $G$.
    Moreover, by \cref{lem:riem_QI_hilbert}, $(\Omega,\hil)$ and $(\Omega,D)$ are quasi-isometric via the identity map, i.e. there exist $L\ge0,K\ge1$ such that $id:(\Omega,\hil)\to(\Omega,D)$ is a $(K,L)$-quasi-isometry (\cref{defn:qiso}).
    Let $m\ge\rk(G)+1$. By \cref{thm:slim_simplex_above_pes_rank}, there exists $\delta_0>0$ such that any compact $m$-simplex $S\subset\Omega$ is $\delta_0$-slim with respect to $\hil$. So $S$ is $(K\delta_0+L)$-slim \wrt\ $D$.
\end{proof}

\subsection{Characterization of slimness using the boundary}

We now state a boundary characterization of slimness of simplices.  To state the results, we recall that a set $E \subset \partial\Omega$ is called an open face of  $\Omega$ if there exists $e \in \partial\Omega$ such that $E=F_{\Omega}(e)$. Any such open face $E$ is itself a properly convex domain in $\Pb(\Span E)$ and thus $\rF(E)$ is well-defined. We relate the slimness of simplices to the supremum of $\rF(E)$ over all open faces $E$ of $\Omega$.
\begin{proposition}
\label{prop:boundary_char_of_slimness}
Let $\Omega$ be a quasi-homogeneous \pcd\ and $1\le m\le d-1$. Then, the following are equivalent.
\begin{enumerate}
    \item $\rF(\Omega) \le m$.
    \item There exists a positive constant $\delta=\delta(\Omega)$ such that every compact $k$-simplex in  $\Omega$ with $k\ge m+1$  is $\delta$-slim.
    \item If $E\subset\partial\Omega$ is any open face of $\Omega$, then $\rF(E) \le m-1$.
\end{enumerate}    
\end{proposition}
\begin{proof}
    The equivalence $(1) \iff (2)$ is \cref{prop: duality between properly embedded simplices and slim simplices}.

    \noindent $(3) \implies (1)$:
    Assume by contradiction that $(1)$ does not hold. Hence, there exists an $(m+1)$-dimensional PES in $\Omega$ that we label $S$. Pick $u_0 \in \relint(F^0(S))$. By \cref{lemma:faces_of_PES}(1), $F^0(S) \subset \partial \Omega$ and hence $u_0 \in \partial \Omega$. Thus $F_{\Omega}(u_0) \subset \bdry$ is a open face  of $\Omega$ in the boundary.  By \cref{lemma:faces_of_PES}(3), $F^0(S)$ is a $m$-dimensional PES in $F_{\Omega}(u_0)$. Thus $\rF(F_{\Omega}(u_0)) \geq m$, which contradicts (3).

    \noindent $(2)\implies(3):$ Assume that $E\subset\partial\Omega$ is an open face of $\Omega$ such that $\rF(E) > m-1$. Let $S \subset \overline{E}$ be an $m$-dimensional PES in $E$ with  $\vrtx(S)=(v_0,\dots,v_m)\subset \overline{E} \subset \partial\Omega$. Fix any point $x\in\Omega$. Then set $v_{m+1}:=x$ and $S'\coloneqq\CH_\Omega(v_0,\dots,v_m,v_{m+1})$.  Then $F^{m+1}(S')=S \subset \partial \Omega$ and for all $j=0,\dots,m$, 
    \begin{align}
    \label{eqn:facets_of_S'}
        F^j(S') \cap \partial \Omega=\CH_{\Omega}(\{v_0,\dots,\wh{v}_j,\dots,v_m\})=F^j(S).
    \end{align}
    
    As $x \in S' \cap \Omega$, \cref{lem:relint_dichotomy} implies that  $\relint(S')\subset \Omega$. 
    For $i\in\{0,\dots,m\}$, pick a sequence $(v_n^i)_{n\in\mathbb{N}}$ that converges to $v_i$. Then set $v_n^{m+1}:=x$ for all $n\in \Nb$ and consider the sequence of compact $(m+1)$-simplices in $\Omega$ given by $S_n'\coloneqq\CH_\Omega(\{v_n^0,\dots,v_n^m,v_n^{m+1}\})$. By \cref{lem: convergence of simplices}, $S_n' \to S'\cap \Omega$ as $n\to\infty$, in the \lht. 
    Moreover, for each $i =0,\dots,m+1$,
    \begin{align}
    \label{eqn:bdry_char_lim_of_faces_S'}
        \lim_{n\to\infty}F^i(S'_n) = F^i(S'), 
    \end{align}
    as closed subsets of $\Pb(\Rb^d)$, and hence of $\overline{\Omega}$.

    We claim that there exists a sequence $\seq{k_n}$ with $k_n \to \infty$ such that, for $n$ sufficiently large, $S_n'$ is not $k_n$-slim -- thus  contradicting the hypothesis in $(2)$. So, to finish the proof, it suffices to prove the above claim. To prove the claim, we assume by contradiction that there is $\alpha\geq 0$ such that $S_n'$ is $\alpha$-slim for any $n\in\mathbb{N}$. 
    We now show that this implies, 
    \begin{align*}
        S \subset \bigcup_{j=0}^m F_{\Omega}(F^j(S)).
    \end{align*}
     Indeed, let $s' \in S$. It is a limit point for a sequence of points $\seq{s_n'}$ in $F^{m+1}(S_n') \subset \Omega$. By $\alpha$-slimness of $\seq{S'_n}$, we can find a fixed sequence $t_n' \in \bigcup_{i=0}^mF^i(S'_n)$ such that $\hil(s_n',t_n') <\alpha$. Up to passing to a subsequence, we can assume that there is a fixed  $j \in\{0,\dots,m\}$ such that each $t_n' \in F^j(S_n')$, and that $t_n' \to t' \in \overline{\Omega}$. Then  \cref{eqn:bdry_char_lim_of_faces_S'} implies that $t' \in F^j(S')$. Thus,   by  \cref{fact: lower semicontinuity of extended distance}, $s' \in F_{\Omega}(t')$. As $s' \in \partial \Omega$, $t' \in F^j(S')\cap \partial \Omega=F^j(S)$ (see \cref{eqn:facets_of_S'}). As $s' \in S$ is arbitrary, we have that 
     $S \subset \bigcup_{j=0}^mF_{\Omega}(F^j(S)).$

    Now we arrive at a contradiction. As $S$ is a $m$-dimensional PES in the relatively open set $E$, $\relint(S) \subset E$. \cref{lemma:faces_of_PES}(4) implies that $ \cup_{j=0}^m F_{\Omega}(F^j(S)) \subset \partial E.$ Thus  $S \subset \partial E$. As $m \geq 1$, $E$ is not a point. Thus, $S \subset \partial E$ implies that $\relint(S) \cap E =\emptyset$, a contradiction.
\end{proof}

 \subsubsection{Proof of \cref{cor:benoist_hyperbolicity_result}}\label{sec:proof_of_benoist_corollary}
 Here we prove that our \cref{prop:boundary_char_of_slimness} recovers a classical result of Benoist on Gromov hyperbolicity of divisible properly convex domains -- \cref{cor:benoist_hyperbolicity_result}. 

But first we need a technical lemma.
\begin{lemma}\label{lem: slimness induced in higher dim}
    Let $\delta\geq 0$ and $m\geq 2$. Suppose $\Omega \subset \RP$ is a properly convex domain such that all compact  $m$-simplices in $\Omega$ are $\delta$-slim. Then any compact $k$-simplex in $\Omega$ with $k>m$ is also $\delta$-slim.
\end{lemma}
\begin{proof}
    Let $m\ge2$ be such that all compact $m$-simplices in $\Omega$ are $\delta$-slim. We will prove by induction on $k\in\mathbb{N}\cup\{0\}$ that, all compact $(m+k)$-simplices in $\Omega$ are $\delta$-slim.
    The base step of the induction with $k=0$ is immediate. For the induction hypothesis, let us assume that the statement holds for some $k\in\mathbb{N} \cup\{0\}$. We will prove that it holds for $k+1$. To do so, let $S$ be any compact $(k+m+1)$-simplex.  We have to show that for any $i\in\{0,\dots,k+m+1\}$, 
    \begin{equation*}
        F^i(S)\subset\bigcup_{j\neq i}\ngh{F^j(S)}{\delta}.
    \end{equation*}
    It suffices to do it for $i=0$ since the other cases are similar. Since $F^0(S)$ is a $(m+k)$-simplex, it is $\delta$-slim by the inductive hypothesis. By \cref{prop: center of a simplex} there exists $p\in\relint(F^0(S))$ such that 
    \begin{equation*}
        p\in\bigcap_{j=1}^{k+m+1}\ngh{F^j(F^0(S))}{\delta}=\bigcap_{j=1}^{k+m+1} \ngh{F^{\{j,0\}}(S)}{\delta}.
    \end{equation*}
    Now note that as $p\in \relint(F^0(S))$, $F^0(S)=\bigcup_{j=1}^{k+m+1}\CH_\Omega(p,F^{\{j,0\}}(S))$.
   Since $\ngh{F^{\{j,0\}}(S))}{\delta}$ is convex for each $j$, we have
    \begin{equation*}
    F^0(S)\subset\bigcup_{j=1}^{k+m+1}\ngh{F^{\{j,0\}}(S)}{\delta} \subset \bigcup_{j=1}^{k+m+1}\ngh{F^{j}(S)}{\delta}.
    \end{equation*}
    This finishes the proof. 
\end{proof}

Now we finish the proof of \cref{cor:benoist_hyperbolicity_result}.

    \noindent(1)$\implies$(3): Assume that $\rF(\Omega)=1$. Then, \cref{prop: duality between properly embedded simplices and slim simplices}((1)$\implies$(3)) implies that for any face $E\subset\partial\Omega$ of $\Omega$ it holds that $\rF(E)=0$. This implies that there are no non-trivial segments in $\partial \Omega$.
    
    \noindent(3)$\implies$(2): Assume that $\partial\Omega$ contains no non-trivial segments. In this case, the space $(\Omega,\hil)$ is uniquely geodesic (see \cref{prop: uniquely geodesic}). 
    In particular, all geodesics triangles in $(\Omega,\hil)$ are compact 2-simplices. Therefore, \cref{prop: duality between properly embedded simplices and slim simplices}((3)$\implies$(2)), implies that (2) holds.

    \noindent (2)$\implies$(1): Assume that $(\Omega,\hil)$ is Gromov hyperbolic. Then, by definition, all geodesic triangles are $\delta$-slim for some $\delta\ge0$. As a consequence, by \cref{lem: slimness induced in higher dim}, all $(k+1)$-simplices in $\Omega$ are $\delta$-slim, for any $k\in\mathbb{N}$.
    Then \cref{prop: duality between properly embedded simplices and slim simplices}((2)$\implies$(1)) implies that $\rF(\Omega)=1$. \qed

\subsection{Slim simplices have a non-empty centroid}

\label{sec:slim_simplices_non_empty_centroid}

Recall the definition of coarse centroid from \cref{defn:coarse_centroid}. In \cref{prop: center of a simplex} below, we will upgrade the $\delta$-slimness of a $k$-simplex to the non-emptiness of its coarse centroid. The proof that we present here relies crucially on an application of the KKM Lemma (\cref{lem:KKM}). One may also interpret \cref{prop: center of a simplex} as a converse to \cref{lem:non-empty centroid implies slimness}.

\begin{proposition}\label{prop: center of a simplex}
    Suppose $\Omega\subseteq\RP$ is a properly convex domain, $\delta \geq 0$, and $k\geq 2$. If $S$ is a compact $k$-simplex in $\Omega$ that is $\delta$-slim, then 
    \begin{equation*}
       \bigcap_{i=0}^k\ngh{F^i(S)}{\delta} \neq \emptyset.
    \end{equation*}
    \end{proposition}
\begin{proof}
First, we will dispense with two exceptional cases -- when either $\delta=0$ or $S$ is a degenerate $k$-simplex. In the former case, $S$ is $0$-slim and hence it is a degenerate $k$-simplex in $\Omega$ -- see  \cref{cor:degenerate_iff_0_slim}.  Then, Corollary \ref{cor: degenerate simplices centroid}  implies that in both of these exceptional cases, $\bigcap_{i=0}^k\ngh{F^i(S)}{0}=\bigcap_{i=0}^k F^i(S)\neq \emptyset$ and we are done. 

    So, from now on, we will assume that $S$ is a non-degenerate $k$-simplex and $\delta>0$. We will prove that  there is a point $p\in S$ such that 
    \begin{equation}
    \label{eqn: existence of centroid}
         \max\limits_{i=0,\dots,k}\hil(p,F^i(S)) < \delta.
     \end{equation}
    Since $S$ is compact, it is contained in $\Omega$.  So it suffices to find a point in $F^0(S)$ satisfying Equation \eqref{eqn: existence of centroid}. 
    Since $S$ is $\delta$-slim, the facet $F^0(S)$ is contained in the $\delta$-neighborhood of the union of the other facets. Hence there exists a sufficiently small $\varepsilon>0$ such that the family $\{U_1, \dots, U_k\}$ defined by 

    $$U_i \coloneqq \{x \in F^0(S) \mid \hil(x, F^i(S))\le\delta-\varepsilon\}$$
    is a closed cover of $F^0(S)$. Notice that to conclude the proof, it suffices to show that the intersection of all $U_i$ is non-empty.

    We will prove the non-triviality of this intersection by applying \cref{lem:KKM}. First, since $\Omega$ is properly convex, we can put it in an affine chart so that $F^0(S)$ becomes an Euclidean simplex in an affine space of dimension $(k-1)$. Moreover, the vertices of $F^0(S)$ span this $(k-1)$-dimensional affine space as $S$ is a non-degenerate simplex. Hence we can apply \cref{lem:KKM} to $F^0(S)$. 
    
    But the cover $U_i$ defined above is not well adapted for the application of the KKM Lemma as $v_i\not\in U_i$ in general. So we define a new cover of $F^0(S)$ by cyclically shifting the indices of $U_i$: consider the cover $V_1, \dots, V_k$ of $F^0(S)$ where  
    $
    V_k \coloneqq U_1\quad\text{and}\quad V_i \coloneqq U_{i+1}
    $
    for $i=1,\dots,k-1$.    
    Let us now verify that $\{V_1, \dots, V_k\}$ satisfies the condition in Lemma \ref{lem:KKM}. Let $I \subseteq \{1,\dots,k\}$ be a subset of indices.
    If $I=\{1,\dots,k\}$, then 
    $$
    \CH_\Omega(\{v_i\mid i\in I\})=F^0(S)=\bigcup\limits_{i\in I}V_i.
    $$
    If $I \subsetneq \{1,\dots,k\}$, there exists an index $j \in I$ such that $j+1 \not\in I$ (modulo $k$).
    Hence, the vertices $\{v_i \mid i \in I\}$ do not include $v_{j+1}$. This implies $\CH_\Omega(v_i\mid i\in I)\subseteq U_{j+1}$ and thus
    $$
    \CH_\Omega(\{v_i\mid i\in I\})\subseteq V_j\subseteq\bigcup\limits_{i\in I}V_i.
    $$
    Thus \cref{lem:KKM} implies that $\bigcap_{i=1}^kV_i=\bigcap_{i=1}^k U_i\neq \emptyset.$ 
\end{proof}

\subsubsection{Proof of \cref{thm:slimness_gives_centroid}} 
\label{sec:proof_of_slimness_gives_centroid}

We obtain \cref{thm:slimness_gives_centroid} as a corollary of \cref{thm:slim_simplex_above_pes_rank} and \cref{prop: center of a simplex}. 
Indeed, let $\delta_0>0$ be as in \cref{thm:slim_simplex_above_pes_rank} and set $\delta'_0:=2\delta_0$. Let $x_0,\dots,x_m \in \Omega$ and consider the compact $m$-simplex in $\Omega$ given by  $S:=\CH_{\Omega}(x_0,\dots,x_m).$ As $m\geq \rF(\Omega)+1$, \cref{thm:slim_simplex_above_pes_rank} implies that $S$ is $\delta_0$-slim. Then, by \cref{prop: center of a simplex}, $C_{\delta_0}(S)=\cap_{i=0}^m \ngh{F^i(S)}{\delta_0} \neq \emptyset$. Now $\delta\geq \delta'_0$ implies that $\frac{\delta}{2} \geq \delta_0$. Then $C_{\frac{\delta}{2}}(S) \supset C_{\delta_0}(S)$. Thus $C_{\frac{\delta}{2}}(x_0,\dots,x_m)=C_{\frac{\delta}{2}}(S) \neq \emptyset$.\qed

\section{Genericity of points and bounds for coarse centroids}
\label{sec: generic_pts_and_diam_bounds}

\subsection{A notion of generic points for the Hilbert metric}

Recall the notion of genericity from \cref{def: generic_points_in_Hilbert_new}.

\begin{remark}\label{rem:alg_vs_metric_genericity}
    In projective geometry, we often say that $(m+2)$-points $x_0,\dots,x_{m+1}$ are \emph{algebraically non-generic} if $\dim\Span\{x_0,\dots,x_{m+1}\}\le m$.
    In this paper, we say that a simplex with algebraically non-generic vertices is degenerate.
    Our notion of genericity is completely different from this algebraic notion of genericity. As we will see in the examples below, there are algebraically non-generic points that are generic in our sense and points that are non-generic in our sense, and that are algebraically generic. 
\end{remark}

\begin{observation}\label{rmk:nongen_implies_nonbound}
    Let $\Omega\subset\RP$ be a \pcd \ and $\delta,D>0$. Then, if $(x_0,\dots,x_{r+1})\in\Omega^{r+2}$ is $(\delta,D)$-non-generic, then $\diam(C_\delta(x_0,\dots,x_{r+1}))> D$.
    \end{observation}
\begin{proof}
    As $(x_0,\dots,x_{r+1})$ is not $(\delta,D)$-generic, there exist $I_0,\dots,I_{q} \subset \{0,\dots,r+1\}$ non-empty proper subsets such that $\bigcap_{i=0}^{q}I_i=\emptyset$ and
         $$\diam\left(\bigcap_{i=0}^{q}\Nc_{\delta} \bigg( \CH_\Omega\left( x_{j} : j\in I_i\right) \bigg) \right) >D.$$
    Let $j \in \{0,\dots,r+1\}$. Since $\bigcap_{i=0}^{q}I_i=\emptyset$, there exists $h_j\in\{0,\dots,q\}$ such that $j\not\in I_{h_j}$. Hence $\CH_\Omega(x_l : l\in I_{h_j})\subset F^j$. Hence, 
    $\bigcap_{i=0}^{q}\Nc_{\delta} \bigg( \CH_\Omega\left( x_{j} : j\in I_i\right) \bigg)\subset C_\delta(x_0,\dots,x_{r+1}).$
    This finishes the proof.
\end{proof}

\subsection{Examples}

Before doing anything else with this notion of genericity, we discuss some examples first. For the rest of this section, $\Omega$ is always a properly convex domain in $\Pb(\Rb^d)$. 

\begin{example}[The case $m=1$]
\label{example: 3 points are generic}
    Any triple in a properly convex domain is $(\delta, D)$-generic for any $\delta > 0$ and sufficiently large $D > 0$. In fact, they are $(\delta, 2\delta)$-generic for any $\delta \ge 0$. 
    
    Fix a triple $(a_0,a_1,a_2)\in\Omega^3$. Let $I_1,\dots,I_q\subset\{1,2,3\}$ be non-empty proper subsets such that $\bigcap_{i=1}^q I_i=\emptyset$. Necessarily, there exists some index $j \in \{1, \dots, q\}$ for which $|I_j| = 1$. Consequently, 
    $$
    \operatorname{diam}\left(\bigcap_{i=1}^q \mathcal{N}_\delta\Big(\operatorname{CH}_\Omega(\{a_s \mid s \in I_i\})\Big)\right) \le \operatorname{diam}(\mathcal{N}_\delta(a_j)) \le 2\delta. 
    $$
\end{example}

\begin{example}[Comparison with algebraic non-genericity]
\label{example: alg_non_generic}
     To emphasize that our metric notion of genericity is independent of the algebraic one, we will see examples of tuples of points that are:
    \begin{itemize}
        \item algebraically non-generic, but are generic in our sense.
        \item algebraically non-generic, and are also non-generic in our sense.
    \end{itemize}

   We consider the projective model of  $\Hb^2$ --  the unit disc centered at the origin $o$ in some affine chart in $\Rb\Pb^2$.  Consider four points $x_0,\dots,x_3$, lying on the circle of radius $s$ around $o$, and labeled in the anti-clockwise direction. Clearly, these points are algebraically non-generic.
    
    Set $I_0 = \{0, 2\}$ and $I_1 = \{1, 3\}$. Notice that $I_0 \cap I_1 = \emptyset$. The convex hulls $W = \operatorname{CH}_{\mathbb{B}^2}(x_0, x_2)$ and $W = \operatorname{CH}_{\mathbb{B}^2}(x_1, x_3)$ form two intersecting line segments that cross at $o$. For any $\delta \ge 0$, the quantity 
    \begin{equation*}
        N_{V,W}(\delta) \coloneqq \operatorname{diam}\Big(\mathcal{N}_\delta(V) \cap \mathcal{N}_\delta(W)\Big)\le 2(s+\delta)
    \end{equation*}
    measures how closely $S\coloneqq \CH_{\mathbb{B}^2}(x_0,\dots,x_3)_{\mathbb{B}^2}$ resembles the 1-dimensional simplex $\CH_{\mathbb{B}^2}(x_0,x_2)$. 
    Let $\alpha_{V,W}$ be the angle between the segments $\overline{ox_0}$ and $\overline{ox_1}$.

    Then, for any $\delta,D\ge0$, we can move the points $x_0,\dots,x_3$ by varying the radius $s$ and the angle $\alpha_{V,W}$
    to produce both $(\delta,D)$-generic and $(\delta,D)$-non-generic configurations of points; \cref{fig: GenericityExample1} shows a $(\delta,D)$-generic tuple $(x_0,\dots,x_3)$, on the left, and a tuple $(y_0,\dots,y_3)$ that is not $(\delta,D)$-generic, on the right.
\end{example}

\begin{example}[Algebraically generic and generic in our sense; see \cref{fig:Generic4points}]
\begin{figure}[h]
        \centering
        \includegraphics[width=0.3\linewidth]{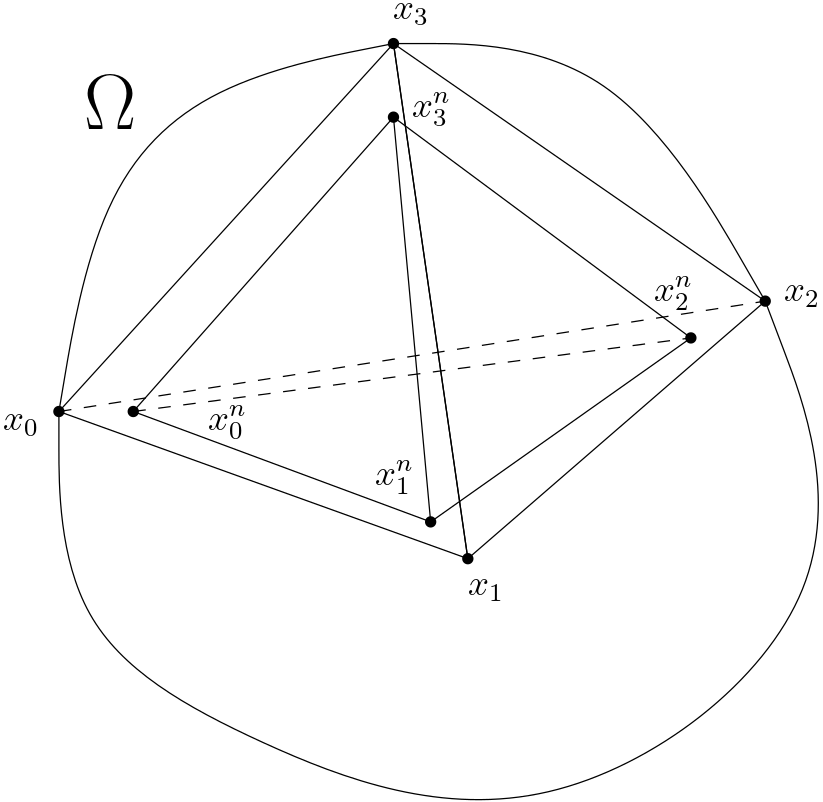}
        \caption{Non-degenerate tetrahedra with generic vertices.}
        \label{fig:Generic4points}
    \end{figure}
\label{ex:non_trivial_example_generic}
Let $\Omega\subset\RP$ be a \pcd and $r=\rF(\Omega)$. By definition, there exists an $r$-dimensional PES with vertices $x_0,\dots,x_r\in\partial\Omega$. Moreover, for any fixed $x_{r+1}\in\partial\Omega$ we have that 
\begin{equation}\label{eq:ex_4generic}
    \relint(\CH_\Omega(x_0,\dots,\wh{x_{i}},\dots,x_{r+1}))\subset\Omega.
\end{equation}
Let $\seq{x_0^n},\dots,\seq{x_{r+1}^n}$ be sequences in $\Omega$ such that $\lim_{n\to\infty}x_i^n=x_i$ for $i=0,\dots,r+1$. Set $F^i_n\coloneqq\CH_\Omega(x_0^n,\dots,\wh{x_i^n},\dots,x_{r+1}^n)$ and $F^i\coloneqq\CH_\Omega(x_0,\dots,\wh{x_i},\dots,x_{r+1})\cap\Omega$. Then, by \cref{lem: convergence of simplices}, $F^i_n\to F^i$.

We will show below that $(x_0^n,\dots,x^n_{r+1})$ is a generic $(r+2)$-tuple for all $n$ sufficiently large. 
Fix any $\delta>0$. Then, any sequence of points $\seq{p_n}$ such that $p_n\in C_\delta(x_0^n,\dots,x_{r+1}^n)$ converges to a point $p\in\bigcap_{i=0}^{r+1}\clngh{F^i}{\delta}$.

We claim that there exists $D\ge0$ such that $\diam(\bigcap_{i=0}^{r+1}\ngh{F^i}{\delta})\le D$. Indeed, if we project this intersection to $F^0$, we get that $\diam(\bigcap_{i=0}^{r+1}\clngh{F^i}{\delta})$ is at most 
$$\diam\{x\in F^0\mid \max_{i=1,\dots,r+1}\hil(x,F^i)\le2\delta\}.$$ Since $F^0$ is a PES, $\diam\{x\in F^0\mid \max_{i=1,\dots,r+1}\hil(x,F^i)\le2\delta\}$ is bounded by some $D$. Hence we have the claim.


Now choose any $D'>D$. We have that $\diam(C_\delta(x_0^n,\dots,x_{r+1}^n))\le D'$. By \cref{rmk:nongen_implies_nonbound}, we have that the tuple $(x_0^n,\dots,x_{r+1}^n)$ is $(\delta,D')$-generic for any sufficiently large $n$.
\end{example}

\begin{example}[Algebraically generic, but non-generic in our sense; see \cref{fig:NonGeneric4points}]
\label{example: alg_generic dim 4}
    On the other hand, there exist tuples of algebraically generic points that are non-generic in our sense. 
    Let $\Omega\subset\mathbb{P}(\mathbb{R}^4)$ be a 3-dimensional divisible properly convex domain. Fix $\delta>0$. If $\rF(\Omega)=2$, then the boundary $\partial\Omega$ contains non-trivial line segments. 
    In particular, we can consider two non-coplanar segments $s_1,s_2\subset\partial\Omega$ and two pairs of distinct points $p_1,q_1\in s_1$ and $p_2,q_2\in s_2$ such that $\operatorname{d}_{s_1}(p_1,q_1)<2\delta$ and $\operatorname{d}_{s_2}(p_2,q_2)<2\delta$. See \cref{fig:NonGeneric4points}.
    
    In this configuration, it is clear that $\diam(\ngh{(p_1,q_1)}{\delta}\cap\ngh{(p_2,q_2)}{\delta})=\infty$. Consequently, we can consider a sequence $(S_n)_n$ of tetrahedra whose vertices are algebraically generic, yet not $(\delta,n)$-generic.

    \begin{figure}[h]
        \centering
        \includegraphics[width=0.4\linewidth]{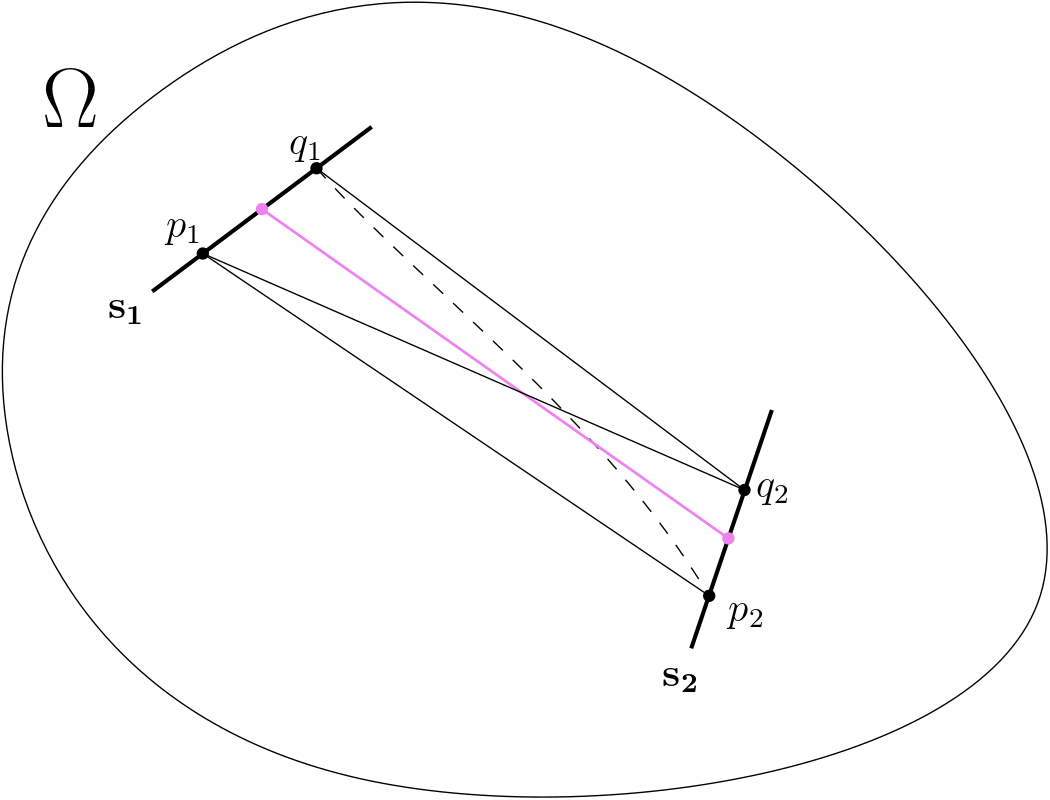}
        \caption{Non-degenerate tetrahedron with non-generic vertices.}
        \label{fig:NonGeneric4points}
    \end{figure}
\end{example}

    \begin{table}[h]
    \centering
    \renewcommand{\arraystretch}{1.5}
    \begin{tabular}{|c|c|c|}
        \hline

        \includegraphics[width=0.3\linewidth, valign=c, margin=0pt 10pt 0pt 10pt]{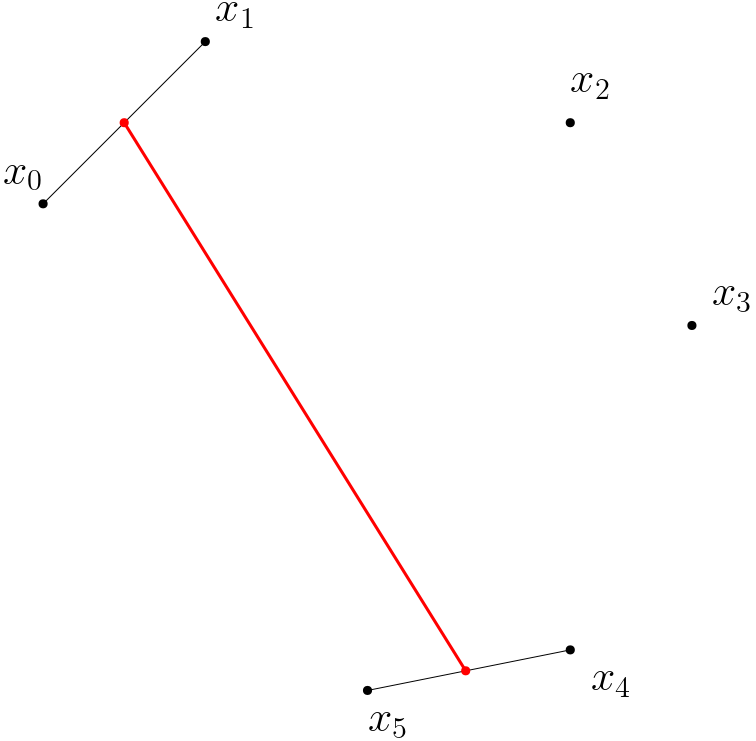} & 
        \includegraphics[width=0.3\linewidth, valign=c, margin=0pt 10pt 0pt 10pt]{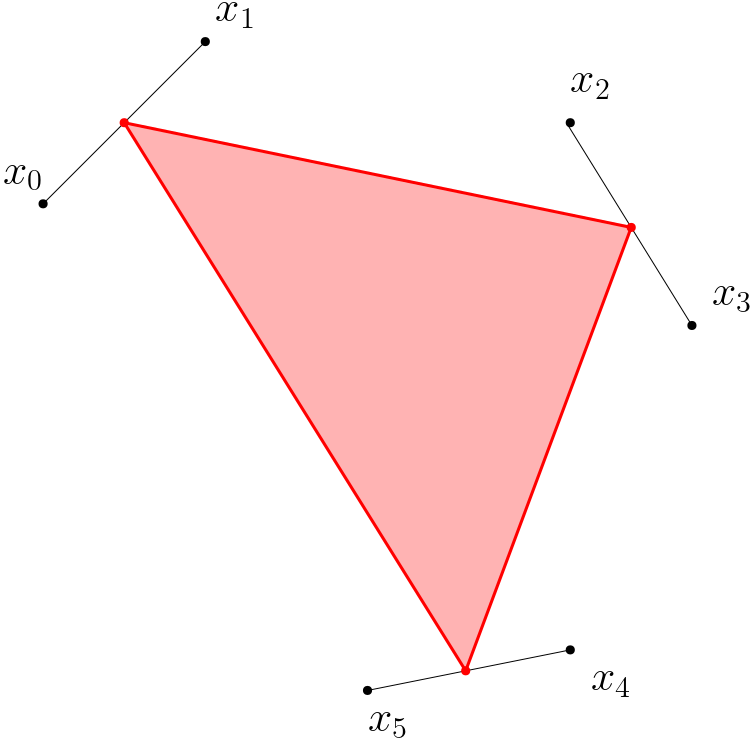} & 
        \includegraphics[width=0.3\linewidth, valign=c, margin=0pt 10pt 0pt 10pt]{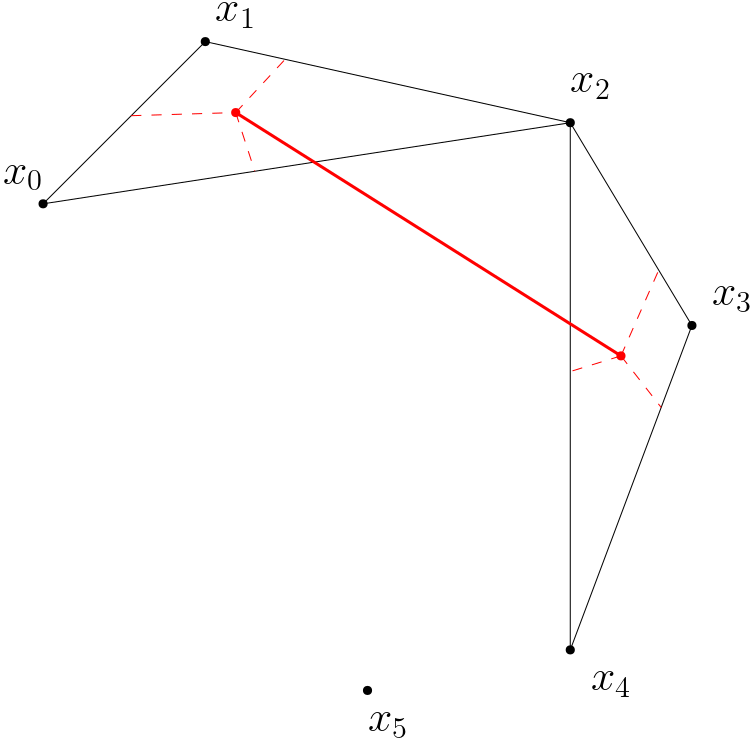} \\
       
        $(2|2|1|1)$ & $(2|2|2)$ & $(3|2|1)$ \\
        \hline
        
        \includegraphics[width=0.3\linewidth, valign=c, margin=0pt 10pt 0pt 10pt]{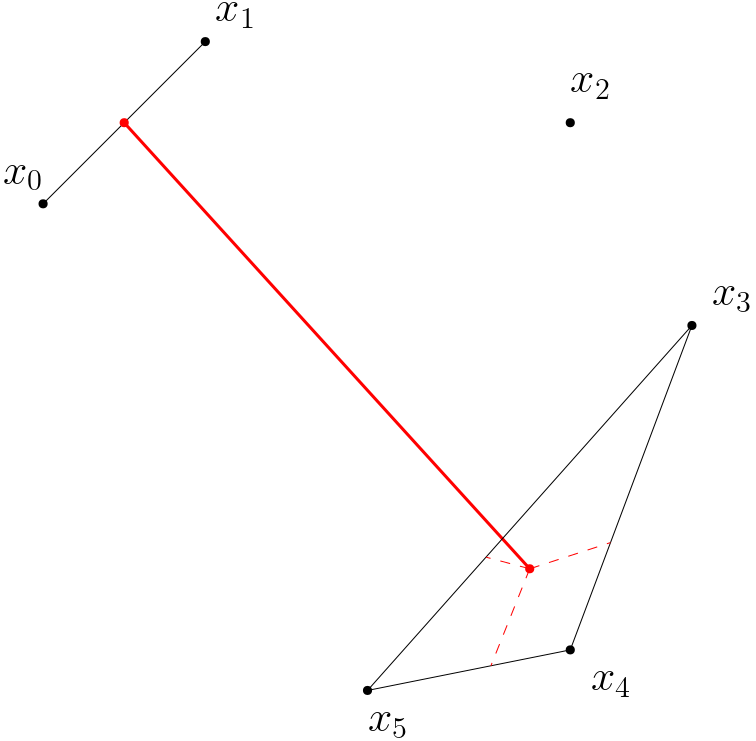} & 
        \includegraphics[width=0.3\linewidth, valign=c, margin=0pt 10pt 0pt 10pt]{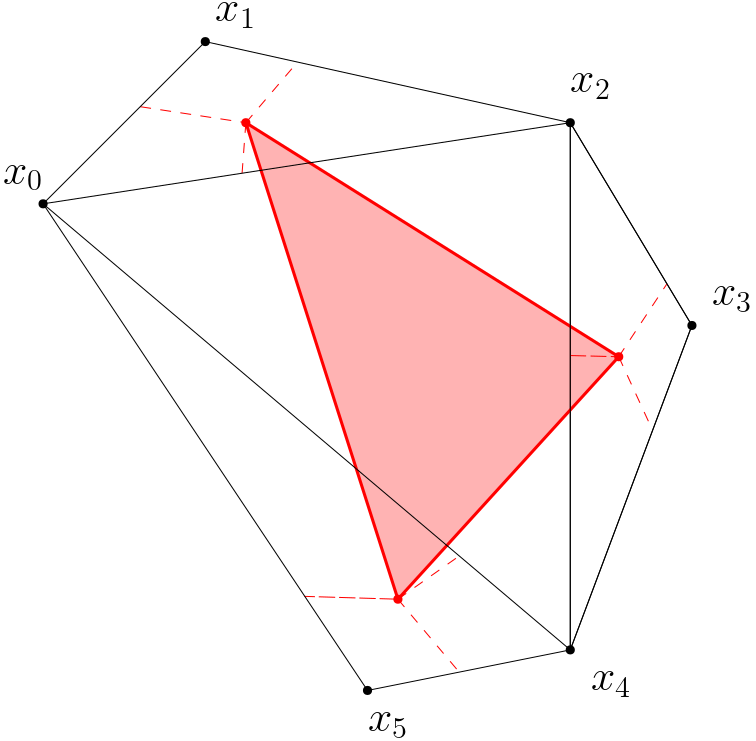} & 
        \includegraphics[width=0.3\linewidth, valign=c, margin=0pt 10pt 0pt 10pt]{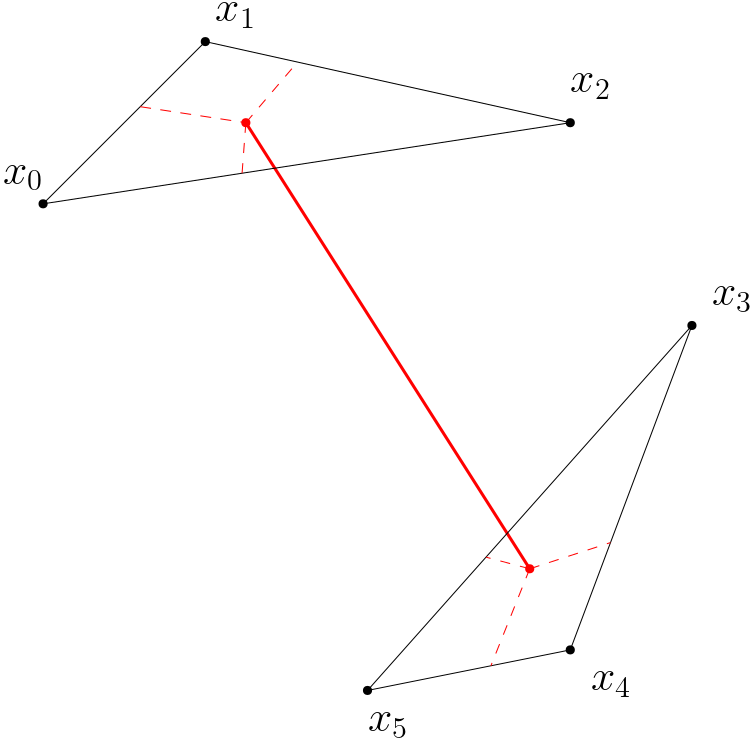} \\
        
        $(3|2|1)$ & $(3|2|1)$ & $(3|3)$ \\
        \hline
        
        \includegraphics[width=0.3\linewidth, valign=c, margin=0pt 10pt 0pt 10pt]{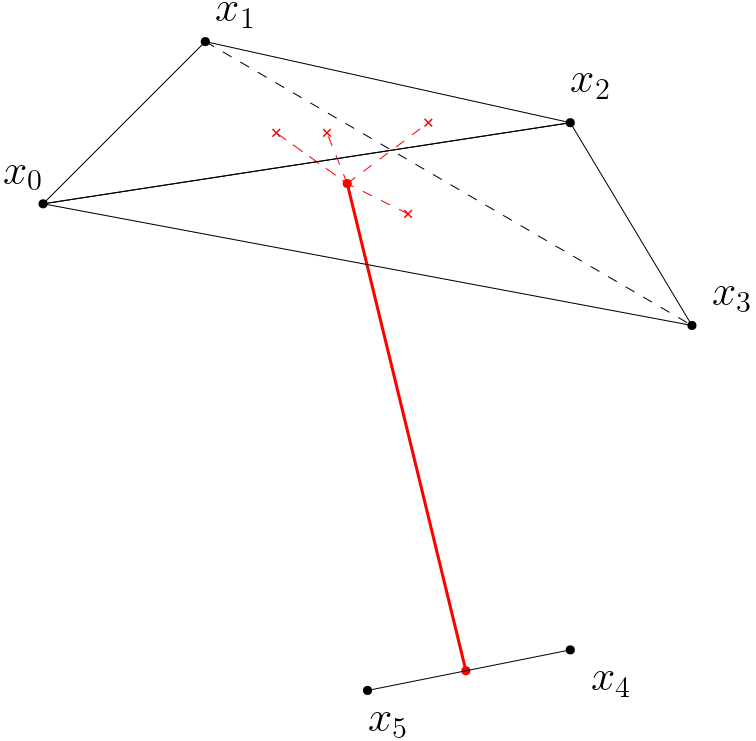} & 
        \includegraphics[width=0.3\linewidth, valign=c, margin=0pt 10pt 0pt 10pt]{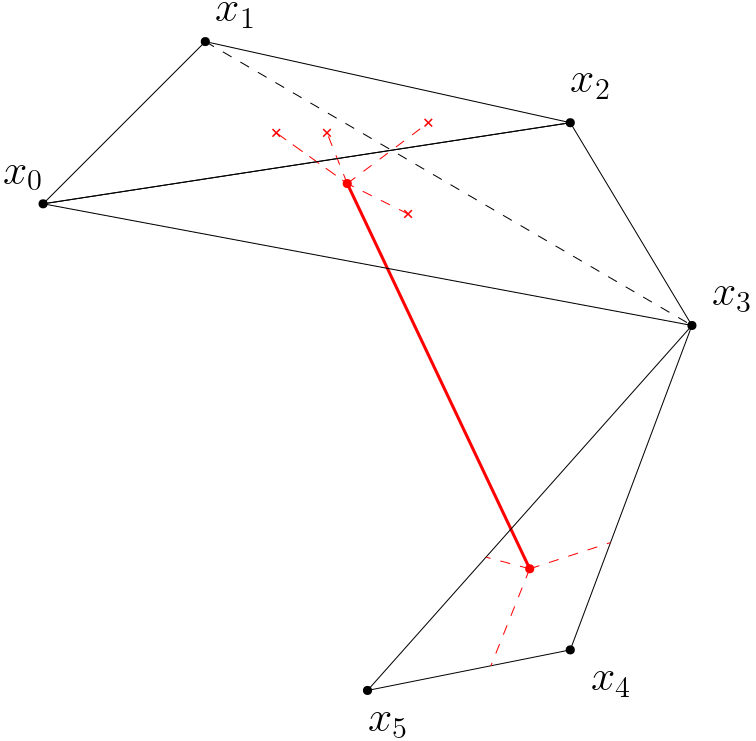} & 
        \includegraphics[width=0.3\linewidth, valign=c, margin=0pt 10pt 0pt 10pt]{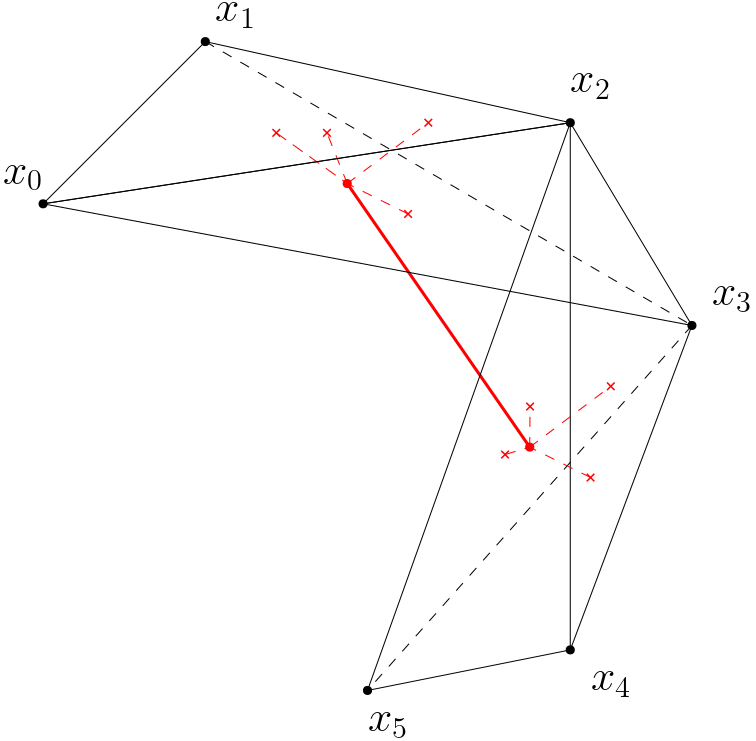} \\
        
        $(4|2)$ & $(4|2)$ & $(4|2)$ \\
        \hline
    \end{tabular}
    \caption{Possible configurations for the vertices of an ideal non-degenerate slim 5-simplex with non-generic vertices.}
    \label{tab:NonGenericIdeal6points}
\end{table}

\begin{example}[Refer to \cref{tab:NonGenericIdeal6points}]
\label{example: alg_generic dim 6}
    As the previous example shows, studying limiting ideal configurations provides good intuition for our genericity condition. We briefly outline the possible configurations for non-generic tuples of 6 ideal points that span a non-degenerate $\delta$-slim 5-simplex with infinite-diameter $\delta$-centroid, for some $\delta>0$.
    Note that the notions of a centroid and genericity naturally extend to simplices having facets whose relative interiors lie in $\Omega$.
    As we formalize in the proof of \cref{newlem: slim simplices have bounded centroid}, these configurations correspond to the integer partitions of 6 for which at least two parts are strictly greater than 1, i.e. the partitions $(2|2|1|1),(2|2|2),(3|2|1),(3|3)$, and $(4|2)$.
    See \cref{tab:NonGenericIdeal6points} for configurations associated to these partitions. Each picture represents the ideal faces of such a 5-simplex, highlighting a subset of the infinite-diameter centroid in red.

For a more detailed discussion, we take the case $(4|2)$. 
    This corresponds to a 5-simplex $S$ whose vertices are distributed between two faces in the boundary of $\Omega$. Because $\relint(S)  \subset \Omega$, these two boundary faces must be distinct. Furthermore, the $\delta$-slimness of $S$ forces the intersection of these boundary faces to have a codimension of at least 2 in each face. 
    These geometric constraints leave us with exactly three ways arrangements of the 6 vertices (see the last row of \cref{tab:NonGenericIdeal6points}): 
	\begin{itemize}
		\item 4 vertices on one face, and 2 vertices on the other -- hence the two faces are disjoint; 
    	\item 4 vertices on one face, and 3 vertices on the other -- hence the two faces intersect at one vertex; or 
    	\item the faces contain 4 vertices each -- hence the two faces intersect along an edge. 
	\end{itemize}
\end{example}

\subsection{Coarse centroid for slim triangles}

We first show that the coarse centroid of a slim triangle has uniformly bounded diameter. For $\delta$-hyperbolic spaces, it is a well-known result that `centers' of geodesic triangles have uniformly bounded diameter (see \cite[Part III.H, Theorem 1.17]{BridsonHaefliger}). We show that in the convex projective setting, the same is true for any $\delta$-slim triangle, although the metric $\hil$ may not be Gromov hyperbolic. The proof is analogous to the $\delta$-hyperbolic. But convexity plays a crucial role at the end of the proof, because of which we can get around the lack of global hyperbolicity.

\begin{lemma}\label{lem: slim triangles have bounded center}
    Let $\Omega\subseteq\RP$ be a properly convex domain and $\delta \geq0$. Suppose $T$ is a compact $2$-simplex in $\Omega$ that is $\delta$-slim. Then $\diam (C_\delta(T))\le6\delta.$
\end{lemma}
\begin{proof}
    Let $\vrtx(T)=(v_0,v_1,v_2)$ and let  $p,q\in C_\delta(T)$. Fix $i\in\{0,1,2\}$. Denote by $F^i$ the facet of $T$ opposite to $v_i$. By definition of $\delta$-centroid,  $\hil(p,F^i)<\delta$ and $\hil(q,F^i)<\delta$,. Then pick $p_i,q_i\in F^i$ so that $\hil(p,p_i)<\delta$ and $\hil(q,q_i)<\delta$. 
    By triangle inequality
    $$
        \hil(p,q)\le\min_{i=0,1,2} \left( \hil(p,p_i)+\hil(p_i,q_i)+\hil(q_i,q) \right) \leq 2\delta+\min_{i=0,1,2} \hil(p_i,q_i).
    $$
    Since we are free to relabel the vertices $v_i$, it suffices to show that $\hil(p_0,q_0)\le 4\delta$. Up to switching the labels $v_1$ and $v_2$, we can assume that $\hil(p_0,v_1)\le\hil(q_0,v_1)$. Then, by convexity, the projective segments $[p_0,p_1]$ and $[q_0,q_2]$ must intersect at some point in $T$; see \cref{fig:FigureProofSlimTriangle}. Call this point $e$. 
    Since $\hil(p_0,p_1)\le2\delta$ and $\hil(q_0,q_2)\le2\delta$, we get $\hil(p_0,q_0)\le\hil(p_0,e)+\hil(e,q_0)\le4\delta$.
\end{proof}
\begin{figure}[h]
    \centering
    \includegraphics[scale=0.12]{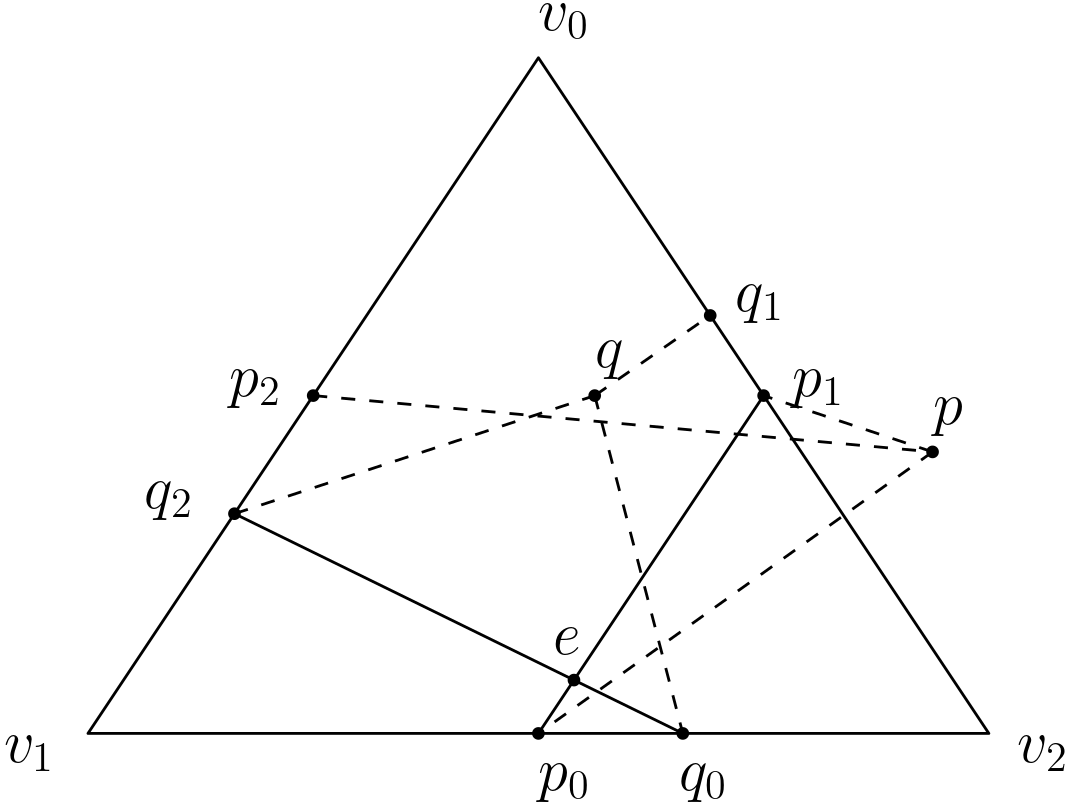}
    \caption{Centroid of a slim triangle.}
    \label{fig:FigureProofSlimTriangle}
\end{figure}
\begin{remark}
    We remark that in the proof above, it is the existence of the intersection point $e:=[p_0,p_1]\cap[q_0,q_2]$ (see \cref{fig:FigureProofSlimTriangle}) that relies on convexity and does not need to hold for a general Gromov hyperbolic metric space. It is this convexity that lets us get around the lack of global hyperbolicity. 
\end{remark}

\subsection{Centroids of generic points are coarsely unique}

We will now prove that any $\delta$-slim simplex, either its vertices are non-generic, or its $\delta$-centroid has uniformly bounded diameter. Fix $\delta,D>0$ and $k\in\mathbb{N}$. Then, we denote the set of $(\delta,D)$-generic $(r+2)$-tuples as 
\begin{equation*}
    \Omega^{(k+1)}_{(\delta,D)}\coloneqq\{(x_0,\dots,x_{r+1})\in\Omega^{r+2}\mid (x_0,\dots,x_{r+1}) \text{ is a $(\delta,D)$-generic tuple}\}.
\end{equation*}

\begin{remark}\label{rmk:existence_generic_tuples_pcd}
    The set $\Omega^{(k+1)}_{(\delta,D)}$ may be empty in general. However, if $D$ is chosen sufficiently large, this is not the case as we will now explain.  First let us discuss the ``boring" example: if $D\ge2\delta$, then for any $o\in\Omega$, the $(k+1)$-tuple $(o,\dots,o)\in\Omega^{(k+1)}_{(\delta,D)}$. 
    
    But the really interesting examples show up in  \cref{ex:non_trivial_example_generic}. There, we generate lots of $(\delta,D)$-generic $(\rF(\Omega)+2)$ tuples when $D$ is sufficiently large. 
\end{remark}

\begin{lemma}\label{newlem: slim simplices have bounded centroid}
    Suppose $\Omega \subset \RP$ is a quasi-homogeneous properly convex domain.  Fix $\delta, D >0$ and $k\geq 2$.  Then there exists $R=R(\delta,D,k)$ such that for any compact $k$-simplex $S$ in $\Omega$ with $\vrtx(S) \in \Omega^{(k+1)}_{(2\delta,D)}$, 
    $$\diam(C_{\delta}(S)) < R.$$

\end{lemma}

\noindent The rest of this section is devoted to the proof of \cref{newlem: slim simplices have bounded centroid}. 

\noindent Fix some $k\ge2$. Assume by contradiction that for any $n\in\mathbb{N}_*$, there exists a $\delta$-slim compact $k$-simplex $S_n$ in $\Omega$, with $(2\delta,D)$-generic vertices, such that
    \begin{equation*}
        \diam(C_\delta(S_n))\ge n.
    \end{equation*}

    By Lemma \ref{lem: slim triangles have bounded center}, we must have $k\ge3$.
    Moreover, we have that $C_\delta(S_n)\neq\emptyset$, for any $n>0$.
    Let $v_n^0,\dots,v_n^k$ be the vertices of $S_n$.
    For each $i=0,\dots,k$, denote by $F^i_n$ the face of $S_n$ opposite to $v_n^i$. As $\diam(C_{\delta}(S_n)) \geq n$,  there exist $a_n,b_n\in C_\delta(S_n)$ such that  
    \begin{equation*}
        \hil(a_n,b_n)\ge n.
    \end{equation*}
    Consider the midpoint $c_n$ of the projective segment $[a_n,b_n]$ (in the Hilbert metric). By \cref{obs: convexity of centroids}, $[a_n,b_n] \in C_{\delta}(S_n)$. Thus $$c_n \in C_{\delta}(S_n) \text{ and } \hil(c_n,S_n)<\delta.$$

    Since $\Omega$ is quasi-homogeneous, there exists a compact set $K\subset \Omega$ such that $\Aut(\Omega) \cdot K =\Omega$. Then, up to translating the sequence $\seq{c_n}$ by elements in $\Aut(\Omega)$, we can assume that  $c_n\in K$ for all $n\in\mathbb{N}$. Up to passing to subsequences, we can assume that there exist $v^0,\dots,v^k,a,b,c\in\overline{\Omega}$ such that $\lim_{n\to\infty}v_n^i=v^i$ for $i=0,\dots,k$  and $a,b,c$ are the limits of \Seq{a_n},\Seq{b_n}, and \Seq{c_n}, respectively. 
    
    Consider the $k$-simplex $S=\CH_\Omega(v^0,\dots,v^k)$ in $\overline{\Omega}$ spanned by $v^0,\dots,v^k$. Note that we allow for the possibility that some of these $v^i$-s are equal and that $S$ may be a degenerate $k$-simplex in $\overline{\Omega}$. Let $F^i\coloneqq\CH_\Omega(v^0,\dots,\wh{v^i},\dots,v^k)$ denote the facet of $S$ opposite to $v^i$.
 
    We now claim that $S$ is a $k$-simplex in $\Omega$, i.e. $S \cap \Omega \neq \emptyset$. 
    To prove this claim, we observe that $S_n\cap \overline{\ngh{K}{\delta}} \neq \emptyset$ for all $n$, as $c_n \in K$ and $\hil(c_n,S_n) < \delta$. Then \cref{cor: compactness of simplices in loc hausdorff top} implies that $S \cap \Omega\neq \emptyset$ and $S_n\cap \Omega \to S \cap \Omega$ in the local Hausdorff topology, possibly up to passing to a subsequence. This proves the claim.

     Moreover, note that $a,b \in \partial \Omega$ as $\hil(a_n,c_n)=\hil(b_n,c_n) \geq \frac{n}{2}$ and $c_n\to c \in \Omega$. Further, note that $c\in (a,b) =[a,b] \cap \Omega$.  Let $A:=F_{\Omega}(a) \subset \partial \Omega$ and $B:=F_{\Omega}(b) \subseteq\partial\Omega$ denote the open faces of $\Omega$ containing $a$ and $b$, respectively. Since $(a,b) \subset \Omega$, \cref{fact: behaviour of faces of Omega} implies the following:

    \begin{observation}
    \label{newobs:how_do_faces_A_B_intersect}
        $A\cap \overline{B}=B\cap \overline{A}=\emptyset$ and $\overline{A}\cap \overline{B} \subset \partial A \cap \partial B$.
    \end{observation}
     
    We now record the indices of the vertices of $S$ that are contained in $\overline{A}$ and $\overline{B}$ respectively via
    $$
    I\coloneqq\{i\in\{0,\dots,k\}\mid v^i\in\overline{A}\}
    \quad\text{and}\quad
    J\coloneqq\{i\in\{0,\dots,k\}\mid v^i\in\overline{B}\}.
    $$

    We will now prove a series of claims (Claim \ref{newclaim:about_I} through \ref{claim:ab_contained_in_nbd_k}) characterizing several properties of the simplex $S$ with the final goal being \cref{neweqn:ab_in_grand_intersection_for_S}. We will then finish the proof by a contradiction. This contradiction will arise from upgrading the containment in \cref{neweqn:ab_in_grand_intersection_for_S} to a containment for the $k$-simplices $S_n$ (for $n$ sufficiently large) and then playing it off against the $(\delta,D)$-genericity of the vertices of $S_n$.

    We will now begin establishing our claims about $S$. We will heavily use the following notation: for any $L \subset \{0,\dots,k\}$, define $$S_L:=\CH_{\Omega}(\{v^j:j\in L\}).$$ So, in particular, $S_I:=\CH_{\Omega}(\{v^i: i\in I \}) \subset S \cap \overline{A}$  and  $S_J:=\CH_{\Omega}(\{v^i : i\in J\}) \subset S \cap \overline{B}.$
    Recall that $d_A$ (resp. $d_B$) denotes the Hilbert metric on the properly convex domain $A$ (resp. $B$). 
    \begin{claim}
    \label{newclaim:about_I}
    $d_A(a,S\cap A)\le\delta$, $S \cap A\neq \emptyset$, $I$ is non-empty, and $S\cap\overline{A}=S_I$.
      Similarly, $d_B(b,S\cap B)\le\delta$, $S \cap B\neq \emptyset$, $J$  is non-empty, and $S\cap \overline{B} =S_J$. 
    \end{claim}

    \begin{proof}
    We will only prove the part for $a$ since the proofs for $b$ are similar. 
    As $S_n\to S$, $a_n\to a$, and $\sup_{n\in \Nb}\hil(a_n,S_n)\le\delta$, \cref{cor:tracking_sets_using_points} implies that $F_{\Omega}(a) \cap S=A \cap S \neq \emptyset$ and $d_{A}(a,A\cap S)\le\delta$. 

    Let $s'\in A\cap S$ be arbitrary. 
    As $v^0,\dots,v^k$ are the vertices of $S$, there is a (non-empty) minimal set $M'\subset \{0,\dots,k\}$ such that $s' \in \relint (S_{M'})$. As $s' \in F_{\Omega}(a)$, \cref{lem:relint_and_faces} implies that $\relint (S_{M'}) \subset F_{\Omega}(a)$. Then  $S_{M'} \subset \overline{F_{\Omega}(a)}=\overline{A}$ and, in particular, $v^i \in \overline{A}$ for all $i \in M'$. So $I$ is non-empty as $ \emptyset \neq M' \subset I$. Moreover, as $S_{M'}\subset S_I$ and $s'$ was arbitrary, $S\cap A \subset S_I$. Then $S \cap  \overline{A}\subset S_I$. Thus $S_I = S \cap \overline{A}$. 
    \end{proof}

    Recall that $F^i$ denotes the facet  opposite to vertex $v^i$ in the $k$-simplex $S$. We show that each facet $F^i$ intersects $A$ (resp. $B$) in a $\delta$-neighborhood of $a$ (resp. $b$).
     \begin{claim}
     \label{newclaim:Fi_intersects_A_B}
         For any $i\in \{0,\dots,k\}$, $d_A(a,F^i\cap A)\le\delta$ and  $d_B(b,F^i\cap B)\le\delta$.
     \end{claim}
     \begin{proof}Note that $F^i(S_n) \to F^i$, $a_n\to a$, and $\sup_{n\in \Nb}\hil(a_n,F^i(S_n))\le\delta$ (as $a_n \in C_{\delta}(S_n)$). Then \cref{cor:tracking_sets_using_points} implies that  $d_A(a,F^i\cap A)\le\delta$. The proof for $b$ is similar. 
     \end{proof}

We now show that each facet of $S_I$ (resp. $S_J$) intersects $A$ (resp. $B$). 
\begin{claim}
\label{newclaim:A_B_intersect_each_face_of_S}
    For any $i \in I$, $S_{I \setminus\{i\}} \cap A$ is non-empty and $d_A(a,S_{I\setminus\{i\}}\cap A) \le \delta$. For any $j \in J$, $S_{J\setminus\{j\}} \cap B \neq \emptyset$ and  $d_B(b,S_{J\setminus\{j\}}\cap B) \le \delta$.
\end{claim}
\begin{proof}
    By \cref{newclaim:about_I} $S\cap\overline{A}=S_I$. So $S\cap A =S_I \cap A$. Then, for any $i\in I$ 
    $$
    F^i\cap A=F^i\cap S \cap A=F^i\cap S_I\cap A=S_{I\setminus\{i\}}\cap A.
    $$
    The result then follows from \cref{newclaim:Fi_intersects_A_B}. The proof for $B$ is similar.
\end{proof}

Next we prove that both $(I\setminus J)$ and $(J\setminus I)$ have sufficiently many elements.

\begin{claim}
        \label{newclaim:size_of_I_J}
         $|I \setminus J| \geq 2$ and $|J \setminus I|\geq 2$. 
    \end{claim}
     \begin{proof} 
     
     We first show that $I \not\subset J$ and $J \not \subset I$. Suppose, if possible, that $I \subset J$. Then $S_I \subset S_J$. 
     Then Claim \ref{newclaim:about_I} implies 
     $S\cap \overline{A} =  S_I \subset S_J=S \cap \overline{B}.$ 
     So $S_I \subset S \cap \overline{A} \cap \overline{B} \subset \partial A \cap \partial B$ (see \cref{newobs:how_do_faces_A_B_intersect}). That is, $S_I \subset \partial A$. 
     But $S\cap \overline{A}=S_I\subset \partial A$ implies that $S\cap A=\emptyset$. This contradicts \cref{newclaim:about_I}. Hence $I\not \subset J$. By a similar reasoning, $J \not \subset I$. 

     Now we show that $|I\setminus J|>1$ and $|J\setminus I|>1$. Suppose, if possible, that $I\setminus J=\{i'_1\}$. Then $S_{I\setminus\{i'_1\}} \subset S_J=S\cap \overline{B}$. By \cref{newclaim:A_B_intersect_each_face_of_S}, we can find $s'' \in S_{I\setminus\{i'_1\}} \cap A$. Then $s'' \in A \cap \overline{B}=\emptyset$, a contradiction (see \cref{newobs:how_do_faces_A_B_intersect}). Thus $|I\setminus J|>1$. The reasoning for $|J\setminus I|>1$ is similar.  
     \end{proof}

    Set $\Oc:=\{(i,j): i \in I\setminus J, j \in J \setminus I\}.$ By \cref{newclaim:size_of_I_J} above, $\Oc$ is non-empty. For any $(i,j) \in \Oc$, we set $$I_{ij}:=(I \cup J)\setminus \{i,j\}.$$ Note that \cref{newclaim:size_of_I_J} implies that $I_{ij}\neq \emptyset$ so that $S_{I_{ij}}$ is well-defined.

    For any $\delta>0$ and $C\subset\Omega$, denote by $\clngh{C}{\delta}$ the closed $\delta$-neighborhood of $C$, i.e.
    $$\clngh{C}{\delta}=\{p\in\Omega\mid \hil(p,C)\le\delta\}.$$
    
        \begin{claim}\label{newclaim:ab_contained_in_nbd_12}For any $(i,j) \in \Oc$, $S_{I_{ij}}\cap \Omega \neq \emptyset$ and 
        $(a,b)\ \subseteq\clngh{S_{I_{ij}}\cap\Omega}{\delta}.$   
        \end{claim}
        \begin{proof} Let $(i_1,j_1)\in \Oc$. For the sake of brevity, we set $I_{\ell_1}:=I_{i_1j_1}$ in this proof.  
        By \cref{newclaim:A_B_intersect_each_face_of_S}, we can find $a' \in F^{i_1}\cap A$ such that $d_A(a,a') \le\delta$. Note that $a'\in S_{I\setminus\{i_1\}} \subset S_{I_{\ell_1}}$. Similarly, there exists $b' \in F^{j_1}\cap B \subset S_{I_{\ell_1}} \cap B$ such that $d_{B}(b,b')\le\delta$.  Thus, $[a',b'] \subset S_{I_{\ell_1}}.$ Then, by \cref{fact: maximum principle for extended distance}, $$d_{\overline{\Omega}}^H((a,b),(a',b')) =\max\{d_A(a,a'),d_{B}(b,b') \}\le \delta.$$ 

        As $(a,b)\subset \Omega$, this implies that $(a',b')\subset \Omega$. Thus, $(a', b')\subset S_{I_{\ell_1}} \cap \Omega$ and $S_{I_{\ell_1}} \cap \Omega$ is non-empty. Finally, 
        \begin{equation*}
            (a,b) \subset \clngh{(a',b')}{\delta} \subset \clngh{S_{I_{\ell_1}}}{\delta}\cap \Omega). \qedhere
        \end{equation*}
        \end{proof}

        For each $p\in I\cap J$, we set $$I_p:=(I\cup J)\setminus\{p\}$$ and show that $S_{I_p}$ satisfies a similar property as above. 
        \begin{claim}
        \label{newclaim:ab_contained_in_nbd_k}
         $(a,b)\ \subseteq\clngh{S_{I_p}\cap\Omega}{\delta}$ for each $p\in I\cap J$.
        \end{claim}
        \begin{proof}
        The proof is similar as above. Fix $p\in I \cap J$. By \cref{newclaim:A_B_intersect_each_face_of_S}, we can find $a_p''\in S_{I\setminus\{p\}}\cap A \subset  S_{I_p} \cap A$ with $d_{A}(a,a_p'')\le\delta$ and $b_p''\in S_{J\setminus\{p\}}\cap A \subset S_{I_p} \cap B$ with $d_{B}(b,b_p'')\le\delta$. Then, $[a_p'',b_p''] \subset S_{I_p}$. By \cref{fact: maximum principle for extended distance},
        $d_{\overline{\Omega}}^H((a,b),(a_p'',b_p'')) =\max\{d_A(a,a'),d_{B}(b,b') \}\le \delta.$ As $(a,b) \subset \Omega$, this implies that $(a_p'',b_p'')\subset \Omega$. Hence, 
        \begin{equation*}
        (a,b) \subset \clngh{S_{I_p}\cap \Omega}{\delta}.\qedhere
        \end{equation*}
        \end{proof}

            Putting \cref{newclaim:ab_contained_in_nbd_12} and \cref{newclaim:ab_contained_in_nbd_k} together, 
            \begin{equation}
            \label{neweqn:ab_in_grand_intersection_for_S}
                (a,b) \subset  \left( \bigcap_{(i,j)\in \Oc} \clngh{S_{I_{ij}}\cap\Omega}{\delta} \right) \bigcap \left( \bigcap_{p\in I\cap J} \clngh{S_{I_p}\cap\Omega}{\delta}\right).
            \end{equation}

     Recall that as $n\to \infty$, $S_n \to S$. Now we will show that $(S_n)_{I_{ij}} \cap \Omega$ and $(S_n)_{I_p}\cap \Omega$ shows similar behavior as above, i.e. the grand intersection of all of their open $2\delta$-neighborhoods contain a large chunk of $(a,b)$. This would contradict that the vertices of each $S_n$ are $(2\delta,D)$-generic. We formalize this below.

Recall that $D>0$ is fixed. Pick $x,y\in (a,b)$ such that $\hil(x,y)>D$. Further recall that $v^0_n, \dots, v^k_n$ are the vertices of $S_n$. For each $n\in \Nb$, consider the following subsets of vertices of $S_n$ given by $$V^n_{I_{ij}}:=\{v^l_n: l \in I_{ij}\} \text{ and } V^n_{I_p}:=\{v^l_n: l \in I_p\}$$ for each $(i,j)\in \Oc$ and $p\in I\cap J$. As $n \to \infty$, $\CH_{\Omega}(V_*^n)\to S_{I_*}$ and $\CH_{\Omega}(V_*^n)\cap ~\Omega \to S_{I_*} \cap ~\Omega$ for $* \in \{ ij: (i,j)\in \Oc\} \cup \{ p: p \in I \cap J\}$ (see \cref{lem: convergence of simplices}). Then, for any $\varepsilon>0$, \cref{neweqn:ab_in_grand_intersection_for_S} implies that for $n$ sufficiently large, the following holds;

    \begin{equation}
    \label{neweqn:change_to_xy_and_Sn}
    (x,y)\ \subseteq \left( \bigcap_{(i,j)\in \Oc} \clngh{\CH_\Omega(V^n_{I_{ij}})\cap \Omega}{\delta+\varepsilon} \right)  \bigcap \left( \bigcap_{p\in I\cap J}\clngh{\CH_\Omega(V^n_{p})\cap \Omega}{\delta+\varepsilon}  \right).
    \end{equation}

    Moreover note that 
    \begin{equation}
    \label{neweqn:I_ij_and_I_p}
        \left(\bigcap_{(i,j)\in \Oc} I_{ij} \right)  \bigcap \left( \bigcap_{p\in I\cap J} I_p \right) =\emptyset.
    \end{equation} 
    Indeed, $\left( \bigcap_{p\in I\cap J} I_p \right)=(I\setminus J) \sqcup (J\setminus I)$. But $(I\setminus J) \bigcap \left(\bigcap_{(i,j)\in \Oc} I_{ij} \right)=\emptyset$ because, given any $i'\in I \setminus J$, there exists $j'\in J\setminus I$ such that $(i',j')\in \Oc$ (see \cref{newclaim:size_of_I_J}) and hence $i'\not\in I_{i'j'}$. By a similar reasoning, $(J \setminus I) \bigcap \left(\bigcap_{(i,j)\in \Oc} I_{ij} \right)=\emptyset$. 

    Now, if we pick $\varepsilon>0$ such that $\delta+\varepsilon<2\delta$, we get from \cref{neweqn:change_to_xy_and_Sn} that 

    \begin{equation}
    \label{neweqn:change_to_xy_and_Sn open}
    (x,y)\ \subseteq \left( \bigcap_{(i,j)\in \Oc} \ngh{\CH_\Omega(V^n_{I_{ij}})\cap \Omega}{2\delta} \right)  \bigcap \left( \bigcap_{p\in I\cap J}\ngh{\CH_\Omega(V^n_{p})\cap \Omega}{2\delta}  \right).
    \end{equation}
    
    Finally, since we assumed that the vertices of each $S_n$ forms a $(2\delta,D)$-generic $(k+1)$-tuple, \cref{neweqn:change_to_xy_and_Sn open} and \cref{neweqn:I_ij_and_I_p} imply that $|I_{ij}|=|I_p|=k$ for each $(i,j)\in \Oc$ and $p\in (I \cap J)$ (cf. \cref{def: generic_points_in_Hilbert_new}). But $|I_{ij}| \leq k-1$ for each $(i,j)$ in the non-empty set $\Oc$. This is a contradiction and it finishes the proof of \cref{newlem: slim simplices have bounded centroid}. \qed


\subsection{A uniform version of \cref{newlem: slim simplices have bounded centroid} for $\delta$ varying in a compact interval}

Recall that $\omegaGen{k+1}$ is the set of $(\delta,D)$-generic tuples. We now strengthen our previous \cref{newlem: slim simplices have bounded centroid}.

\begin{lemma}\label{lem: slim simplices have bounded centroid_uniform_version}
Suppose $\Omega \subset \RP$ is a quasi-homogeneous properly convex domain and $k\geq 2$. 
Fix $D>0$ and $0<\delta_1 \leq \delta'_1<\infty$.  Then there exists $R=R(\delta_1,\delta'_1,D,k)$ such that: for any $\delta\in[\delta_1,\delta'_1]$ and any compact $k$-simplex $S$ in $\Omega$ with $\vrtx(S) \in \Omega^{(k+1)}_{(2\delta,D)}$,
$$\diam (C_\delta(S))<R.$$ 
\end{lemma}
\begin{proof}
    In this proof, we will closely follow the proof of \cref{newlem: slim simplices have bounded centroid}. Hence, in some parts, we will hide the details are ask the reader to look at the relevant details in the proof of \cref{newlem: slim simplices have bounded centroid}.
    
    Assume by contradiction that for any $n\in\mathbb{N}$, there exist $\delta_n\in[\delta_1,\delta'_1]$ and a  compact simplex $S_n$ in $\Omega$, with $(2\delta_n,D)$-generic vertices, such that
    \begin{equation*}
        \diam(C_{\delta_n}(S_n))\ge n.
    \end{equation*}
    Up to passing to a subsequence, we may assume that $\delta_n \to \delta$ for some $\delta \in [\delta_1,\delta_1']$. Without loss of generality, we may assume that $\delta_n < \delta$ up to truncating finitely many elements from the beginning of the sequence $\seq{\delta_n}$. Indeed, suppose on the contrary that $\delta_n \geq \delta$ infinitely often. Since  $\Omega^{k+1}_{(2\delta_n,D)}\subset \Omega^{k+1}_{(2\delta,D)}$, \cref{newlem: slim simplices have bounded centroid} implies that $n \leq \diam(C_{\delta_n}(S_n)) \leq R(\delta,D,k)$. This is a contradiction since only finitely many such $n$ exist. 

    By Lemma \ref{lem: slim triangles have bounded center}, we must have $k\ge3$.
    Moreover, we have that $C_{\delta_n}(S_n)\neq\emptyset$, for any $n>0$.
    Let $v_n^0,\dots,v_n^k$ be the vertices of $S_n$.
    For each $i=0,\dots,k$, denote by $F^i_n$ the face of $S_n$ opposite to $v_n^i$. As $\diam(C_{\delta_n}(S_n)) \geq n$,  there exist $a_n,b_n\in C_{\delta_n}(S_n)$ such that  
    \begin{equation*}
        \hil(a_n,b_n)\ge n.
    \end{equation*}
    Therefore, if we consider the midpoint $c_n$ of the projective segment $[a_n,b_n]$, then, by \cref{obs: convexity of centroids}, $[a_n,b_n] \in C_{\delta_n}(S_n)$, $c_n \in C_{\delta_n}(S_n)$, and $\hil(c_n,S_n)<\delta_n<\delta.$

    Since $\Omega$ is quasi-homogeneous, we may assume that there exists a compact set $K\subset \Omega$ such that $c_n\in K$ for all $n\in\mathbb{N}$. Up to passing to subsequences, we can assume that there exist $v^0,\dots,v^k,a,b,c\in\overline{\Omega}$ such that $\lim_{n\to\infty}v_n^i=v^i$ for $i=0,\dots,k$  and $a,b,c$ are the limits of \Seq{a_n},\Seq{b_n}, and \Seq{c_n}, respectively. 
    
    Consider the $k$-simplex $S=\CH_\Omega(v^0,\dots,v^k)$ in $\overline{\Omega}$ spanned by $v^0,\dots,v^k$. Note that $S$ may be a degenerate $k$-simplex in $\overline{\Omega}$. Let $F^i\coloneqq\CH_\Omega(v^0,\dots,\wh{v^i},\dots,v^k)$ denote the facet of $S$ opposite to $v^i$.
 
    We now claim that $S$ is a $k$-simplex in $\Omega$, i.e. $S \cap \Omega \neq \emptyset$. 
    To prove this claim, we observe that $S_n\cap \overline{\ngh{K}{\delta}} \neq \emptyset$ for all $n$, as $c_n \in K$ and $\hil(c_n,S_n) < \delta$. Then \cref{cor: compactness of simplices in loc hausdorff top} implies that $S \cap \Omega\neq \emptyset$ and $S_n\cap \Omega \to S \cap \Omega$ in the local Hausdorff topology, possibly up to passing to a subsequence. This proves the claim. 
        
    Moreover, note that $a,b \in \partial \Omega$ as $\hil(a_n,c_n)=\hil(b_n,c_n) \geq \frac{n}{2}$ and $c_n\to c \in \Omega$. Further, note that $c\in (a,b) =[a,b] \cap \Omega$.  Let $A:=F_{\Omega}(a) \subset \partial \Omega$ and $B:=F_{\Omega}(b) \subseteq\partial\Omega$ denote the open faces of $\Omega$ containing $a$ and $b$, respectively. Since $(a,b) \subset \Omega$, \cref{fact: behaviour of faces of Omega} implies the following:

    \begin{observation}
    \label{obs:how_do_faces_A_B_intersect}
        $A\cap \overline{B}=B\cap \overline{A}=\emptyset$ and $\overline{A}\cap \overline{B} \subset \partial A \cap \partial B$.
    \end{observation}
     
    We now record the indices of the vertices of $S$ that are contained in $\overline{A}$ and $\overline{B}$ respectively via
    $$
    I\coloneqq\{i\in\{0,\dots,k\}\mid v^i\in\overline{A}\}
    \quad\text{and}\quad
    J\coloneqq\{i\in\{0,\dots,k\}\mid v^i\in\overline{B}\}.
    $$

    We will use the same notation as in the proof of \cref{newlem: slim simplices have bounded centroid}: for any $L \subset \{0,\dots,k\}$, define $$S_L:=\CH_{\Omega}(\{v^j:j\in L\}).$$ So, in particular, $S_I:=\CH_{\Omega}(\{v^i: i\in I \}) \subset S \cap \overline{A}$  and  $S_J:=\CH_{\Omega}(\{v^i : i\in J\}) \subset S \cap \overline{B}.$

    Recall that $d_A$ (resp. $d_B$) denotes the Hilbert metric on the properly convex domain $A$ (resp. $B$). 

    We state six claims (Claim \ref{claim:about_I} through \ref{claim:ab_contained_in_nbd_k}) which are exactly the same as the five claims in the proof of \cref{newlem: slim simplices have bounded centroid}. Once we observe that $\sup_{n\in \Nb}\delta_n=\delta$, we can repeat the proofs of Claims \ref{newclaim:about_I}  through \ref{newclaim:ab_contained_in_nbd_k} verbatim and get the proofs of Claims \ref{claim:about_I} through \ref{claim:ab_contained_in_nbd_k}. Hence, we hide the proofs of these claims here. 
    \begin{claim}
    \label{claim:about_I}
    $d_A(a,S\cap A)\le\delta$, $S \cap A\neq \emptyset$, $I$ is non-empty, and $S\cap\overline{A}=S_I$. Similarly, $d_B(b,S\cap B)\le\delta$, $S \cap B\neq \emptyset$, $J$  is non-empty, and $S\cap \overline{B} =S_J$. 
    \end{claim}
    
     \begin{claim}
     \label{claim:Fi_intersects_A_B}
         For any $i\in \{0,\dots,k\}$, $d_A(a,F^i\cap A)\le\delta$ and  $d_B(b,F^i\cap B)\le\delta$.
     \end{claim}

\begin{claim}
\label{claim:A_B_intersect_each_face_of_S}
    For any $i \in I$, $S_{I \setminus\{i\}} \cap A$ is non-empty and $d_A(a,S_{I\setminus\{i\}}\cap A) \le \delta$. For any $j \in J$, $S_{J\setminus\{j\}} \cap B \neq \emptyset$ and  $d_B(b,S_{J\setminus\{j\}}\cap B) \le \delta$.
\end{claim}

    \begin{claim}
        \label{claim:size_of_I_J}
         $|I \setminus J| \geq 2$ and $|J \setminus I|\geq 2$. 
    \end{claim}

    Set $\Oc:=\{(i,j): i \in I\setminus J, j \in J \setminus I\}.$ By \cref{claim:size_of_I_J} above, $\Oc$ is non-empty. For any $(i,j) \in \Oc$ and $p\in I\cap J$, we set $$I_{ij}:=(I \cup J)\setminus \{i,j\}\quad\text{and}\quad I_p:=(I\cup J)\setminus\{p\}.$$
    Note that \cref{claim:size_of_I_J} implies that $I_{ij}\neq \emptyset$ so that $S_{I_{ij}}$ is well-defined.
    As above, for any $\delta>0$ and $C\subset\Omega$, denote by $\clngh{C}{\delta}$ the closed $\delta$-neighborhood of $C$, i.e. $\clngh{C}{\delta}=\{p\in\Omega\mid \hil(p,C)\le\delta\}$.

        \begin{claim}\label{claim:ab_contained_in_nbd_12} 
        For any $(i,j) \in \Oc$, $S_{I_{ij}}\cap \Omega \neq \emptyset$ and 
        $(a,b)\ \subseteq\clngh{S_{I_{ij}}\cap\Omega}{\delta}.$   
        \end{claim}

        \begin{claim}
        \label{claim:ab_contained_in_nbd_k}
         $(a,b)\ \subseteq\clngh{S_{I_p}\cap\Omega}{\delta}$ for each $p\in I\cap J$.
        \end{claim}

            Putting \cref{claim:ab_contained_in_nbd_12} and \cref{claim:ab_contained_in_nbd_k} together, 
            \begin{equation}
            \label{eqn:ab_in_grand_intersection_for_S}
                (a,b) \subset  \left( \bigcap_{(i,j)\in \Oc} \clngh{S_{I_{ij}}\cap\Omega}{\delta} \right) \bigcap \left( \bigcap_{p\in I\cap J} \clngh{S_{I_p}\cap\Omega}{\delta}\right).
            \end{equation}

    Recall that as $n\to \infty$, $S_n \to S$. Now we will show that $(S_n)_{I_{ij}} \cap \Omega$ and $(S_n)_{I_p}\cap \Omega$ shows similar behavior as above for $n$ sufficiently large, contradicting the $(2\delta_n,D)$-genericity of the vertices of each $S_n$. We formalize this below.

    Recall that $D>0$ is fixed. Pick $x,y\in (a,b)$ such that $\hil(x,y)>D$.  
    For each $n\in \Nb$, consider the following subsets of vertices of $S_n$ given by $$V^n_{I_{ij}}:=\{v^l_n: l \in I_{ij}\} \text{ and } V^n_{I_p}:=\{v^l_n: l \in I_p\}$$ for each $(i,j)\in \Oc$ and $p\in I\cap J$. As $n \to \infty$, $\CH_{\Omega}(V_*^n)\to S_{I_*}$ and $\CH_{\Omega}(V_*^n)\cap ~\Omega \to S_{I_*} \cap ~\Omega$ for $* \in \{ ij: (i,j)\in \Oc\} \cup \{ p: p \in I \cap J\}$ (see \cref{lem: convergence of simplices}). 
    
    Hence, for any $\varepsilon>0$, \cref{eqn:ab_in_grand_intersection_for_S} implies that for $n$ sufficiently large, we have:
    \begin{equation}
    \label{eqn:change_to_xy_and_Sn}
    (x,y)\ \subseteq \left( \bigcap_{(i,j)\in \Oc} \clngh{\CH_\Omega(V^n_{I_{ij}})\cap \Omega}{\delta+\varepsilon} \right)  \bigcap \left( \bigcap_{p\in I\cap J}\clngh{\CH_\Omega(V^n_{p})\cap \Omega}{\delta+\varepsilon}  \right).
    \end{equation}

    Moreover, by the same reasoning as in the proof of \cref{neweqn:I_ij_and_I_p} inside the proof of \cref{newlem: slim simplices have bounded centroid}, we have
    \begin{equation}
    \label{eqn:I_ij_and_I_p}
        \left(\bigcap_{(i,j)\in \Oc} I_{ij} \right)  \bigcap \left( \bigcap_{p\in I\cap J} I_p \right) =\emptyset.
    \end{equation}

    Now pick any $0<\varepsilon<\frac{\delta}{3}$. Since $\delta_n\to\delta$ and $\delta_n<\delta$ for any $n\in\mathbb{N}$, we have $\delta-\delta_n<\varepsilon$, for any $n$ sufficiently large. For such $\varepsilon$, we have $2\delta_n>2\delta-2\varepsilon>\delta+\varepsilon$.
    Therefore, from \cref{eqn:change_to_xy_and_Sn}, we get

    \begin{equation}
    \label{eqn:change_to_xy_and_Sn open}
    (x,y)\ \subseteq \left( \bigcap_{(i,j)\in \Oc} \ngh{\CH_\Omega(V^n_{I_{ij}})\cap \Omega}{2\delta_n} \right)  \bigcap \left( \bigcap_{p\in I\cap J}\ngh{\CH_\Omega(V^n_{p})\cap \Omega}{2\delta_n}  \right),
    \end{equation}
    for any sufficiently large $n\in\mathbb{N}$.

    
    Finally, since we assumed that the vertices of each $S_n$ forms a $(2\delta_n,D)$-generic $(k+1)$-tuple, \cref{eqn:I_ij_and_I_p} and \cref{eqn:change_to_xy_and_Sn open} imply that $|I_{ij}|=|I_p|=k$ for each $(i,j)\in \Oc$ and $p\in (I \cap J)$. But $|I_{ij}| \leq k-1$ for each $(i,j)$ in the non-empty set $\Oc$ (cf. \cref{def: generic_points_in_Hilbert_new}). This is a contradiction and it finishes the proof of \cref{lem: slim simplices have bounded centroid_uniform_version}.
\end{proof}

\subsection{Proof of \cref{prop:CM_2_for_pcd}}
\label{sec:proof_of_CM_2_for_pcd}
    Let $r\ge\rF(\Omega)$.
     Let $\delta_0'>0$ as in \cref{thm:slimness_gives_centroid}.
    The non-emptiness of $C_{\frac{\delta'}{2}}(V)$ for any $V\in\Omega^{(r+1)}_{(\delta',D)}$ follows directly from \cref{thm:slimness_gives_centroid}.
    Since we will only be working with $(r+1)$ tuples, we use the following notation for this proof: we write $\Omega_{(\delta,D)}$ to denote the set of all $(\delta,D)$-generic $(r+1)$-tuples in $\Omega$.

    For any $\delta, D>\delta_0'$, we define
    $$
        \Psi(\delta,D)\coloneqq\inf\left\{M\ge0\mid\ \diam(C_{\frac{\delta'}{2}}(V))\le M\ \text{ for all} \ V\in \bigcup_{\delta' \in [\delta'_0,\delta]}\Omega_{(\delta',D)}\right\}.
    $$

    Indeed, each $C_{\delta'/2}(V)$ is non-empty because $\delta' \geq \delta'_0$ (\cref{thm:slimness_gives_centroid}). Moreover, since $[\delta_0',\delta]$ is compact,  \cref{lem: slim simplices have bounded centroid_uniform_version} implies that there is an $M \in [0,\infty)$ such that  $\Psi(\delta,D) \leq M$. For instance, $M=R(\delta'_0,\delta,D,r+1)$, where $R(\cdot)$ is as in \cref{lem: slim simplices have bounded centroid_uniform_version}, suffices.

    \begin{remark*}
        \emph{By \cref{obs: generic_increase_parameters}, $\Omega_{(\delta',D)}\subseteq\Omega_{(\delta'_0,D)}$, for any $\delta' \geq \delta_0'$. Thus, for any $\delta \geq \delta'_0$, we have the equality of sets: $\bigcup_{\delta' \in [\delta'_0,\delta]}\Omega_{(\delta',D)}=\Omega_{(\delta'_0,D)}.$ \\
        But in what follows, we will always write it as the union on the left  because, for any $(r+1)$-tuple $V$, we would like to remember all the $\Omega_{(\delta',D)}$-s that the $V$ lies in. 
        This information is used in defining $\Psi(\delta,D)$ above. Indeed, for $V\in \Omega_{(\delta',D)}$ for $\delta' \geq \delta'_0$, we will record $\diam(C_{\delta'/2}(V))$, which is at least as large as $\diam(C_{\delta'_0/2}(V))$. That is, the extra information about the $(\delta',D)$-genericity of $V$ lets us find a coarse centroid $C_{\delta'/2}(V)$ with larger diameter than $C_{\delta'_0/2}(V)$, although the underlying tuple $V$ is the same.}
    \end{remark*}

    We will show that $\Psi:[\delta_0',\infty)\times[\delta_0',\infty)\to[0,\infty)$ is non-decreasing in both coordinates.

    Fix some $\delta\ge\delta_0'$. For any $D\ge D'\ge \delta_0'$, \cref{obs: generic_increase_parameters} implies that $\Omega_{(\delta',D')} \subset \Omega_{(\delta',D)}.$
    This implies that
     \begin{align*}
        \{M\ge0\mid \diam(C_{\frac{\delta'}{2}}(V))\le M \ \text{ for all}\   V\in \bigcup_{\delta'\in[\delta'_o,\delta]}\Omega_{(\delta',D)} \}\\
        \subseteq\{M\ge0\mid \diam(C_{\frac{\delta'}{2}}(V))\le M\ \text{ for all}\   V\in \bigcup_{\delta'\in[\delta'_o,\delta]} \Omega_{(\delta',D')} \}
    \end{align*}
    Consequently, $\Psi(\delta,D')\le\Psi(\delta,D)$.

    On the other hand, fix some $D\ge \delta_0'$, and $\delta\ge \wh{\delta}\ge\delta_0'$. By \cref{obs: generic_increase_parameters}, $\Omega_{(\delta',D)} \subset \Omega_{(\wh{\delta},D)}$ for any $\delta' \geq \wh{\delta}.$  
    Then, we have,
    \begin{align*}
        \{M\ge0\mid \diam(C_{\frac{\delta'}{2}}(V))\le M \ \text{ for all}\   V\in  \bigcup_{\delta'\in[\delta'_o,\delta]}\Omega_{(\delta',D)}  \}\\
        \subseteq\{M\ge0\mid \diam(C_{\frac{\delta'}{2}}(V))\le M\ \text{ for all}\  V\in  \bigcup_{\delta'\in[\delta'_o,\wh{\delta}]}\Omega_{(\delta',D)}  \}
    \end{align*} 
    
    To prove this, let $M \geq 0$ be such that $\diam(C_{\frac{\delta'}{2}}(V))\le M$ for all $V \in \bigcup_{\delta' \in [\delta'_0,\delta]}\Omega_{(\delta',D)}.$ Now as  $ \bigcup_{\delta' \in [\delta'_0,\wh{\delta}]}\Omega_{(\delta',D)} \subset \bigcup_{\delta' \in [\delta'_0,\delta]}\Omega_{(\delta',D)}$, it follows that  $\diam(C_{\frac{\delta'}{2}}(V))\le M$ for any $\delta'\in [\delta'_0,\wh{\delta}]$. Hence we have this containment. 

    Consequently, we have $\Psi(\wh{\delta},D)\le\Psi(\delta,D)$. \qed

\section{Coarse $r$-median metric spaces: General definition}

\subsection{Coarse fillings}

\begin{definition} 
Let $(X,d)$ be a metric space. Fix $r\ge1$.
A \textbf{coarse $r$-filling} on $X$ is a pair $(F,\ctrlf)$, where $F$ is a family of $r+1$ maps $F=\{ F_i:X^{i+1} \to \Pc(X)\}_{i=0,\dots,r}$ and $\ctrlf:\mathbb{R}_{\ge0} \to \mathbb{R}_{\ge0}$ is a non-decreasing function such that for any $k \in \{0,\dots,r\}$ and any $a_0,\dots,a_k \in X$, the following properties hold,
\begin{enumerate}
    \item [(F0)] (point filling) $d^{\Haus}(a_0, F_0(a_0))\leq \ctrlf(0)$;
    \item [(F1)] (face containment) for any $i=0,\dots,k$ 
    $$
    F_{k-1}(a_0,\dots,\wh{a_i},\dots,a_k) \subseteq\ngh{F_k(a_0,\dots,a_k)}{\ctrlf(0)};
    $$
    \item [(F2)] (symmetry) for any permutation $\sigma$ of $(k+1)$-symbols, $$d^{\Haus}(F_k(a_0,\dots,a_k),F_k(a_{\sigma(0)},\dots, a_{\sigma(k)}))\le \ctrlf(0);$$  
    \item [(F3)] (localization) for any $k=0,\dots,r-1$,
    $$
    F_{k+1}(a_0,\dots,a_k,a_k) \subseteq\ngh{F_k(a_0,\dots,a_k)}{\ctrlf(0)};
    $$
    \item [(F4)] (weak convexity)
    if $x\in \ngh{F_k(a_0,\dots,a_k)}{R}$, then $$F_k(a_0,\dots,a_{j-1},x,a_{j+1},\dots,a_k) \subseteq \ngh{F_k(a_0,\dots,a_k)}{\ctrlf(R)}$$ for any $j=0,\dots,k$.
    \label{defn:coarse_filling_weak}
\end{enumerate}
We will call $\ctrlf$ the \textbf{control function} of the coarse $r$-filling on $(X,d)$.
Moreover, we say that $(F,\ctrlf)$ is an \textbf{$r$-filling} if the control function $\ctrlf$ satisfies $\ctrlf(0)=0.$
\end{definition}

This seemingly elementary definition implies several interesting properties for the filling maps that we will prove in \cref{sec:prop_of_coarse_filling} below. Among them, we state a crucial right now, while deferring its proof until the very end of  \cref{sec:prop_of_coarse_filling}. This result says that for a coarse filling, we can often choose the control function to be affine. Recall that a function $f:[0,\infty) \to [0,\infty)$ is said to be affine if $f(x)=ax+b$. Note that $f$ is necessarily non-decreasing since $f$ maps $\Rb_{\geq 0}$ to itself and hence, $a \geq 0$.

\begin{restatable}{proposition}{AffineCtrlf}
\label{prop: affine_control_function weak}
    Suppose $(X,d)$ is a geodesic metric space and $(F,\ctrlf)$ is a coarse $r$-filling on it. Then the function $\ctrlf$ can be replaced by a non-decreasing affine function.
\end{restatable}

Since the filling $(F,\ctrlf)$ is a coarse object, we need a notion for the equivalence between such objects. 

\begin{definition}
\label{defn:coarse_equiv_of_fillings}
    We say that two coarse $r$-fillings $(F,\ctrlf)$ and $(F',\ctrlf')$ on $(X,d)$ are \emph{ coarsely equivalent} 
    if there exists $C'\geq 0$ such that, for all $i=0,\dots,r$, we have 
    \begin{align*}
     \sup_{a_0,\dots,a_i \in X} d^{\Haus}\left( F_i(a_0,\dots,a_i),F'_i(a_0,\dots,a_i) \right) \leq C'.
    \end{align*}
\end{definition}
\begin{remark*}
    Let $(F,\ctrlf)$ and $(F',\ctrlf')$ be two equivalent coarse $r$-fillings on a metric space $(X,d)$. Let $C'\ge0$ as in \cref{defn:coarse_equiv_of_fillings}. Define the function $\ctrlf'':\mathbb{R}_{\ge0}\to\mathbb{R}_{\ge0}$ as $\ctrlf''(R+C')\coloneqq\ctrlf(R)+2C'$ for any $R\ge0$. Then, by the triangle inequality, the pair $(F',\ctrlf'')$ is a coarse $r$-filling.
\end{remark*}

We will use the following notation notion in \cref{defn:generic_point_for_coarse_filling_weak} below: if $I\subset \{0,\dots,k\}$, we define $$F_{|I|-1}(a_i: i\in I):=F_{p-1}(a_{j_1},\dots,a_{j_p}),$$ 
where $I=\{j_1,\dots,j_p\}$ and $j_1 < \dots < j_p$.

\subsection{Generic tuples}
The following definition is motivated by \cref{def: generic_points_in_Hilbert_new} -- the notion of generic points for properly convex domains. Intuitively, we want to capture configurations of points where combinatorially distant subsets of points have geometrically distant fillings.

\begin{definition}\label{def: generic_tuples_general}
    Let $(X,d)$ be a metric space with a coarse $r$-filling $(F,\ctrlf)$. Fix $\delta,D\geq 0$. 
    We say that an $(r+2)$-tuple $(a_0,\dots,a_{r+1})\in X^{r+1}$ is $\boldsymbol{(\delta,D)}$\textbf{-generic} (with respect to $(F,\ctrlf)$)
    if: for any family of non-empty proper subsets $I_1,\dots,I_q \subseteq \{0,\dots,r+1\}$ such that 
    \begin{equation*}
    \bigcap_{j=1}^q I_j=\emptyset\quad\text{and}\quad\diam\left(\bigcap_{j=1}^q\ngh{F_{\abs{I_j}-1}(a_i\mid i\in I_j)}{\delta}\right)> D
    \end{equation*}
    we must have $|I_1|=\dots=|I_q|=r+1$. 
    
    Otherwise, we say that the points $a_0,\dots,a_{r+1}$ are $\boldsymbol{(\delta,D)}$\textbf{-non-generic} (with respect to $(F,\ctrlf)$). 
    \label{defn:generic_point_for_coarse_filling_weak}
\end{definition}
When the filling is clear from context, we will drop the trailing reference to $(F,\ctrlf)$ in the definitions above.

\begin{notation} \label{notn:set_of_generic_pts}
    Suppose $(X,d)$ is a metric space, $r \in \Nb$, and $\delta,D \geq 0$. We define $$X^{(r+2)}_{(\delta,D)}:=\{(x_0,\dots,x_{r+1})\in X^{r+2} \mid (x_0,\dots,x_{r+1}) \text{ is } (\delta,D)\text{-generic}\}.$$ Then, obviously, $X^{(r+2)}\setminus X^{(r+2)}_{(\delta,D)}$ is the set of $(\delta,D)$-non-generic tuples.
\end{notation}

\begin{observation}
\label{obs: generic_increase_parameters}
    Fix $\delta,D \geq 0$. If $D'>D$, then $X^{(r+2)}_{(\delta,D)}\subset X^{(r+2)}_{(\delta,D')}$. If $\delta' > \delta$, then $X^{(r+2)}_{(\delta',D)}\subset X^{(r+2)}_{(\delta,D)}$.
\end{observation}

\begin{remark}\label{rmk:existence_generic_tuples}
    The set $X^{(r+2)}_{(\delta,D)}$ may be empty in general. However, if $D\ge2((r+2)\ctrlf(0)+\delta)$, this set is always non-empty. Indeed, for any fixed $o\in X$, we can consider the $(r+2)$-tuple $(o,\dots,o)\in X^{r+2}$. In this case, for any $h=0,\dots,r+1$, by applying (F1) multiple times, we get
    $$
    F_h(o,\dots,o)\subseteq\ngh{o}{(h+1)\ctrlf(0)}.
    $$
    Consequently,  $(o,\dots,o)\in X^{(r+2)}_{(\delta,2((r+2)\ctrlf(0)+\delta))}\subset X^{(r+2)}_{(\delta,D)}$, by \cref{obs: generic_increase_parameters}.
\end{remark}

\subsection{Definition of a coarse $r$-median structure on a metric space}

\begin{definition}\label{defn:coarse_r_median}
    Fix $r\ge1$. We say that a metric space $(X,d)$ has a \emph{\textbf{coarse $r$-median structure}} if there exists 
    \begin{itemize}
        \item a coarse $r$-filling $(F,\ctrlf)$ on $(X,d)$,
        \item a constant $\delta_0\ge0$, 
        \item a non-constant affine function\footnote{That is, there exist $a>0$ and $b\ge0$ such that $\lambda(\delta)=a(\delta-\delta_0)+b$ for any $\delta \in [\delta_0,\infty)$} $\lambda:[\delta_0,\infty) \to [0,\infty)$, and
        \item a function $\Psi:[\delta_0,\infty) \times [\delta_0,\infty) \to [0,\infty)$ that is non-decreasing in both entries,
        
    \end{itemize} such that the following properties hold:

    \begin{enumerate}
        \item[(CM1)] for any $a_0,\dots,a_{r+1} \in X$,  
        $$ \bigcap_{j=0}^{r+1} \ngh{F_r(a_0,\dots,\wh{a}_j,\dots, a_{r+1})}{\lambda(\delta_0)} \neq \emptyset, 
        $$ 
        \item[(CM2)] given any $\delta,D \geq \delta_0$, for any $(a_0,\dots, a_{r+1}) \in X^{(r+2)}_{(\delta,D)}$,
        \begin{align*}
           \diam\left( \bigcap_{j=0}^{r+1} \ngh{F_{r}(a_0,\dots,\wh{a}_j,\dots, a_{r+1})}{\lambda(\delta)}\right)\leq \Psi(\delta,D).
        \end{align*}
        
    \end{enumerate}

    We say that $(F,\ctrlf,\delta_0,\lambda, \Psi)$ are the \textbf{parameters} of the coarse $r$-median structure on $(X,d)$ associated to $(F,\ctrlf)$.

    Moreover, this coarse $r$-median structure will be called an \textbf{$r$-median structure} when $(F,\ctrlf)$ is an $r$-filling, $\delta_0=\Psi(0,0)=0$, and $\lambda$ is linear.
\end{definition}

If the functions $\ctrlf$ or $\Psi$ are affine, then the coarse $r$-median structure often has better properties; see for example \cref{prop:equiv_coarse_1_median} or \cref{cor:coarse_median_on_asymp_cone}. So we define:

\begin{definition}
\label{defn:coarse_r_median_with_affine_params}
    Let $(X,d)$ be a metric space that admits a coarse $r$-median with parameters $(F,\ctrlf,\delta_0,\lambda,\Psi)$. We say that:
    \begin{itemize}
        \item the \emph{coarse $r$-median is $\Psi$-affine} if $(R,D) \mapsto \Psi(R,D)$ is affine in both coordinates.
        \item the \emph{coarse $r$-median is $(\ctrlf,\Psi)$-affine} if $x \mapsto \ctrlf(x)$ is affine and $(R,D) \mapsto \Psi(R,D)$ is affine in both coordinates. 
    \end{itemize}
\end{definition}

\subsection{Coarse $r$-median map on a metric space}
We have a coarse median map, similarly as in \cref{sec: Coarse rF-median map}.
\begin{definition}\label{defn:coarse_r_median_map_general}
    Fix $r\in\mathbb{N}$. Let $(X,d)$ be a metric space endowed with a coarse $r$-median structure with parameters $\mathcal{CM}=(F,\ctrlf,\delta_0,\lambda,\Psi)$. Then, the \emph{\textbf{coarse $r$-median map} associated to $\mathcal{CM}$} is the function $\mu_{\mathcal{CM}}:X^{r+2}\to\Pc(X)\setminus\{X,\emptyset\}$ given by 
    $$
    \mu_{\mathcal{CM}}(x_0,\dots,x_{r+1})\coloneqq \bigcap_{j=0}^{r+1} \ngh{F_r(a_0,\dots,\wh{a}_j,\dots, a_{r+1})}{\lambda(\delta_0)}.
    $$
\end{definition}
\begin{remark} 
\label{rmk:relation_coarse_median_map_on_generic}
Recall the notation $X^{(r+2)}_{(\delta,D)}$ from \cref{notn:set_of_generic_pts}. 
\begin{enumerate}
    \item Condition (CM1) in \cref{defn:coarse_r_median} implies that $\mu_{\mathcal{CM}}$ is well-defined.
    \item Fix $D_0\coloneqq 2((r+2)\ctrlf(0)+\delta_0)$. By \cref{rmk:existence_generic_tuples}, the set $X^{(r+2)}_{(\delta_0,D_0)}$ is non-empty. Condition (CM2) and \cref{defn:coarse_centroid} imply that the restriction of $\mu_{\mathcal{CM}}$ on $X^{(r+2)}_{(\delta_0,D_0)}$ yields a map that takes values (coarsely) in $X$. 
    More precisely, for any $(x_0,\dots,x_{r+1})\in X^{(r+2)}_{(\delta_0,D_0)}$, choosing some $\mu_{(\delta_0,D_0)}(x_0,\dots,x_{r+1})\in\mu_{\mathcal{CM}}(x_0,\dots,x_{r+1})$, gives a coarsely well defined map 
    $$
    \mu_{(\delta_0,D_0)}:X^{(r+2)}_{(\delta_0,D_0)}\to X.
    $$
    This is because $\diam(\mu_{\mathcal{CM}}(x_0,\dots,x_{r+1}))$ is uniformly bounded on $X^{(r+2)}_{(\delta_0,D_0)}$. 
    \item In a similar way, for any fixed $\delta\ge\delta_0$ and $D\ge 2((r+2)\ctrlf(0)+\delta)$, we have a well-defined map
    $\mu_{\delta}:X^{r+2}\to\Pc(X)\setminus\{X,\emptyset\}$ given by 
    $$
    \mu_{\delta}(x_0,\dots,x_{r+1})\coloneqq \bigcap_{j=0}^{r+1} \ngh{F_r(a_0,\dots,\wh{a}_j,\dots, a_{r+1})}{\lambda(\delta)},
    $$
    and a coarsely well defined map $\mu_{(\delta,D)}:X^{(r+2)}_{(\delta,D)}\to X$ defined by 
    \begin{align*}
        (x_0,\dots,x_{r+1}) \mapsto \mu_{(\delta,D)}(x_0,\dots,x_{r+1}) \in\bigcap_{j=0}^{r+1} \ngh{F_r(x_0,\dots,\wh{x}_j,\dots, x_{r+1})}{\lambda(\delta)}.
    \end{align*}
    In particular, (CM2) implies that for any $(x_0,\dots,x_{r+1})\in X^{(r+2)}_{(\delta,D)}$, we have 
    $$
    d(\mu_{(\delta_0,D_0)}(x_0,\dots,x_{r+1}),\mu_{\delta}(x_0,\dots,x_{r+1}))\le\Psi(\delta,D).
    $$
\end{enumerate}
\end{remark}

\subsection{Some properties of coarse fillings}
\label{sec:prop_of_coarse_filling}
We start by recording some immediate consequences of the definition of a coarse filling (\cref{defn:coarse_filling_weak}). 

\begin{observation} \label{rem: vertex containment}
Suppose $(F,\ctrlf)$ is a coarse $r$-filling on $(X,d)$, $k \in \{0,\dots, r\}$, and $a_0,\dots,a_k,a_{k+1} \in X$.
    Then
    \begin{enumerate}
        \item $a_i\in\ngh{F_k(a_0,\dots,a_k)}{\ctrlf(0)}$ for any $i\in \{0,\dots,k\}$.
        \item $d^{\Haus}(F_{k+1}(a_0,\dots,a_k,a_{k+1}),F_k(a_0,\dots,a_k))\le \ctrlf(d(a_k,a_{k+1})+\ctrlf(0))+\ctrlf(0)$.
    \end{enumerate}
\end{observation}
\begin{proof}
    (1) is immediate from Conditions (F0) and (F1) in \cref{defn:coarse_filling_weak}. 
    To prove (2), first note that  $a_k\in\ngh{F_{k+1}(a_0,\dots,a_k,a_k)}{\ctrlf(0)}$ by part (1), and hence, $$a_{k+1}\in\ngh{F_{k+1}(a_0,\dots,a_k,a_k)}{\ctrlf(0)+d(a_k,a_{k+1})}.$$
    Then (F4) implies
    $$F_{k+1}(a_0,\dots,a_k,a_{k+1})\subset\ngh{F_{k+1}(a_0,\dots,a_k,a_k)}{\ctrlf(\ctrlf(0)+d(a_k,a_{k+1}))}.$$ Finally (F3) implies
    $$F_{k+1}(a_0,\dots,a_k,a_{k+1})\subset\ngh{F_{k}(a_0,\dots,a_k)}{\ctrlf(\ctrlf(0)+d(a_k,a_{k+1}))+\ctrlf(0)}.$$
    On the other hand, (F1) implies $F_{k}(a_0,\dots,a_k)\subset\ngh{F_{k+1}(a_0,\dots,a_{k+1})}{\ctrlf(0)}.$
 \end{proof}

\begin{lemma}[Lipschitz property]\label{lem: lipschitz property}
    Suppose $(F,\ctrlf)$ is a coarse $r$-filling on $(X,d)$, $k \in \{0,\dots, r\}$, and $a_0,a_1,\dots,a_k\in X$. If $i\in \{0,\dots,k\}$ and $a'_i \in X$, then 
        \begin{equation*}
        d^{\Haus}\big(F_k(a_0,\dots,a'_i,\dots,a_k), F_k(a_0,\dots,a_i,\dots,a_k)\big) \leq \ctrlf(d(a_i,a'_i)+\ctrlf(0)).
        \end{equation*}
\end{lemma}
\begin{proof}By \cref{rem: vertex containment}(1), $d(a_i,F_k(a_0,\dots,a_k))<\ctrlf(0)$. Then $$d(a_i',F_k(a_0,\dots,a_k))<\ctrlf(0)+d(a_i,a'_i).$$ Then, property (F4) implies that 
$$
F_k(a_0,\dots,a'_i,\dots,a_k) \subseteq \ngh{F_k(a_0,\dots,a_i,\dots,a_k)}{\ctrlf(d(a_i,a_i')+\ctrlf(0))}.
$$
By switching the role of $a_i$ and $a_i'$, we obtain the result. 
\end{proof}

\begin{lemma}\label{lem:lipschitz_property_of_filling weak}
    Suppose $(X,d)$ is a geodesic metric space and $(F,\ctrlf)$ is a coarse $r$-filling on it. Then there exist constants $A ,B \geq 0$, depending on $\ctrlf$, such that $$d^{\Haus}(F_k(a_0,a_1,\dots,a_k), F_k(a_0',a_1',\dots,a_k')) \leq A \left(\sum_{j=0}^kd(a_j,a_j') \right) +B.$$
\end{lemma}
\begin{proof} Let $R:=d(a_0,a_0')$. \cref{lem: lipschitz property} implies that
    \begin{equation}\label{eqn: lipschitz}
        d^{\Haus}(F_k(a'_0,a_1,\dots,a_k), F_k(a_0,a_1,\dots,a_k)) \leq \ctrlf(R+\ctrlf(0)).
    \end{equation}

    If $\ctrlf$ is the zero function, there is nothing more to prove. So without loss of generality, we can pick $R_0 >0 $ so that $\ctrlf(R_0)>0$. Set $B\coloneqq \ctrlf(R_0+\ctrlf(0))$. As $\ctrlf$ is non-decreasing, $$B=\max_{R \in [0,R_0]}\ctrlf(R+\ctrlf(0)).$$ 
    
    Now note that it suffices to prove that
    \begin{align}\label{eqn:lip_in_first_coord weak}
       d^{\Haus}(F_k(a_0,a_1,\dots,a_k), F_k(a_0',a_1,\dots,a_k)) \leq A d(a_0,a_0')+B,
    \end{align}
    for some $A >0$. To prove this inequality, let $p\coloneqq \lfloor \frac{R}{R_0}\rfloor$ where $R = d(a_0,a_0')$. Now join $a_0$ and $a_0'$ by a geodesic, and pick points $u_0,\dots,u_{p+1}$ on this geodesic such that $u_0=a_0, u_{p+1}=a_{0}'$, and $d(u_0,u_1)=\dots=d(u_{p-1},u_{p})=R_0$. By triangle inequality, we get from \eqref{eqn: lipschitz} that
    \begin{align*}
        d^{\Haus}&(F_{k}(a_0,a_1,\dots,a_{k}),F_{k}(a_0',\dots,a_{k})) \\
        &\leq \sum_{j=0}^{p}d^{\Haus}(F_{k}(u_j,a_1,\dots,a_{k}),F_{k}(u_{j+1},a_1,\dots,a_{k})) \\
        & \leq B (p+1) = Bp+ B  \leq \frac{B}{R_0} R + B= A \cdot d(a_0,a_0')+B,
    \end{align*}
    where $A\coloneqq \frac{B}{R_0}$. 
\end{proof}

Now we recall \cref{prop: affine_control_function weak} and provide its proof.
\AffineCtrlf*
\begin{proof}
    We want to show that we can replace $\ctrlf$ in Definition \ref{defn:coarse_filling_weak} with an affine function. First, we notice that it suffices to check property (F4). Let $a_0,\dots,a_k\in X$ and fix $x\in\ngh{F_k(a_0,\dots,a_k)}{R}$. Then, there exists a point $y\in F_k(a_0,\dots,a_k)$ such that $d(x,y)<R$. By Lemma \ref{lem:lipschitz_property_of_filling weak}, for any $i=0,\dots,k$ we have
    \begin{equation}\label{eqn: linearity of control function 1}
        d^{\Haus}(F_k(a_0,\dots,a_{i-1},x,a_{i+1},\dots,a_k), F_k(a_0,\dots,a_{i-1},y,a_{i+1},\dots,a_k) \leq A\cdot R +B.
    \end{equation}
    Since $y\in\ngh{F_k(a_0,\dots,a_k)}{0}$, property (F4) implies that 
    \begin{equation}\label{eqn: linearity of control function 2}
        F_k(a_0,\dots,a_{i-1},y,a_{i+1},\dots,a_k)\subset\ngh{F_k(a_0,\dots,a_k)}{\ctrlf(0)}.
    \end{equation}
    Combining \eqref{eqn: linearity of control function 1} and \eqref{eqn: linearity of control function 2} with triangle inequality we get
    $$F_k(a_0,\dots,a_{i-1},x,a_{i+1},\dots,a_k)\subset\ngh{F_k(a_0,\dots,a_k)}{A\cdot R+B+\ctrlf(0)}.$$
    Therefore, we can replace $\ctrlf$ with the affine map $R\mapsto A\cdot R+(B+\ctrlf(0))$. It must be non-decreasing since $A$ can be chosen to be non-negative, see the proof of \cref{lem:lipschitz_property_of_filling weak}.
\end{proof}

\subsection{Example of coarse $r$-medians from convex projective geometry}
\label{sec: proof_of_pcd_are_coarse_median}
For the rest of this section, let $\Omega\subset\rpd$ be a quasi-homogeneous properly convex domain. For any $k\in\mathbb{N}$, we define $F_k:\Omega^{k+1}\to\Pc(\Omega) \setminus\{\emptyset,\Omega\}$ by 
\begin{equation*}
    (a_0,\dots,a_k) \mapsto F_k(a_0,\dots,a_k)\coloneqq\CH_\Omega(a_0,\dots,a_k).
\end{equation*}

\begin{proposition}\label{prop: filling_pcd}
    Let $\Omega\subset\rpd$ be a quasi-homogeneous properly convex domain.
    Fix $r\ge\rF(\Omega)$. Let $\ctrlf:\mathbb{R}_{\ge0}\to \mathbb{R}_{\ge0}$ denote the identity map.
    Then, the pair $(\Fc_r\coloneqq\{F_k\}_{k=0,\dots,r},\ctrlf)$, with $F_k$ defined as above, is an $r$-filling on $(\Omega,\hil)$.
\end{proposition}
\begin{proof}
    By definition of convex hull (\cref{def: convex_hull_pcd}), it is clear that $(\Fc_r,\ctrlf)$ satisfies properties (F0)--(F3). 
    To prove that (F4) holds, fix some $k=0,\dots,r$, $a_0,\dots,a_k\in\Omega$ and $R\ge0$. Let $x\in\ngh{F_k(a_0,\dots,a_k)}{R}$. 
    For any $i=0,\dots,k$, $x,a_0,\dots,\wh{a_i},\dots,a_k\in\ngh{F_k(a_0,\dots,a_k)}{R}$. Then, \cref{cor: convexity of neighborhoods}
    implies that 
    $$
    \CH_\Omega(x,a_0,\dots,\wh{a_i},\dots,a_k)\subset\ngh{F_k(a_0,\dots,a_k)}{R}.
    $$
    Hence, property (F4) holds.
\end{proof}
\begin{remark}\label{rmk: two_def_genericity}
    The two notions of genericity for tuples, \cref{def: generic_points_in_Hilbert_new} and \cref{def: generic_tuples_general}, coincide. Then, following \cref{notn:set_of_generic_pts}, we denote by $\Omega^{(r+2)}_{(\delta,D)}$ the set of all $(\delta,D)$-generic $(r+2)$-tuples in $\Omega$.
\end{remark}

\subsubsection{Proof of \cref{thm:main_coarse_r_median_on_pcd}}
\label{sec:proof_of_main_coarse_r_median_on_pcd}
    The proof will follow from \cref{prop:CM_2_for_pcd}, and \cref{rmk: two_def_genericity}. 
    
    Let $r\ge\rF(\Omega)$. Consider the $r$-filling $(\Fc_r,\ctrlf)$ from \cref{prop: filling_pcd}. 
    Let $\delta'_0>0$ be the constant in \cref{prop:CM_2_for_pcd}. 

    Let $\Psi:[\delta'_0,\infty)\times[\delta'_0,\infty)\to[0,\infty)$ the function in \cref{prop:CM_2_for_pcd}, restricted to this product of sub-intervals.

    Let $\lambda:[\delta_0',\infty)\to[0,\infty)$ be given by $\lambda(\delta)=\frac{\delta}{2}$.
    It follows from \cref{prop:CM_2_for_pcd} that for any $\delta \geq \delta_0'$ and any $a_0,\dots,a_{r+1}\in\Omega$, we have
    \begin{equation*}
    \bigcap_{i=0}^{r+1}\ngh{F_r(a_0,\dots,\wh{a_i},\dots,a_{r+1})}{\lambda(\delta)}\neq\emptyset.
    \end{equation*}
    Therefore (CM1) holds.
    \cref{prop:CM_2_for_pcd} and \cref{rmk: two_def_genericity} imply that (CM2) holds with this choice of $\lambda,\Psi$.
    This proves that $(\Fc_r,\ctrlf,\delta_0',\lambda,\Psi)$ is a coarse $r$-median.

\subsubsection{Coarse $r$-median map and $\Aut(\Omega)$ invariance}\label{sec: Coarse rF-median map}
Let $\Omega\subset\rpd$ be as above and fix $r\ge\rF(\Omega)$. Then the \emph{coarse $r$-median map} $\mu_{\delta_0}:\Omega^{r+2} \to \Pc(\Omega)-\{\emptyset,\Omega\}$ of \cref{defn:coarse_r_median_map_general} is given by the following.

\begin{definition}
\label{defn:coarse_median_map_conv_proj_case}
    Fix some $\delta_0>0$ as in \cref{thm:main_coarse_r_median_on_pcd}. Define $\mu_{\delta_0}:\Omega^{r+2} \to \Pc(\Omega)\setminus\{\emptyset,\Omega\}$ by 
\begin{align*}
    \mu_{\delta_0}(x_0,\dots,x_{r+1})\coloneqq C_{\frac{\delta_0}{2}}(x_0,\dots,x_{r+1}),
\end{align*}
for all $(x_0,\dots,x_{r+1})\in\Omega^{r+2}$. This map $\mu_{\delta_0}$ is called the \emph{coarse $r$-median map}. 
\end{definition}

We conclude this section by observing that for any fixed $D\ge2\delta>2\delta_0$, given a $(\delta,D)$-generic $(r+2)$-tuple $V$ in a quasi-homogeneous domain $\Omega\subset\rpd$, there is a coarsely well-defined center of the simplex whose vertex set is $V$; moreover, the notion of center is equivariant \wrt\ the action of $\Aut(\Omega)$. 
More precisely, the coarse $r$-median map induces a map from the set of generic tuples in $\Omega$ to $\Omega$, as follows. 
Recall that $\Omega^{(r+2)}_{(\delta,D)}$ is the set of  $(\delta,D)$-generic tuples.
Then, we define the map
$\mu_{(\delta,D)}:\Omega^{(r+2)}_{(\delta,D)}\to\Omega$
given by
$$
\mu_{(\delta,D)}(x_0,\dots,x_{r+1})\coloneqq \com(C_{\frac{\delta}{2}}(x_0,\dots,x_{r+1})).
$$
We remark that when $r=1$, since $D\ge2\delta$, all triples of points are $(\delta,D)$-generic. In particular, $\Omega^{(3)}_{(\delta,D)}=\Omega^3$, and $\mu_{(\delta,D)}$ is a map that coarsely defines the center of any straight triangle. Note also that in this case, $\mu_{(\delta,D)}$ does not depend on $D$.

\section{Equivalence of coarse 1-median and Bowditch's coarse median}
\label{sec:equiv_with_bowditch_median}
In this section, we recall Bowditch's classical notion of coarse median and  prove that it is equivalent to our notion of coarse 1-median.

\subsection{Coarse median spaces}
Coarse medians were introduced by Bowditch \cite{BowditchCoarse} as a generalization of median algebras. They provide a unified framework for median metric spaces like metric trees and $(\Rb^p,\ell^1)$, finite-dimensional CAT(0) cube complexes, Gromov hyperbolic spaces like $\Hb^2$, products of Gromov hyperbolic spaces like $\Hb^2 \times \Hb^2$, and mapping class groups of most orientable compact surfaces. Let us recall the classical definition due to Bowditch. Informally, Bowditch asks for the existence of a coarsely Lipschitz maps on 3-tuples (see M1), and the ability to approximate finite sets using CAT(0) cube complexes in a quantitative controlled way (see M2). 
\begin{definition}[\cite{BowditchCoarse}]\label{defn:bowditch_defn-coarse_median}
  Suppose $(X,d)$ is a metric space. A \emph{coarse median on $X$} is a triple $(X,d,\mu)$ with $\mu:X^3 \to X$ provided there exist $k \geq 0$ and $h:\Nb \cup \{0\} \to [0,\infty)$ such that:
  \begin{enumerate}
      \item[(M1)] If $a_i,a'_i \in X$ for $i=1,2,3$, then $$d(\mu(a_1,a_2,a_3),\mu(a'_1,a'_2,a'_3)) \leq k(\sum_{i=1}^3d(a_i,a'_i))+h(0).$$
      \item[(M2)] Suppose $A \subset X$ with $1 \leq |A| \leq p <\infty$ for some $p\in \Nb$. Then there is a finite median algebra $(M,\mu_M)$ and maps $i:A \to M$ and $\lambda:M \to X$ such that:
      \begin{enumerate}
           \item for all $a \in A$, $d(a,\lambda \circ i(a)) \leq h(p)$, and 
          \item for all $x,y,z \in A$, $d(\lambda \circ \mu_M(x,y,z), \mu(\lambda \circ i (x), \lambda \circ i (y), \lambda \circ i (z))) \leq h(p)$.
      \end{enumerate}
  \end{enumerate}
\end{definition}

Bowditch's notion of a coarse median is  quasi-isometry invariant ({\cite[Lemma 8.1]{BowditchCoarse}}) and passes to $\ell_1$-products (\cite{ZeidlerThesis}). Bowditch proved that relatively hyperbolic groups, with `good' peripheral subgroups, also admit coarse medians.
\begin{proposition}[{\cite{BowditchRelHyp}}]\label{prop: Bowditch relative hyperbolicity}
    Let $G$ be a group. Assume that $G$ is relatively hyperbolic with respect to a family $\mathcal{H}$ of proper subgroups. Then, if all the elements of $\mathcal{H}$ admit a coarse median structure, then $G$ also admits a coarse median structure.
\end{proposition}

\subsection{Coarse interval structure} In \cite{NWZ2021}, Niblo, Wright, and Zhang introduced the notion of a coarse interval structure and proved that it is equivalent to Bowditch's notion of coarse median. 
We now state the definition of a coarse interval structure. An astute reader may already observe the similarity between this definition and that of the notion of a coarse 1-filling (\cref{defn:coarse_filling_weak}) and coarse 1-median in our sense (\cref{defn:coarse_r_median}).

\begin{definition}[{\cite[Definition 3.11]{NWZ2021}}]
\label{defn:nwz_coarse_interval}
\label{def: coarse interval structure}
     Let $(X,d)$ be a metric space. A \emph{coarse interval structure} on $X$ is a triple $(X,d,[\cdot,\cdot])$, where $[\cdot,\cdot]:X^2\to\mathcal{P}(X)$ satisfies the following:
    \begin{itemize}
        \item [(I1)] There exists $k_0\ge0$ such that for all $a,b\in X$
        $$
        d^{\Haus}(a,[a,a])\le k_0\quad\text{and}\quad d^{\Haus}([a,b],[b,a])\le k_0.
        $$
        \item [(I2)] There exists a non-decreasing function $\Phi:\mathbb{R}_{\ge0}\to\mathbb{R}_{\ge0}$ such that for all $a,b\in X$ and $R\ge0$, if $c\in\ngh{[a,b]}{R}$, then
        $$
        [a,c]\subset\ngh{[a,b]}{\Phi(R)}.
        $$
        \item [(I3)] There exists a constant $k_1\ge0$ such that for all $a,b,c\in X$
        $$
        \ngh{[a,b]}{k_1}\cap \ngh{[c,b]}{k_1} \cap \ngh{[a,c]}{k_1}\neq \emptyset.
        $$
        \item [(I4)] There exists a non-decreasing function $\Psi:[k_1,\infty)\to[0,\infty)$ such that for all $a,b,c\in X$ and $R\ge k_1$
        $$
            \diam(\ngh{[a,b]}{R}\cap \ngh{[c,b]}{R} \cap \ngh{[a,c]}{R})\le\Psi(R).
        $$
    \end{itemize}
    In this case, $(k_0,k_1,\Phi,\Psi)$ are called \emph{parameters} of the coarse interval structure. We say that the parameters are \emph{affine} if both $\Phi$ and $\Psi$ are affine functions. 
\end{definition}

\subsection{Equivalence of coarse 1-median and Bowditch's coarse median}

We will say that the coarse 1-median structure in our sense has affine parameters if the parameter functions $\ctrlf$ and $\Psi$ are affine (\cref{defn:coarse_r_median_with_affine_params}).

\begin{proposition}
\label{prop:equiv_coarse_1_median}
   For a metric space $(X,d)$, the following are equivalent:
    \begin{enumerate}
        \item $(X,d)$ has a coarse median in the sense of Bowditch (\cref{defn:bowditch_defn-coarse_median}).
        \item $(X,d)$ has a coarse interval structure with affine parameters (\cref{defn:nwz_coarse_interval}).
        \item $(X,d)$ has a coarse 1-median structure in our sense that is $(\ctrlf,\Psi)$-affine (\cref{defn:coarse_r_median}). 
    \end{enumerate}
\end{proposition}

We will spend the rest of this section proving this result. 

\noindent $(1) \iff (2)$ is the content of \cite[Theorem 3.15]{NWZ2021}. Indeed, they show that if $(X,d,\mu)$ is a coarse median space in the sense of Bowditch, then $[\cdot,\cdot]:X^2\to \Pc(X)$ defined by $[a,b]=\{\mu(a,x,b)\mid x\in X\}$ induces a coarse interval structure on $X$. Conversely, a coarse interval structure $(X,d,[\cdot,\cdot])$ on $X$, with affine parameters $(k_0,\ k_1, \Phi, \Psi)$, induces a coarse median structure via the map $\mu:X^3\to X$ that takes $(a,b,c) \in X^3$ to some point  in $\ngh{[a,b]}{k_1}\cap \ngh{[c,b]}{k_1} \cap \ngh{[a,c]}{k_1}$. This map is coarsely well-defined due to (I3) and (I4) in \cref{defn:nwz_coarse_interval}.

Now we will prove that $(2) \iff (3)$. This follows from the two lemmas below.
\begin{lemma}\label{lem: coarse medians are coarse 1-medians}
Let $(X,d)$ be a metric space. Assume that $[\cdot,\cdot]:X\times X\to \Pc(X)$ is a coarse interval structure with parameters $k_0,k_1,\Phi,\Psi$.
\begin{enumerate}
    \item Define the maps $F_0:X\to\Pc(X)$ and $F_1:X\times X\to \Pc(X)$ by $$F_0(a_0)=\{a_0\}\quad \text{ and }\quad F_1(a_0,a_1)=[a_0,a_1]$$ for all $a_0,a_1 \in X$. Define $\ctrlf:\mathbb{R}_{\ge0}\to\mathbb{R}_{\ge0}$ by $\ctrlf(R)\coloneqq\Phi(R)+3k_1$. Then  $(F\coloneqq\{F_0,F_1\},\ctrlf)$ is a coarse 1-filling on $X$.
    \item Set $\delta_0\coloneqq k_1$. Define $\lambda:[\delta_0,\infty)\to[0,\infty)$ by $\lambda(\delta)\coloneqq\delta$, and $\Psi':[\delta_0,\infty)\times[\delta_0,\infty)\to[0,\infty)$ by $\Psi'(\delta,D)\coloneqq\Psi(\delta)$. Then $(X,d)$ admits a coarse 1-median with parameters $(\delta_0,\lambda,\Psi')$ with respect to the coarse 1-filling $(F,\ctrlf)$ above. 
\end{enumerate}
Moreover, if the coarse interval structure has affine parameters, then the resulting coarse 1-median structure is $(\ctrlf,\Psi)$-affine.
\end{lemma}
\begin{proof}
(1) We have to show that properties (F0)--(F4) hold.
Property (F0) is immediate. Property (F1) follows by (I1) and (I3). Indeed, given $a,b\in X$ we have $F_0(a)=\{a\}$ by definition. Property (I3) implies that $\ngh{[a,a]}{k_1}\cap \ngh{[a,b]}{k_1} \cap \ngh{[a,b]}{k_1}\neq \emptyset$. By (I1) we have 
$$
\ngh{[a,a]}{k_1}\subset B_d(a,2k_1).
$$
Hence, $a\in\ngh{[a,b]}{3k_1}$ and, in particular, $a\in\ngh{F_1(a,b)}{\ctrlf(0)}$. Properties (F2) and (F3) follow directly from (I1), while (F4) follows from (I2) and (I1).
Therefore, the pair $(F,\ctrlf)$ is a coarse 1-filling.

\noindent (2) Since $\Psi$ is non-decreasing, we get that $\Psi'$ is non-decreasing in both entries. Moreover, $\lambda$ is clearly linear and non-constant. 

We claim that all triples of points in $X$ are $(\delta,2\delta)$-generic for any $\delta\ge0$. 
Indeed, given a set of three points $\{a_0,a_1,a_2\}\subset X$, let $I_1,\dots,I_k \subseteq \{0,1,2\}$ be a family of subsets such that $\bigcap_{i=1}^kI_i=\emptyset$ and there exists some $j\in\{1,\dots,k\}$ such that $\abs{I_j}<2$. Then, $I_j$ is a singleton, and hence the intersection $\bigcap_{i=1}^k\ngh{F_{\abs{I_i}-1}(a_\ell\mid \ell\in I_i)}{\delta}$ is contained in some ball of radius $\delta$. Consequently, we have
$$\diam\left(\bigcap_{i=1}^k\ngh{F_{\abs{I_i}-1}(a_\ell\mid \ell\in I_i)}{\delta}\right)<2\delta.$$
Hence, properties (CM1) and (CM2) follow directly from (I4) and the definition of $\lambda$ and $\Psi'$.

To prove the ``moreover'' part, assume $\Phi$ and $\Psi$ to be affine. Then, by definition, $\ctrlf$ is affine and $\Psi'$ is affine \wrt\ both entries, hence the coarse 1-median is $(\ctrlf,\Psi)$-affine.
\end{proof}

\begin{lemma}\label{lem: coarse 1-medians are coarse medians}
Let $(X,d)$ be a metric space. Assume that $(F=\{F_0,F_1\},\ctrlf)$ is a coarse 1-filling on $X$ for which $X$ has a coarse 1-median with parameters $\delta_0$, $\lambda$ and $\Psi$. Define the map $[\cdot,\cdot]:X\times X\to \Pc(X)$ by $[a_0,a_1]:=F_1(a_0,a_1)$. Then $[\cdot,\cdot]:X\times X\to \Pc(X)$ is a coarse interval structure on $X$.

Moreover, if the coarse 1-median structure is $(\ctrlf,\Psi)$-affine, then the resulting coarse interval structure has affine parameters.
\end{lemma}
\begin{proof}
    Define $k_0\coloneqq k_1\coloneqq \max\{\ctrlf(0),\delta_0,\lambda(0)\}$, $\Phi:\mathbb{R}_{\ge0}\to\mathbb{R}_{\ge0}$ as $\Phi(R)=\ctrlf(R)$ and $\Psi':[k_1,\infty)\to[0,\infty)$ as $\Psi'(R)=\Psi(\lambda^{-1}(R),2\lambda^{-1}(R))$.
    We show that properties (I1)--(I4) of Definition \ref{def: coarse interval structure} hold with parameters $k_0,k_1,\Phi,$ and $\Psi'$. Property (I1) follows from (F0)--(F3). Property (I2) follows from (F4) and property (I3) follows from (CM1). We showed in the proof of Lemma \ref{lem: coarse medians are coarse 1-medians} that all triples of points in $X$ are $(\lambda^{-1}(R),2\lambda^{-1}(R))$-generic, for any $R\ge k_1\ge\lambda(0)$. Therefore, property (I4) follows from (CM2) and the fact that $\Psi$ is non-decreasing.

    Moreover, if we assume that the coarse 1-median is $(\ctrlf,\Psi)$-affine, then $\Phi\equiv\ctrlf$ and $\Psi'$ are affine, since $\lambda$ is affine.
\end{proof}

\subsection{Coarse $1$-median spaces from convex projective geometry} \label{sec:coarse_1_median_in_conv_proj_geom}
In this section, we will show that any  divisible and relatively hyperbolic (with respect to abelian peripherals) properly convex domain $\Omega \subset \RP$ admits a coarse 1-median. It is a rich family of properly convex domains, consisting of many examples coming from a variety of sources, like Coxeter reflection groups \cite{YB2006, DGKLM2024}, deformation of cusped hyperbolic 3-manifolds, etc.

Recall our definition of  $r_0$ from the Introduction: $r_0(\Omega,\hil)$ is the minimum $r$ such that a $(\Omega,\hil)$ admits a coarse $r$-median. By \cref{thm:main_coarse_r_median_on_pcd}, $r_0(\Omega,\hil) \leq \rF(\Omega)$. Our examples here will be divisible properly convex domainss that admit a coarse $1$-median, so that $r_0(\Omega,\hil)=1$, although $\rF(\Omega) \geq 2$. This shows that our question about symmetric spaces (\cref{ques:r0_vs_rPES}) will definitely fail for generic divisible domains.

These examples will come from properly convex domains with isolated simplices -- the convex projective analogues of CAT(0) spaces with isolated flats. Fix a properly convex domain $\Omega \subset \RP$ and a discrete subgroup $\Gamma <\Aut(\Omega)$ such that $\Gamma$ divides $\Omega$, i.e. $\Omega/\Gamma$ is compact. We will say that a PES in $\Omega$ is maximal if it is not properly contained in another PES. Let $\Sc_{\max}(\Omega)$ denote the set of all maximal PES of dimension at least 2 in $\Omega$. We say that $\Omega/\Gamma$ has \emph{isolated simplices} if $\Sc_{\max}(\Omega)$ is closed, and discrete in the local Hausdorff topology. 

\begin{remark}
\label{rem:isolated_simplices_rank_one}
If $\Omega/\Gamma$ has isolated simplices with $\Sc_{\max}(\Omega)\neq \emptyset$, then:
\begin{enumerate}
    \item $\Omega/\Gamma$ is \emph{rank one} (in the sense of \cite{I2024}): $\Gamma$ contains a biproximal matrix $\gamma$ and $\gamma$ acts by a translation along a projective line in $\Omega$ \cite[Prop. 6.8]{I2024}. 
    \item $\rF(\Omega)\geq 2$ (because $\Sc_{\max}(\Omega) \neq \emptyset$).
\end{enumerate}
\end{remark}
\begin{theorem}[{\cite[Theorem 1.13 and 1.18]{IZRelaHyp}}]\label{thm: Islam-Zimmer relative hyperbolicity}
    Let $\Omega\subseteq\RP$ be a properly convex domain and $\Gamma <\Aut(\Omega)$ is a discrete subgroup such that $\Omega/\Gamma$ is compact. Then, the following are equivalent: 
    \begin{enumerate}
        \item $\Omega/\Gamma$ has isolated simplices.
        \item There exist $S_1,\dots,S_m \in \Sc_{\max}(\Omega)$ such that $\Sc_{\max}(\Omega)=\sqcup_{i=1}^m \Gamma \cdot S_i$ and $\Gamma$ is relatively hyperbolic with respect to $$\{{\Stab}_\Gamma(S_1),\dots,{\Stab}_\Gamma(S_m)\},$$ where each $\Stab_{\Gamma}(S_j)$ is virtually isomorphic to $\Zb^{\dim(S_j)}$ for $j\in\{1,\dots,m\}$.
    \end{enumerate}
\end{theorem}

As a consequence of this theorem, we can put a coarse 1-median on domains with isolated simplices.
\begin{corollary}\label{cor: coarse median on divisible properly convex domains with isolated flats}
    Let $\Omega\subseteq\RP$ and $\Gamma <\Aut(\Omega)$ be a discrete subgroup such that $\Omega/\Gamma$ is compact. Assume that $\Sc_{\max}(\Omega)\neq \emptyset$ and $\Omega/\Gamma$ has isolated simplices. Then:
    \begin{enumerate}
        \item $(\Omega,\hil)$ admits a coarse 1-median structure, which implies $r_0(\Omega,\hil)=1$.
        \item $\Omega/\Gamma$ is rank one and $\rF(\Omega) \geq 2$ (see \cref{rem:isolated_simplices_rank_one}).
    \end{enumerate}
\end{corollary}
\begin{proof}
    By \cref{thm: Islam-Zimmer relative hyperbolicity}, $\Sc_{\max}(\Omega)=\sqcup_{j=1}^m \Gamma \cdot S_j$ and the group $\Gamma$ is relatively hyperbolic \wrt\ $\{H_1,\dots,H_m\}$  where each $H_j$ is a subgroup of $\Stab_{\Gamma}(S_j)$ with $H_j \cong \Zb^{\dim(S_j)}$  for $j=1,\dots,m.$ Now note that each $H_j$ admits a coarse 1-median for $j=1,\dots,m$. Indeed, $H_j$ is quasi-isometric to $(\Rb^{\dim(S_j)},\ell_1)$ and $(\Rb^{p},\ell_1)$ admits a coarse 1-median for any $p\geq 1$. Since the coarse 1-median structure is a quasi-isometry invariant, $H_j$ also admits a coarse 1-median structure. Then, by \cref{prop: Bowditch relative hyperbolicity} and \cref{prop:equiv_coarse_1_median}, we conclude that $\Gamma$ admits a coarse 1-median. As $\Gamma$ acts properly, isometrically, co-compactly on $(\Omega,\hil)$, the \v{S}varc-Milnor Lemma implies that $\Gamma$ is quasi-isometric to $(\Omega,\hil)$. Again, because of quasi-isometry invariance, we can put a coarse 1-median on $(\Omega,\hil)$. Thus $r_0(\Omega,\hil)=1$. 
    The rest is \cref{rem:isolated_simplices_rank_one}.
\end{proof}

\section{Quasi-isometry invariance}\label{sec: quasi-isom}

A principle of coarse geometry is that large-scale geometric properties should be invariant under quasi-isometries. Therefore, the goal of this section is to prove that admitting a coarse $r$-median structure is indeed a quasi-isometry invariant.

We will do this in different steps. First, we explicitly construct a push-forward of a coarse $r$-filling along a quasi-isometry; then, we study how generic points behave when we apply a quasi-isometry. Finally, we show that the constructed coarse $r$-filling is indeed associated to a coarse $r$-median structure. 

\subsection{Quasi-isometry and basic lemmas}
We begin by recalling some basics of coarse geometry.
\begin{definition}
\label{defn:qiso}
    Let $(X,d_X)$ and $(Y,d_Y)$ be metric spaces. Fix $K\ge1$ and $L\ge0$. Then, a map $f:X\to Y$ is 
    \begin{enumerate}
        \item a  $(K,L)$-\emph{quasi-isometric embedding} if for all $x_1,x_2\in X$, 
        \begin{equation*}
        \frac{1}{K}d_X(x_1,x_2)-L \le d_Y(f(x_1),f(x_2)) \le Kd_X(x_1,x_2)+L.
        \end{equation*}
        \item a $(K,L)$-\emph{quasi-isometry} if  $f$ is a $(K,L)$-quasi-isometric embedding and $Y \subset \Nc_{L}(f(X))$.
        \item a \emph{quasi-inverse} of a $(K,L)$-quasi-isometry $g:Y\to X$ if there exists $C' \geq 0$ such that
        $$\sup_{x \in X}d_X(x,g\circ f(x))\le C'\quad\text{and}\quad \sup_{y\in Y}d_Y(y,f\circ g(y))\le C'.$$
    \end{enumerate}
\end{definition}

    Note that a quasi-inverse is only coarsely well-defined. Below we show that we can pick a quasi-invserse where $C'=2KL$.

\begin{proposition}
    \label{prop:quasi-inverse}
    Let $f:X\to Y$ be a $(K,L)$-quasi-isometry. Then:
   \begin{enumerate}
       \item if $g'$ is a quasi-inverse of $f$, then $g'$ is a $(K,2KC_Y'+KL)$-quasi-isometry, where  $C_Y':=\sup_{y\in Y}d_Y(y,f \circ g'(y))$.
       \item there exists a quasi-inverse $g:Y\to X$ such that $$\sup_{x \in X}d_X(x,g\circ f(x))\le 2KL \quad\text{and}\quad  \sup_{y\in Y}d_Y(y,f\circ g(y))\le L.$$ In particular, $g$ is a $(K,3KL)$-quasi-isometry.
   \end{enumerate} 
\end{proposition}
\begin{proof}
(1) Let $g'$ and $C_Y'$ be as in the statement of the proposition. Let $y_1, y_2 \in Y$. By definition, $d_Y(y_1, f(g'(y_1))) \le C_Y'$ and $d_Y(y_2, f(g'(y_2))) \le C_Y'$.
   As $f$ is a $(K,L)$-quasi-isometry, 
	\begin{align*}
	\frac{1}{K} d_X(g'(y_1), g'(y_2)) - L &\le d_Y(f(g'(y_1)), f(g'(y_2))) \\
	&\le d_Y(f(g'(y_1)), y_1) + d_Y(y_1, y_2) + d_Y(y_2, f(g'(y_2)))\\
	& \le 2C_Y' + d_Y(y_1, y_2).
	\end{align*}
	Hence, we get
    \begin{equation}\label{eqn: quasi-inverse 1}
    d_X(g'(y_1), g'(y_2)) \le K d_Y(y_1, y_2) + 2KC_Y'+KL.
    \end{equation}
    On the other hand,
	\begin{align*}
	d_Y(y_1, y_2) &\le d_Y(y_1, f(g'(y_1))) + d_Y(f(g'(y_1)), f(g'(y_2))) + d_Y(f(g'(y_2)), y_2)\\
	&\le 2C_Y' + K d_X(g'(y_1), g'(y_2)) + L.
	\end{align*}
	Hence, we get
    \begin{equation}
        \label{eqn: quasi-inverse 2}
        d_X(g'(y_1), g'(y_2)) \ge \frac{1}{K} d_Y(y_1, y_2) - \frac{2C_Y'+L}{K}.
    \end{equation}  
    But $\frac{2C_Y'+L}{K} \le K(2C_Y'+L)$ as $K\geq1$. Then \eqref{eqn: quasi-inverse 1} and \eqref{eqn: quasi-inverse 2} finish the proof.

(2)    Since $f: X \to Y$ is a $(K,L)$-quasi-isometry,  for every $y \in Y$ there exists at least one $x \in X$ such that $d_Y(y, f(x)) \le L$. Using the axiom of choice, for each $y \in Y$, we pick one such $x$ and define our map $g: Y \to X$ as $g(y) = x$. Then
    \begin{equation*}
        d_Y(y,f(g(y)))\le L \quad\text{for all } y \in Y.
    \end{equation*}
  For any $x \in X$, by definition of $g$, $d_Y(f(x), f(g(f(x)))) \le L.$ Then, as $f$ is a $(K,L)$-quasi-isometry, we have 
  \begin{equation*}
  \sup_{x \in X} d_X(x, g(f(x))) \le 2KL. \qedhere
  \end{equation*}
\end{proof}

We quickly recall some facts about quasi-isometries that we will need.

\begin{proposition}\label{prop: quasi-isometry quasi-commutes with intersection}
    Let $(X, d_X)$ and $(Y, d_Y)$ be metric spaces, and let $f: X \to Y$ be a $(K, L)$-quasi-isometric embedding. Let $A_1, \dots, A_k \subset X$ be a finite collection of subsets. Then,
    $$f\left(\bigcap_{i=1}^k A_i\right)\subset\bigcap_{i=1}^k f(A_i)\subset f\left(\bigcap_{i=1}^k \ngh{A_i}{KL}\right).$$
\end{proposition}

\begin{proof}
    Only the second inclusion needs a proof. Let $y \in \bigcap_{i=1}^k f(A_i)$. Then for each $i=1,\dots,k$, there exists $x_i\in A_i$ such that $y=f(x_i)$. 
	Since $f(x_1)=f(x_2)=\dots=f(x_k)= y$ and $f$ is a $(K,L)$-quasi-isometric embedding, we get that for all $i=2,\dots,k$
    $$d_X(x_1, x_i)\le K(d_Y(f(x_1), f(x_i))+L)\le KL.$$
	Thus, $x_1\in\bigcap_{i=1}^k \ngh{A_i}{KL}$ and hence $y\in f\left(\bigcap_{i=1}^k \ngh{A_i}{KL}\right)$.
\end{proof}

The next two propositions are routine and hence we skip the proofs. 

\begin{proposition}\label{prop: quasi-isometry quasi-commutes with taking ngh}
Let $f: X \to Y$ be a $(K, L)$-quasi-isometry. Given $A \subset X$,
$$\ngh{f(A)}{\frac{\delta}{K}-2L} \subset \ngh{f(\ngh{A}{\delta})}{L}\quad\text{and}\quad f(\ngh{A}{\delta}) \subset \ngh{f(A)}{K\delta+L}.$$
\end{proposition}

\begin{proposition}\label{prop:coarse containment under QI}
Let $f: X \to Y$ be a $(K, L)$-quasi-isometry, and $A,B \subset X$. 
\begin{enumerate}
    \item If $r>0$ and $A \subset \Nc_r(B)$, then $f(A)\subset \Nc_{Kr+L}(f(B)).$ 
    \item $d^{\Haus}_Y(f(A),f(B)) \leq K \cdot d^{\Haus}_X(A,B)+L$.
    \item $\diam(A)\leq K \left( \diam(f(A))+L \right)$.
\end{enumerate}
\end{proposition}

\subsection{Push-forward of a coarse filling}

We now define the push-forward of a coarse $r$-filling and show that the notion is coarsely well-defined. The geometric intuition is the following: given a quasi-isometry $f:X\to Y$, to fill a tuple of points in Y, we use a quasi-inverse to map them back to $X$, fill them in $X$, and then push the resulting set to $Y$ via the quasi-isometry $f$.

\begin{definition}\label{def: push-forward of a coarse r-filling}Suppose that:
\begin{enumerate}
    \item $f:(X,d_X)\to (Y,d_Y)$ is a $(K,L)$-quasi-isometry, 
    \item $(F,\ctrlf)$ with $F=\{F_i:X^{i+1}\to\Pc(X)\}_{i=0,\cdots,r}$ is a coarse $r$-filling on $X$, and
    \item $g:Y\to X$ be a quasi-inverse of $f$ such that $\sup_{y\in Y}d_Y(y,f\circ g(y))\le C'$.
\end{enumerate} 
Then, we define the \emph{push-forward of $(F,\ctrlf)$ via $g$} as the pair $(F_*,\ctrlf_*)$ with $F_*=\{F_{i,*}:Y^{i+1}\to\Pc(Y)\}_{i=0,\dots,r}$ where 
\begin{align*}
    & F_{i,*}(y_0, \dots, y_i) := f\left( F_i \big( g(y_0), \dots, g(y_i) \big) \right), ~~\text{and} \\
    & \ctrlf_*(R) := K \cdot \ctrlf \big(K\cdot(R + L+C') \big) + L+C' \nonumber.
\end{align*}
\end{definition}

\begin{remark} \
\begin{enumerate}
    \item The push-forward $(F_*,\Lambda_*)$ is a coarse $r$-filling on $Y$ (\cref{prop: filling quasi-isometry}). 
    \item By \cref{lem: push-forward coarsely independent 1}, the push-forward $(F_*,\Lambda_*)$  is coarsely independent of the choice of the quasi-inverse $g$ in the following sense:  changing the quasi-inverse produces coarse fillings that are coarsely equivalent (Definition \ref{def: push-forward of a coarse r-filling}).
\end{enumerate}
\end{remark}

\begin{proposition}\label{prop: filling quasi-isometry}
    Suppose $f:X \to Y$ is a $(K,L)$-quasi-isometry and $g$ is its quasi-inverse such that $\sup_{y\in Y}d_Y(y,f\circ g(y))\le C'$. If $(F,\ctrlf)$ is a coarse $r$-filling on $X$, then the push-forward $(F_*,\ctrlf_*)$ of $(F,\Lambda)$ via $g$ is a coarse $r$-filling on $Y$.
\end{proposition}

\begin{proof}    
     For the rest of the proof, fix $i\in\{0,\dots,r\}$ and $y_0,\dots,y_i\in Y$. Set  
    $x_j\coloneqq g(y_j)$ for each $j=0,\dots,i$. Then, by definition, $d_Y(y_j,f(x_j) \leq C'$ for each $j=0,\dots,i$. Observe that by definition
    $$F_{i,*}(y_0, \dots, y_i)=f( F_i(x_0,\dots,x_i)).$$ 

We now verify (F0)-(F4) in Definition \ref{defn:coarse_filling_weak} for $F_*$. To verify each (Fq) where $q \in \{0,\dots,4\}$, we use that $(F,\ctrlf)$ satisfies (Fq) and then apply \cref{prop:coarse containment under QI}. 

    \noindent (F0) Applying the triangle inequality, \cref{prop:coarse containment under QI}(2), and (F0) to $(F,\ctrlf)$, we get
     \begin{align*}
            d^{\Haus}_Y(y_0, F_{0,*}(y_0)) &\le d_Y(y_0, f(x_0)) + d^{\Haus}_Y(f(x_0), f(F_0(x_0)) \\
            &\le C' + K \cdot d^{\Haus}_X(x_0, F_0(x_0)) + L  
            \\
            &\le C'+L + K \cdot \ctrlf(0) \le \ctrlf_*(0). 
        \end{align*}

      \noindent (F1)  Fix $j=0,\dots,i$.
        Since $(F,\ctrlf)$ satisfies (F1), we have
        $$F_{i-1,*}(y_0,\dots,\wh{y_j},\dots,y_i)=F_{i-1}(x_0,\dots,\wh{x_j},\dots,x_i)\subset \ngh{F_{i}(x_0,\dots,x_i)}{\ctrlf(0)}.$$
        Then, by \cref{prop:coarse containment under QI}(1),
        \begin{align*}
        f(F_{i-1}(x_0,\dots,\wh{x_j},\dots,x_i)) &\subset \ngh{f(F_{i}(x_0,\dots,x_i))}{K\ctrlf(0)+L}\\
        \implies F_{i-1,*}(y_0,\dots,\wh{y_j},\dots,y_i) & \subset \ngh{F_{i,*}(y_0,\dots,y_i)}{\ctrlf_*(0)}.
        \end{align*}

      \noindent (F3) Since $(F,\ctrlf)$ satisfies (F3), \cref{prop:coarse containment under QI}(1) implies that
        \begin{align*}
        f(F_{i+1}(x_0,\dots,\wh{x_j},\dots,x_i,x_i)) &\subset \ngh{f(F_{i}(x_0,\dots,x_i))}{K\ctrlf(0)+L}\\
        \implies F_{i+1,*}(y_0,\dots,\wh{y_j},\dots,y_i,y_i) & \subset \ngh{F_{i,*}(y_0,\dots,y_i)}{\ctrlf_*(0)}.
        \end{align*}
    \begin{notation*}
     For the rest of the proof, set $\mathbf{x}\coloneqq(x_0,\dots,x_i)$ and $\mathbf{y}\coloneqq(y_0,\dots,y_i)$.
    \end{notation*}

        \noindent (F2) Let $\sigma$ be a permutation of $\{0, \dots, i\}$. Let $\mathbf{x}_\sigma\coloneqq(x_{\sigma(0)}, \dots, x_{\sigma(i)})$ and $\mathbf{y}_\sigma\coloneqq(y_{\sigma(0)}, \dots, y_{\sigma(i)})$.   Since $(F,\ctrlf)$ satisfies (F2),  \cref{prop:coarse containment under QI}(2) implies that 
         \begin{align*}
             d^{\Haus}_Y &(F_{i,*}(\mathbf{y}), F_{i,*}(\mathbf{y}_\sigma)) = d^{\Haus}_Y(f(F_i(\mathbf{x})), f(F_i(\mathbf{x}_\sigma)))\\
            &\le K \cdot d^{\Haus}_X(F_i(\mathbf{x}), F_i(\mathbf{x}_\sigma)) + L \le K \cdot \ctrlf(0) + L\le \ctrlf_*(0).
        \end{align*}

         \noindent (F4) Fix $j=0,\dots,i$ and $z \in \ngh{F_{i,*}(\mathbf{y})}{R}=\ngh{f(F_{i}(\mathbf{x}))}{R}$. Then there exists $w \in F_i(\mathbf{x})$ such that
        $d_Y(z, f(w)) \le R.$
        As $f$ is a $(K,L)$-quasi-isometry, 
        \begin{align*}
            \frac{1}{K} d_X(g(z), w) - L &\le d_Y(f(g(z)), f(w)) \\
            &\le d_Y(f(g(z)), z) + d_Y(z, f(w)) \leq C'+R
        \end{align*}
        Thus, $d_X(g(z), w)  \le K\cdot (C'+L + R).$
        This implies that $g(z) \in \ngh{F_i(\mathbf{x})}{K\cdot(R + L+C')}$. Since $(F,\ctrlf)$ satisfies (F4),
        $$
        F_i(x_0,\dots,x_{j-1},g(z),x_{j+1},\dots,x_i)\subseteq \ngh{F_i(\mathbf{x})}{\ctrlf(K\cdot(R + C'+L))}.
        $$
       Then, by \cref{prop:coarse containment under QI}(1),  
       \begin{align*}
        F_{i,*}(y_0,\dots,y_{j-1},z,y_{j+1},\dots,y_i)&\subseteq\ngh{F_{i,*}(\mathbf{y})}{K \cdot \ctrlf(K\cdot(R + L+C')) + L}\\
        &\subset\ngh{F_{i,*}(\mathbf{y})}{\ctrlf_*(R)}. \qedhere
        \end{align*}
\end{proof}

Recall the definition of coarse equivalence between two coarse fillings from \cref{defn:coarse_equiv_of_fillings}. We now show that push-forwards defined in \cref{def: push-forward of a coarse r-filling} are coarsely independent of the choice of the quasi-inverse.
\begin{lemma}\label{lem: push-forward coarsely independent 1}
    Let $(X,d_X)$ be endowed with a coarse $r$-filling $(F,\ctrlf)$ and $f:(X,d_X)\to (Y,d_Y)$ be a $(K,L)$-quasi-isometry. If $g,g':Y\to X$ are two quasi-inverses of $f$, then the push-forwards of $(F,\ctrlf)$ via $g$ and $g'$ respectively are coarsely equivalent. 
\end{lemma}
\begin{proof}
    Since $g,g'$ are quasi-inverses, there exists $C'$ such that 
    $$\sup_{y\in Y}d_Y(y, f\circ g(y)) \le C'\quad\text{and}\quad \sup_{y\in Y}d_Y(y, f\circ g'(y))) \le C'.$$
    Set $C:=K\cdot(r+1)\cdot \ctrlf(2KC'+KL+\ctrlf(0))+ L$. \\
    Fix $i\in\{0,\dots,r\}$ and $y_0, \dots, y_i\in Y$. We will finish the proof by showing that 
    $$
    d_Y^{\Haus}(f(F_i(g(y_0),\dots,g(y_i))),f(F_i(g'(y_0),\dots,g'(y_i))))\le C. 
    $$
    To do this, we let $x_j\coloneqq g(y_j)$ and $x_j'=g'(y_j)$ for any $j=0,\dots,i$. Then
    $$
    d_Y(y_j, f(x_j)) \le C'\quad\text{and}\quad d_Y(y_j, f(x'_j)) \le C'.$$
    Then for each $j$,
    $$
    d_Y(f(x_j), f(x'_j)) \le d_Y(f(x_j), y_j) + d_Y(y_j, f(x'_j)) \le 2C'.
    $$
    Since $f$ is a $(K,L)$-quasi-isometry,  $d_X(x_j, x'_j) \le K(2C'+ L).$
    
    It follows from \cref{lem: lipschitz property}  that
    \begin{equation*}
        d^{\Haus}(F_i(x_0,\dots,x_{j-1},x_j',x_{j+1},\dots,x_i), F_i(x_0,\dots,x_i)) \leq \ctrlf(2KC'+KL+\ctrlf(0)).
    \end{equation*}
    Therefore, by successive applications of the triangle inequality,
    \begin{align*}
    d_X^{\Haus}(F_i(x_0, \dots, x_i ), F_i(x_0', \dots, x_i')) &\leq (i+1)\ctrlf(2KC'+KL+\ctrlf(0))
    \end{align*}
    Then, by \cref{prop:coarse containment under QI}(1), 
    \begin{equation*}\label{eqn: invariance quasi-isometry 1}
        d_Y^{\Haus}(f(F_i(x_0, \dots, x_i )), f(F_i(x_0', \dots, x_i'))) \le C. \qedhere
    \end{equation*}
\end{proof}

We finally show that the push-forward operation  is coarsely involutive.
\begin{lemma}\label{lem: double push-forward}
    Let $f: (X, d_X) \to (Y, d_Y)$ be a $(K,L)$-quasi-isometry and let $g: Y \to X$ be a quasi-inverse of $f$. Let $(F, \ctrlf)$ be a coarse $r$-filling on $X$. If $(F_*, \ctrlf_*)$ is the push-forward of $(F, \ctrlf)$ via $g$, and  $(F_{**}, \ctrlf_{**})$ is the push-forward of $(F_*, \ctrlf_*)$ via $f$, then the coarse fillings $(F_{**}, \ctrlf_{**})$ and $(F, \ctrlf)$ are coarsely equivalent. 
    
    Moreover, if $g$ is as in \cref{prop:quasi-inverse}(2), then the Hausdorff distance between the fillings $F_{**}$ and $F$ is uniformly bounded above by $C=2KL + (r+1)\ctrlf(2KL + \ctrlf(0))$.
\end{lemma}
\begin{proof}
    Because of \cref{lem: push-forward coarsely independent 1}, it suffices to prove this result for the particular quasi-inverse $g:Y\to X$ as in \cref{prop:quasi-inverse}, i.e. we may assume that $\sup_{x\in X}d_X(x, g\circ f(x))\leq 2KL$.  
    
    Fix any $i \in \{0, \dots, r\}$ and any $x_0, \dots, x_i \in X$. The double push-forward filling $F_{**} = \{F_{i,**}\}_{i=0,\dots,r}$ is defined by pushing forward $F_{i,*}$ via $f$, i.e. 
    $$
    F_{i,**}(x_0, \dots, x_i) = g(F_{i,*}(f(x_0), \dots, f(x_i)))=g\left( f \left( F_{i}(g(f(x_0)), \dots, g(f(x_i))\right)\right).
    $$
    Setting $x'_j \coloneqq g(f(x_j))$ for each $j=0,\dots,i$, we have 
    $$F_{i,**}(x_0, \dots, x_i)=g\circ f (F_i(x_0',\dots,x_i')). $$
    Set $C \coloneqq 2KL + (r+1)\ctrlf(2KL + \ctrlf(0))$. We will finish the proof by showing that
    $$
    d_X^{\Haus}(F_{i,**}(x_0, \dots, x_i), F_i(x_0, \dots, x_i)) \le C.
    $$
    To prove this,  note that
    $$
    d_X(x_j, x'_j) = d_X(x_j, g(f(x_j))) \le 2KL.
    $$
    \cref{lem: lipschitz property} implies that 
    \begin{align*}
        d_X^{\Haus}(F_i(x'_0, \dots, x'_i), F_i(x_0, \dots, x_i)) &\le (i+1)\ctrlf(2KL + \ctrlf(0)).
    \end{align*}

    On the other hand, since $g\circ f$ is a $(1,2KL)$-quasi-isometry (by the definition of the quasi-inverse $g$), \cref{prop:coarse containment under QI}(2) implies that 
    \begin{equation*}
    d_X^{\Haus}(g\circ f(F_i(x'_0, \dots, x'_i)), F_i(x'_0, \dots, x'_i)) \le 2KL.
    \end{equation*}
    Then, 
    \begin{align*}
        d_X^{\Haus}&(F_{i,**}(x_0, \dots, x_i), F_i(x_0, \dots, x_i)) \\
        &\le d_X^{\Haus}(g \circ f(F_i(x'_0, \dots, x'_i)), F_i(x'_0, \dots, x'_i)) + d_X^{\Haus}(F_i(x'_0, \dots, x'_i), F_i(x_0, \dots, x_i))\nonumber\\
        &\le 2KL + (r+1)\ctrlf(2KL + \ctrlf(0))=C.\nonumber \qedhere
    \end{align*}
\end{proof}

\subsection{Generic tuples and quasi-isometries}

To verify that this coarse $r$-filling is associated to a coarse $r$-median structure, we have to understand how generic tuples behave under quasi-isometries. The next lemma guarantees that the preimages of generic tuples remain generic, with respect to an affine modification of the parameters.

\begin{lemma}\label{lem: generic points via quasi-isometry}
    Let $f:X\to Y$ be a $(K,L)$-quasi-isometry and $g:Y\to X$ be a quasi-inverse. Assume that $X$ is endowed with a coarse $r$-filling $(F,\ctrlf)$ and consider the push-forward filling $(F_*,\ctrlf_*)$ on $Y$. Let $\delta_* \geq L,D_*\ge0$, and let  $(y_0, \dots, y_{r+1})$ be a $(\delta_*,D_*)$-generic tuple in $Y$. Then $(g(y_0), \dots, g(y_{r+1}))$ is a $(\delta,D)$-generic tuple in $X$, where 
    $\delta:=\frac{1}{K}(\delta_*-L)$ and $D:= K(D_*+L).$
\end{lemma}
\begin{proof}
    Let $x_i\coloneqq g(y_i)$ for $i=0,\dots,r+1$. Suppose by contradiction that the tuple $(x_0, \dots, x_{r+1})$ is not $(\delta, D)$-generic in $X$. Then, by definition of genericity, there exist $I_1,\dots,I_l\subseteq\{0,\dots,r+1\}$ such that
    \begin{itemize}
        \item $\bigcap_{i=1}^l I_i=\emptyset$,
        \item $\diam\left( \bigcap_{i=1}^l\ngh{F_{|I_i|-1}(V_i)}{\delta} \right)>D$, where $V_i:=\{x_s:s \in I_i\}$, and
        \item there exists $j\in\{1,\dots,l\}$ such that $\abs{I_j}\le r$.
    \end{itemize} 
    Let $p,q\in\bigcap_{i=1}^l\ngh{F_{\abs{I_i}-1}(V_i)}{\delta}$ 
    such that $d_X(p,q)>D.$
    Therefore, for any $i=0,\dots,l$ there exist
    $p_i,q_i\in F_{\abs{I_i}-1}(V_i)$ such that 
    $$d_X(p,p_i)\le \delta\quad\text{and}\quad d_X(q,q_i)\le \delta.$$
   Then, as $f$ is a $(K,L)$-quasi-isometry, we have 
    \begin{align*}
        d_Y(f(p),f(p_i))&\le K\delta+L ~\text{ and }~ d_Y(f(q),f(q_i))\le K\delta+L
    \end{align*}
    for any $i=0,\dots,l$, 
    This implies that
    \begin{align*}
        f(p),f(q)&\in \bigcap_{i=1}^l\ngh{f\left(F_{\abs{I_i}-1}(V_i)\right)}{K\delta+L}.
    \end{align*}
    On the other hand, 
    \begin{align*}
        d_Y(f(p),f(q))&\ge \dfrac{1}{K}d_X(p,q)-L> \dfrac{D}{K}-L.
    \end{align*}

    We now claim that this implies: the tuple $(y_0,\dots,y_{r+1})$ is not $(\delta_*, D_*)$-generic, where $\delta_*,D_*$ is as in the statement of this result. Recall that $\delta_*=K\delta+L$ and $D_*=\frac{D}{K}-L$. Indeed, to prove this claim, define, for any $i=1,\dots,l$, the set $V_i^*:=\{y_s: s \in I_i\}$. Then,  $g(V_i^*)=V_i$,  and $F_{|I_i|-1, *}(V^*_i)=f(F_{|I_i|-1}(V_i))$.   Thus $$\diam \left( \bigcap_{i=1}^l\ngh{F_{|I_i|-1,*}(V_i^*)}{\delta_*} \right)>D_*.$$ The last inequality follows because $f(p),f(q) \in \bigcap_{i=1}^l\ngh{F_{|I_i|-1,*}(V_i^*)}{\delta_*}$ and $d_Y(f(p),\allowbreak f(q))>D_*.$ Hence, the subsets $I_1,\dots, I_l \subset  \{0,\dots,r+1\}$ contradict the genericity of $(y_0,\dots,y_{r+1}).$
   
   By the claim above, the tuple $(y_0,\dots,y_{r+1})$ is not $(\delta_*, D_*)$-generic. But this is a contradiction.
\end{proof}

By combining this with the double push-forward operation from Lemma \ref{lem: double push-forward}, we immediately obtain that quasi-isometries map generic tuples to generic tuples.

\begin{corollary}\label{cor: image of generic points is generic}
    Let $f:X\to Y$ be a $(K,L)$-quasi-isometry and $g:Y\to X$ be a quasi-inverse as in \cref{prop:quasi-inverse} (2). Assume that $X$ is endowed with a coarse $r$-filling $(F,\ctrlf)$ and let  $(F_*,\ctrlf_*)$ denote its push-forward by $g$. Fix 
    $$C\coloneqq 2KL+(r+1)\ctrlf(2KL+\ctrlf(0)), ~\delta\ge C+L-\frac{L}{K}, ~\text{ and } D\ge0.$$  If $(x_0, \dots, x_{r+1})$ is a $(\delta,D)$-generic tuple in $X$, then $(f(x_0),\dots, f(x_{r+1}))$ is a $(\delta_*,D_*)$-generic tuple in $Y$, with 
    $\delta_*= \frac{1}{K}(\delta-C)-3L ~\text{and}~ D_*= K(D+3KL).$
\end{corollary}
\begin{proof}
    It follows from Lemma \ref{lem: generic points via quasi-isometry} applied to $F_{**}$ and from Lemma \ref{lem: double push-forward}.
\end{proof}

Finally, we have all the ingredients necessary to prove the main theorem of this section. In the following proposition, we establish the quasi-isometry invariance of coarse $r$-medians.

\subsection{Quasi-isometry invariance of coarse medians}

\begin{proposition}\label{prop: median quasi-isometry}
    Suppose $f:(X,d_X) \to (Y,d_Y)$ is a $(K,L)$-quasi-isometry. If $(X,d_X)$ admits a coarse $r$-median structure, then $(Y,d_Y)$ also admits a coarse $r$-median structure.
\end{proposition}
\begin{proof}
    Assume that $(F, \ctrlf)$ is a coarse $r$-filling on $(X,d_X)$ associated with a coarse $r$-median structure with parameters $\delta_0, \lambda$ and $\Psi$. Fix a quasi-inverse $g$ for $f$ as in \cref{prop:quasi-inverse}(2). Consider the coarse $r$-filling $(F_*,\ctrlf_*)$ obtained by pushing forward $(F,\ctrlf)$ via $g$ (see \cref{def: push-forward of a coarse r-filling}).
    Now, we associate to $(F_*,\Lambda_*)$ a coarse $r$-median structure.
    Define $C\coloneqq 2KL+(r+1)\ctrlf(2KL+\ctrlf(0))$ and the parameters: 
    \begin{align*}
    \delta_{0,*} &\coloneqq  \max\{K\delta_0 + 2L, K\lambda^{-1}(K(K\lambda(\delta_0)+9L))+L\},&\\
    \Psi_*(\delta_*,D_*)&\coloneqq K \left(\Psi\left(\dfrac{\delta_*-L}{K},K(D_*+L)\right)+3KL\right),\quad &\forall\  \delta_*\ge\delta_{0,*}, \text{ and }\\
    \lambda_*(\delta_*)&\coloneqq\frac{1}{K}\lambda\left(\frac{\delta_*-L}{K}\right)-8L,\quad &\forall\  \delta_*\ge\delta_{0,*}.
    \end{align*}
    Note that $\lambda_*$ is affine and non-constant since $\lambda$ is affine and non-constant; and $\Psi_*$ is non-decreasing in both entries since $\Psi$ is non-decreasing in both entries.

    For the rest of this proof, fix $y_0, \dots, y_{r+1} \in Y$. Define $x_i=g(y_i)$ for any $i=0,\dots,r+1$. Then recall that $F_{i,*}(y_0,\dots,y_i)=f(F_i(x_0,\dots,x_i))$ and $d_Y(y_i,f(x_i))\le L$.
        For any $j=0,\dots, r+1$, set $$\Vc_j:=F_r(x_0, \dots, \wh{x}_j, \dots, x_{r+1})$$ and $$\Vc_j^*:=F_{r,*}(y_0,\dots,\wh{y}_j,\dots,y_{r+1})=f(\Vc_j).$$

   We will now verify the conditions (CM1) and (CM2). 
        
        \noindent  (CM1) Let $p \in \bigcap_{j=0}^{r+1} \ngh{\Vc_j}{\lambda(\delta_0)}$. Then, by \cref{prop: quasi-isometry quasi-commutes with taking ngh}, 
        \begin{align*}
            f(p) \in  \bigcap_{j=0}^{r+1} \ngh{f(\Vc_j)}{K\lambda(\delta_0)+L}=\bigcap_{j=0}^{r+1} \ngh{\Vc_j^*}{K\lambda(\delta_0)+L}.
        \end{align*}
        By the choice of $\delta_{0,*}$, we have $\lambda_*(\delta_{0,*}) \geq K \lambda(\delta_0)+L$. Thus 
        $
        \bigcap_{j=0}^{r+1} \ngh{\Vc_j^*}{\lambda_*(\delta_{0,*})}
        $
        is non-empty.
        
       \noindent  (CM2) Let $\delta_*,D_*\ge\delta_{0,*}$.
   
       Consider a set $\{y_0,\dots,y_{r+1}\}$ of $(\delta_*,D_*)$-generic points. For any $i=0,\dots,r+1$ define $x_i\coloneqq g(y_i)$.
       Then, by \cref{lem: generic points via quasi-isometry}, $\{x_0,\dots,x_{r+1}\}$ is a $(\delta,D)$-generic tuple where  
       $\delta=\frac{\delta_*-L}{K}$ and $D= K(D_*+L).$
       Hence, the coarse $r$-median structure on $X$ implies that
       \begin{align}
       \label{eqn:diam_bound_on_Vc_j}
       \diam\left(\bigcap_{j=0}^{r+1} \ngh{\Vc_j}{\lambda(\delta)}\right)\le\Psi(\delta,D).
       \end{align}
       
       \begin{claim}
       \label{claim:main_ineq_in_QI_invariance}
            $\diam \left(\bigcap_{j=0}^{r+1}\ngh{\Vc_j^*}{\frac{\lambda(\delta)-6KL}{K}-2L}\right) \leq K (\Psi(\delta,D)+3KL).$
        \end{claim}
      
        To prove this claim, recall that $f(\Vc_j)=\Vc_j^*$ and  then, \cref{prop: quasi-isometry quasi-commutes with taking ngh} implies that  
        \begin{align*}
         \diam \left(\bigcap_{j=0}^{r+1}\ngh{\Vc_j^*}{\frac{\lambda(\delta)-6KL}{K}-2L}\right)
        \le\diam\left(\bigcap_{j=0}^{r+1}\ngh{f\left(\ngh{\Vc_j}{\lambda (\delta)-6KL}\right)}{L}\right).
        \end{align*}
        Set $\Wc_j:=\left(\ngh{\Vc_j}{\lambda (\delta)-6KL}\right)$ for each $j=0,\dots,r+1$. 
        As $g$ is a $(K,3KL)$-quasi-isometry, \cref{prop:coarse containment under QI}(3) implies that 
        \begin{align*}
            \diam\left(\bigcap_{j=0}^{r+1}\ngh{f(\Wc_j)}{L}\right) 
        \le K \cdot  \diam\left(g\left(\bigcap_{j=0}^{r+1}\ngh{f(\Wc_j)}{L}\right)\right) +3K^2L.\\
        \end{align*}
        Now, in order to prove \cref{claim:main_ineq_in_QI_invariance}, it suffices to show that $$\diam\left(g\left(\bigcap_{j=0}^{r+1}\ngh{f(\Wc_j)}{L}\right) \right) \leq \Psi(\delta,D).$$
        To prove this diameter bound, note that \cref{prop: quasi-isometry quasi-commutes with intersection} and \cref{prop: quasi-isometry quasi-commutes with taking ngh} imply that 
        \begin{align*}
            g\left(\bigcap_{j=0}^{r+1}\ngh{f(\Wc_j)}{L}\right) \subset \bigcap_{j=0}^{r+1}\ngh{g \circ f(\Wc_j)}{KL+3KL}. 
        \end{align*}
        As $g\circ f$ is coarsely $2KL$ close to the identity map on $X$, 
        \begin{align*}
            g\left(\bigcap_{j=0}^{r+1}\ngh{f(\Wc_j)}{L}\right) &\subset \bigcap_{j=0}^{r+1}\ngh{\ngh{\Wc_j}{2KL}}{KL+3KL}\\
            &\subset \bigcap_{j=0}^{r+1}\ngh{\Wc_j}{6KL}=\bigcap_{j=0}^{r+1}\ngh{\Vc_j}{\lambda(\delta)},
        \end{align*}
        where the last equality is immediate from the definition of $\Wc_j$. Then \cref{eqn:diam_bound_on_Vc_j} implies that $\diam\left(g\left(\bigcap_{j=0}^{r+1}\ngh{f(\Wc_j)}{L}\right) \right) \leq \Psi(\delta,D)$. This finishes the proof \cref{claim:main_ineq_in_QI_invariance}.
 
    We now conclude verifying (CM2) by observing that 
    \begin{align*}
    \Psi_*(\delta_*,D_*)
    &=K \left(\Psi\left(\frac{1}{K}(\delta_*-L),K(D_*+L)\right)+3KL\right)
    =K (\Psi(\delta,D)+3KL), ~~\text{ and }\\
    \lambda_*(\delta_*)&=\frac{\lambda\left(\frac{\delta_*-L}{K}\right)-6KL}{K}-2L
    =\frac{\lambda(\delta)-6KL}{K}-2L. \qedhere
    \end{align*}
      \end{proof}

\section{Products of coarse median spaces}\label{sec: products}

In this section, we investigate the behavior of coarse $r$-medians under the product of metric spaces. A fundamental property of coarse medians is that they are stable under direct products. We show that the property of admitting a coarse $r$-median is indeed preserved under finite products when endowed with the $\ell_1$-metric.

\begin{proposition}\label{prop: product of coarse medians}
    Suppose $(X,d_X),(Y,d_Y)$ are two metric spaces that both admit a coarse $r$-median structure. Then $X \times Y$ endowed with the $\ell_1$-product distance admits a coarse $r$-median structure.
    Moreover, if both $X$ and $Y$ admit an $r$-median, then also $X \times Y$ does.
\end{proposition}
\begin{proof}
    For $\diamond \in\{X,Y\}$, let $(F^{\diamond},\ctrlf^{\diamond})$ be a coarse $r$-filling for $\diamond$ associated with a coarse $r$-median structure with parameters $$\delta^{\diamond}_0\ge0,\quad \lambda^{\diamond}:[\delta_0^{\diamond},\infty)\to[0,\infty)\quad\text{and}\quad\Psi^{\diamond}:[\delta_0^{\diamond},\infty)\times[0,\infty)\to[0,\infty).$$
    In particular, the family $F_{\diamond}=\{F_i^{\diamond}:{\diamond}^{i+1}\to\mathcal{P}({\diamond})\}_{i=0,\dots,r}$, with control functions $\ctrlf^{\diamond}$, satisfies conditions (F0)--(F4).
    Define the family of maps $F=\{F_i:(X\times Y)^{i+1}\to\mathcal{P}(X\times Y)\}_{i=0,\dots,r}$ given by 
    for any $i=0,\dots,r$ and any $(x_0,y_0),\dots,(x_i,y_i)\in (X\times Y)$. Moreover, define the map $\ctrlf:\mathbb{R}_{\ge0}\to\mathbb{R}_{\ge0}$ as
    $$
    \ctrlf(R)\coloneqq \ctrlf^X(R)+\ctrlf^Y(R),
    $$
    for all $R\ge0$.
    By definition of the $\ell_1$ product metric, the pair $(F,\ctrlf)$ satisfies properties (F0)--(F4). Hence, $(F,\ctrlf)$ is a coarse $r$-filling.
    Now, we define the constants
    $$
    \delta_0\coloneqq\delta_0^X+\delta_0^Y,
    $$
    and the functions $\lambda:[\delta_0,\infty)\to[0,\infty))$ and $\Psi:[\delta_0,\infty)\times[\delta_0,\infty)\to[0,\infty$ by
    $$
    \lambda(\delta)\coloneqq\lambda^X(\delta)+\lambda^Y(\delta),
    \quad \text{and}\quad 
    \Psi(\delta,D)\coloneqq\Psi^X(\delta,D)+\Psi^Y(\delta,D).
    $$
    Note that since $\lambda^X$ and $\lambda^Y$ are affine, then also $\lambda$ is affine; moreover, since $\Psi^X$ and $\Psi^Y$ are non-decreasing in both entries, $\Psi$ is non-decreasing as well.
    Given a tuple of $(\delta,D)$-generic points for $X\times Y$, if $\delta,D\ge\delta_0$, then the tuples given by the projection on each factor are $(\delta,D)$-generic.
    
    Since both $(F^X,\ctrlf^X,\delta_0^X,\Psi^X)$ and $(F^Y,\ctrlf^Y,\delta_0^Y,\Psi^Y)$ satisfy conditions (CM1) and (CM2), by definition of the $\ell_1$ product metric, we see that $(F,\ctrlf,\delta_0,\lambda,\Psi)$ also satisfies these conditions. Hence, we get a coarse $r$-median structure on $X\times Y$.
    Clearly, if the coarse $r$-medians on $X$ and $Y$ were both $r$-medians, then the result would be an $r$-median on $X\times Y$.
\end{proof}

\section{Applications to symmetric spaces}
\label{sec:appl_to_symm_spaces}

    \noindent \textbf{Proof of \cref{thm:coarse_median_irred_symm_sp}.} Let $\Omega$ be the convex projective model of $G/K$, i.e. there exist an isomorphism $f:G\to \Aut(\Omega^0)$ and a diffeomorphism $\phi:G/K \to \Omega$ such that $\phi (g \cdot)=f(g)\phi(\cdot)$. Let $d$ be a $G$-invariant distance function of $G$ as in the statement of the theorem. We equip $\Omega$ with the $\Aut(\Omega)^0$-invariant distance $d_\phi(x,y):=d(\phi^{-1}(x),\phi^{-1}(y))$ for all $x,y \in \Omega$. 
    
    First $\Aut(\Omega)^0$ acts transitively on $\Omega$. So, by \cref{thm:main_coarse_r_median_on_pcd}, $(\Omega,\hil)$ admits a coarse $r$-median for any $r \geq \rF(\Omega)$. By \cref{lem: flat rank for symmetric spaces}, $\rF(\Omega)=\rk(G)$.  Moreover, by \cref{lem:riem_QI_hilbert}, $(\Omega,\hil)$ is quasi-isometric to $(\Omega,d_{\phi})$. Then \cref{prop: median quasi-isometry} implies that $(\Omega,d_{\phi})$ carries a coarse $r$-median structure for any $r \geq \rk(G)$. Since $\phi:(G/K,d) \to (\Omega,d_{\phi})$ is an isometry, $(G/K,d)$ carries a coarse $r$-median structure for any $r \geq \rk(G).$

    \noindent \textbf{Proof of \cref{thm:coarse_median_red_symm_sp}.} By \cref{thm:coarse_median_irred_symm_sp}, each factor $(G_i/K_i,d_i)$ carries a coarse $r$-median structure for any $r\geq \rk(G_i)$. Let $d'=d_1+\dots+d_n$ be the $\ell_1$-distance on $Y=\prod_{i=1}^n (G_i/K_i). $  By \cref{prop: product of coarse medians}, $(Y,d')$ admits a coarse $r$-median for any $r \geq r_*=\max_{i=1}^n \rk(G_i)$. 

    Let $d$ be a $\prod_{i=1}^n G_i$-invariant distance on $Y$ as in the statement of the theorem. Then $(Y,d)$ is biLipschitz with $(Y,d')$. So, by \cref{prop: median quasi-isometry}, $(Y,d)$ carries a coarse $r$-median structure for any $r \geq r_*$.

\section{Ultralimits of coarse median spaces}\label{sec: ultralimits}
In this section, we investigate how coarse $r$-median structures behave under taking ultralimits. We will show that under suitable uniform bounds on the parameters, the ultralimit of a sequence of metric spaces, each endowed with a coarse $r$-median structure, inherits a coarse $r$-median structure. First we recall the standard definition of the ultralimit of a sequence of metric spaces; see \cite[Section 10.4]{drutuKapovic} for details.

Let $\omega$ be a non-principal ultrafilter on $\Nb$ and let \Seq{X_n,d_n} be a sequence of metric spaces. We will write $\mathbf{a}=(a_n)_{n\in\mathbb{N}}$ for an element of $\prod_n X_n$.
The product space $\prod_n X_n$ carries an extended metric $d_\omega$ given by 
$$
d_\omega(\mathbf{a},\mathbf{b})\coloneqq\omegalim d_n(a_n,b_n) \in [0,\infty].
$$
Recall that if $\seq{\alpha_n}$ is a sequence of real numbers, then there is a unique $\omegalim \alpha_n \in [-\infty,\infty]$ defined by: for every open neighborhood $U$ of $\omegalim \alpha_n$ in $[-\infty,\infty]$, $\{n \in \Nb: \alpha_n \in U\} \in \omega$. Moreover, recall that if there is an $r \in \Rb$ such that $\{n \in \Nb \mid \alpha_n \leq r\} \in \omega$, then $\omegalim \alpha_n \leq r$. Similarly, $\{n \in \Nb \mid \alpha_n \geq r\} \in \omega$ implies $\omegalim \alpha_n \geq r$.

Fix a base point $\mathbf{o}=(o_n)_{n\in\mathbb{N}}\in\prod_n X_n$. Then, we consider the set
$
X_\omega^0\coloneqq\{\mathbf{a}\in\prod_n X_n\mid d_\omega(\mathbf{a},\mathbf{o})<\infty\}.
$
Indeed, $\mathbf{a} \in X^0_{\omega}$ if and only if there exists $r=r(\mathbf{a})\geq 0$ such that $\{n \in \Nb: d_n(o_n,a_n)<r\} \in \omega$.

We notice that $d_\omega(\mathbf{a},\mathbf{b})$ is finite for all $\mathbf{a},\mathbf{b}\in X_\omega^0$. On $X_\omega^0$, consider the equivalence relation: $\mathbf{a}\sim\mathbf{b}$ if and only if $d_\omega(\mathbf{a},\mathbf{b})=0$. Then we define
$$
X_\omega\coloneqq X_\omega^0/\sim.
$$
In fact, $d_\omega$ is an actual distance on $X_\omega$. The metric space $(X_\omega,d_\omega)$ is called \emph{ultralimit} of the sequence \Seq{X_n,d_n} \wrt\ the ultrafilter $\omega$ and the base point $\mathbf{o}$. Moreover $(X_\omega,d_\omega)$ is a complete metric space.
\begin{remark}\label{rmk: ultralimit dependence on the paramenters}
    We notice that the ultralimit depends on the choice of $\omega$. On the other hand, the choice of two different base points $\mathbf{o}=(o_n)_{n\in\mathbb{N}}$ and $\mathbf{o'}=(o'_n)_{n\in\mathbb{N}}$ produces isometric ultralimits, provided $\omegalim{d_n(o_n,o'_n)}<\infty$.
\end{remark}

\begin{notation}
    Let $\omega\subset\mathcal{P}({\mathbb{N}})$ be a non-principal ultrafilter. Then, we will use the following convention.
    \begin{itemize}
        \item Let $\seq{x_n}\subset\mathbb{R}$ be a sequence of real numbers. We say that $\seq{x_n}$ is $\omega$-bounded if there exists $k\ge0$ such that $\{n\in\mathbb{N}\mid \abs{x_n}\le k\}\in\omega$. In this case, $(\omegalim x_n)$ exists and $|\omegalim x_n| \leq k$.
        \item Let $\seq{f_n:\mathbb{R}\to \mathbb{R}}$ be a sequence of functions. We say that $\seq{f_n}$ is $\omega$-bounded if for every $x\in\mathbb{R}$, the sequence $\seq{f_n(x)}$ is $\omega$-bounded. In this case, there is a well-defined function $f_{\omega}:\Rb \to \Rb$ defined by
        \begin{equation*}
            f_{\omega}(x):=\omegalim f_n(x),
        \end{equation*}
        which is called the \emph{pointwise} $\omegalim$ \emph{of the sequence} $\seq{f_n}$. 
    \end{itemize}
\end{notation}
\begin{lemma}\label{lem: omega-limit of sequences of functions}
Let $\omega\subset\mathcal{P}({\mathbb{N}})$ be a non-principal ultrafilter.
Let $\seq{f_n:\mathbb{R}_{\ge0}\to \mathbb{R}_{\ge0}}$ be an $\omega$-bounded sequence of functions. Then, the following hold.
\begin{enumerate}
\item The pointwise $\omegalim$ of $\seq{f_n}$, $f_{\omega}: \Rb_{\geq 0} \to \Rb_{\geq 0}$, exists.  
\item If $f_n$ is non-decreasing for all $n\in\mathbb{N}$, then $f_\omega$ is non-decreasing
\item If $f_n$ is affine for all $n\in\mathbb{N}$, then $f_\omega$ is affine
\end{enumerate}
\end{lemma}

\begin{proof}

\noindent (1) It suffices to show that $f_{\omega}(x)$ is non-negative. If $x \in \mathbb{R}_{\ge 0}$, then $f_n(x)\geq 0$ and thus, $f_{\omega}(x)=\omegalim f_n(x)\ge 0$.

\noindent (2)
Assume that $f_n$ is non-decreasing, for every $n \in \mathbb{N}$. Let $x, y \in \mathbb{R}_{\ge 0}$ be such that $x \le y$. 
Because each $f_n$ is non-decreasing, we have $f_n(x) \le f_n(y)$ for all $n \in \mathbb{N}$. 
Since $\omega$ is an ultrafilter, $\{n \in \mathbb{N} \mid f_n(x) \le f_n(y)\} \in \omega$. 
Applying the $\omega$-limit to both sides, $f_\omega(x)=\omegalim f_n(x) \le \omegalim f_n(y)=f_\omega(y).$

\noindent (3)
Since $f_n$ is affine for every $n \in \mathbb{N}$, there exist $a_n,b_n\ge0$ such that $f_n(x)=a_nx+b_n$ for any $x\in \mathbb{R}_{\ge 0}$ and $n\in\mathbb{N}$. Note that $b_n=f_n(0)$ and $a_n=f_n(1)-f_n(0)$. 
Because $\seq{f_n}$ is $\omega$-bounded, the sequences $\seq{f_n(0)}$ and $\seq{f_n(1)}$ are $\omega$-bounded. So there exist $a,b\ge0$ such that $\omegalim b_n=b$ and $\omegalim a_n=a$. Then, for any $x\ge0$, \begin{equation*}
f_\omega(x) = \omegalim f_n(x)= \omegalim \left( a_n x + b_n\right)= ax+b.\qedhere
\end{equation*}
\end{proof}

The main part of this section will be to prove that: given a sequence of metric spaces \Seq{X_n,d_n} with a coarse $r$-median structure with parameters 
$$(F^n,\ctrlf_n,\delta_0^n, \lambda_n:[\delta_0^n,\infty)\to[0,\infty),\Psi_n:[\delta_0^n,\infty)\times[\delta_0^n,,\infty)\to[0,\infty)),$$
the ultralimit $(X_\omega,d_\omega)$ has a coarse $r$-median structure, provided the following conditions are satisfied:
\begin{itemize}[leftmargin=14pt]
    \item $\seq{\ctrlf_n}$ is $\omega$-bounded as a sequence of real valued functions;
    \item $\seq{\delta_0^n}$ is $\omega$-bounded;
    \item $\seq{\lambda_n}$ is $\omega$-bounded as a sequence of real valued functions;
    \item $\seq{\Psi_n}$ is $\omega$-bounded \wrt\ both entries as a sequence of real valued functions;
\end{itemize}

\begin{notation}
\label{notation:bounded_parameters}
    We will say that the parameters of $(F^n,\ctrlf_n,\delta_0^n,\lambda_n,\Psi_n)_{n\in\mathbb{N}}$ are \emph{$\omega$-bounded} if the above conditions hold.
\end{notation}

We will build up the coarse $r$-median on $(X_\omega,d_\omega)$ in several steps, as we did in \cref{sec: quasi-isom}. The first step is \cref{prop: filling for ultralimit}: construct a coarse $r$-filling on $X_\omega$ by taking the ultralimit of the fillings $F^n$. The next step is \cref{lem: generic points ultralimit}: study the behavior of sequences that approximate generic points in $X_{\omega}$. In the final step, \cref{prop: median for ultralimit}, we define a coarse $r$-median on $X_\omega$. 

Let us now fix the setup that we will be working in for the rest of this section.

\medskip

\noindent \hypertarget{link:setup}{\textbf{Setup:}} \emph{Fix $r \in \mathbb{N}$. We further fix the following:}

\begin{itemize}
    \item \emph{A sequence of metric spaces with each $(X_n,d_n)$ admitting a coarse $r$-filling $(F^n,\ctrlf_n)$ where $F^n=\{F^n_i:X_n^{i+1}\to\mathcal{P}(X_n)\}_{i=0,\dots,r}$, and $\seq{\ctrlf_n}$ is $\omega$-bounded.} 
    \item \emph{A non-principal ultrafilter $\omega$ and a base point $\mathbf{o}=(o_n)_{n\in\mathbb{N}}$, and then, consider the ultralimit $(X_\omega,d_\omega)$.}
    \item \emph{For any $\mathbf{a}\in X_\omega$, we fix a sequence $(a_n)_{n\in\mathbb{N}}$ such that $\mathbf{a}=[(a_n)_{n\in\mathbb{N}}]$. We say that the \emph{sequence $(a_n)$ represents} $\mathbf{a}$.}
    \item \emph{ In addition, assume that there is a non-decreasing function $\ctrlf_{\omega}:[0,\infty) \to [0,\infty)$, which is the pointwise $\omega$-limit of $\seq{\ctrlf_n}$ (\cref{lem: omega-limit of sequences of functions}). }
    \item \emph{The constant $C_0\coloneqq\lim_{\varepsilon\to0^+}(r+1)(\ctrlf_\omega(\varepsilon+\ctrlf_\omega(0)))$. The reason behind the existence of this limit is explained at the end of the proof of \cref{prop: filling coarsely well defined}. }
    
\end{itemize}

\begin{definition}
\label{def: ultralimit of fillings}
   Let $i\in\{0,\dots,r\}$ and suppose $\mathbf{a_0},\dots,\mathbf{a_i}\in X_\omega$ are represented by $(a^0_n)_{n\in\mathbb{N}},\dots,(a^i_n)_{n\in\mathbb{N}}$, respectively. Define $F^{\omega}\coloneqq\{F^{\omega}_i:X_\omega^{i+1}\to\Pc(X_\omega)\}_{i=0,\dots,r}$ via 
    \begin{equation*}
        F^{\omega}_i(\mathbf{a_0},\dots,\mathbf{a_i})\coloneqq\{\mathbf{p}=[(p_n)_{n\in\mathbb{N}}]\mid p_n\in F_i^n(a_n^0,\dots,a_n^i)\},
    \end{equation*}
    We call $F^{\omega}$ the \emph{ultralimit of the sequence of coarse $r$-fillings} $(F^n,\ctrlf_n)$. 
\end{definition}
We now show that, since $\seq{\ctrlf_n}$ is $\omega$-bounded, the sets defined in \cref{def: ultralimit of fillings} are coarsely well-defined.
\begin{proposition}\label{prop: filling coarsely well defined}
    Let $i\in\{0,\dots,r\}$ and $\mathbf{a_0},\dots,\mathbf{a_i}\in X_\omega$. The sets $F^{\omega}_i(\mathbf{a_0},\dots,\mathbf{a_i})$ satisfy the following: 
	\begin{enumerate}[label=(\roman*)]
        \item  $d_\omega(\mathbf{o},\mathbf{p})$ is finite for every $\mathbf{p} \in F^{\omega}_i(\mathbf{a_0},\dots,\mathbf{a_i})$, and 
        \item  $F^{\omega}_i(\mathbf{a_0},\dots,\mathbf{a_i})$ is coarsely independent of the choice of the sequences representing $\mathbf{a_0}, \dots ,\mathbf{a_i}$, i.e. if $\mathbf{a_j}$ is represented by both $(a^j_n)$ and $(b^j_n)$ for each $j\in \{0,\dots,k\}$, then 
        $$
        d_{\omega}^H(\{[(p_n)_n] \in X_\omega| p_n\in F^n_i(a^0_n,\dots,a^i_n) \}, \{[(p_n)_n] \in X_\omega| p_n\in F^n_i(b^0_n,\dots,b^i_n) \}) \leq C_0.
        $$        
    \end{enumerate}
\end{proposition}
\begin{proof}
    \noindent (\textit{i}) Let $\mathbf{p} \in F^{\omega}_i(\mathbf{a_0},\dots,\mathbf{a_i})$. Then $\mathbf{p}=[(p_n)_{n\in\mathbb{N}}]$ where $p_n\in F_i^n(a_n^0,\dots,a_n^i)$ for each $n\in \Nb$.
    
    Since $\mathbf{a_j}\in X_\omega$, for any $j \in \{0,\dots,i\}$ there exists $k_j>0$ such that 
    \begin{equation*}\label{eqn: filling ultralimit 1}
        \{n\in\mathbb{N}\mid d_n(o_n,a^j_n)\le k_j\}\in\omega.
    \end{equation*}
    Define $k=\max_{0\leq j \leq i} k_j$. Since $\omega$ is closed under taking supersets, 
    $$
    \{n\in\mathbb{N}\mid d_n(o_n,a^j_n)\le k\}\in\omega.
    $$
    Set $E:=\bigcap\limits_{j=0}^i\{n\in\mathbb{N}\mid d_n(o_n,a^j_n)\le k\}$. Then $E\in\omega$.
    It follows from \cref{lem: lipschitz property} and (F3) that for any $n\in E$,
    \begin{align}
    F_i^n(a_n^0,\dots,a_n^i)&\subseteq\ngh{F_i^n(o_n,\dots,o_n)}{(i+1)(\ctrlf_n(k+\ctrlf_n(0)))}\nonumber\\
    &\subseteq\ngh{o_n}{(i+1)(\ctrlf_n(k+\ctrlf_n(0))+\ctrlf_n(0))}.\nonumber
    \end{align}
    Let $c_n\coloneqq(i+1)(\ctrlf_n(k+\ctrlf_n(0))+\ctrlf_n(0))$.
    Then, for all $n\in E$, $$d_n(o_n,p_n)\le c_n.$$
    Since $\seq{\ctrlf_n}$ is $\omega$-bounded and non-decreasing, there exists $c>0$ such that $\{n\in \Nb\mid c_n\leq c\}\in\omega$.
    Since $\omega$ is non-principal, closed under taking supersets, and closed under intersections, we get that 
    $
    \{n\in\mathbb{N}\mid d_n(p_n,o_n)\le c\}\in\omega.
    $
    Hence $d_\omega(\mathbf{o},\mathbf{p})<\infty$.

    (\textit{ii}) We consider another collection of sequences $(\Tilde{a}^0_n)_{n\in\mathbb{N}},\dots,(\Tilde{a}^i_n)_{n\in\mathbb{N}}$ representing $\mathbf{a_0},\dots,\mathbf{a_i}$.    
    By definition $\omegalim d_n(a_n^j,\Tilde{a}_n^j)=0$ for each $j \in\{ 0,\dots,i\}.$ 
    Consequently, for any $\varepsilon>0$ 
    $$E_\varepsilon\coloneqq\{n\in\mathbb{N}\mid d_n(a^0_n,\Tilde{a}^0_n) < \varepsilon,\dots,d_n(a^i_n,\Tilde{a}^i_n)<\varepsilon\}\in\omega.$$
    Now let
    \begin{equation*}
        \Fc^n_i:=F^n_i(a^0_n,\dots,a^i_n) \text{ and } \wt{\Fc}^n_i:=F^n_i(\wt{a}^0_n,\dots,\wt{a}^i_n). 
    \end{equation*}
    Then, \cref{lem: lipschitz property} implies that for any $n\in E_\varepsilon$, $d_n^{\Haus}(\Fc_i^n,\wt{\Fc}_i^n)< C_n(\varepsilon),$
    where $C_n(\varepsilon)\coloneqq (i+1)(\ctrlf_n(\varepsilon+\ctrlf_n(0)))$.
    Since $\omega$ is closed under taking supersets, for any $\varepsilon>0$, we have
    \begin{equation*}\label{eqn: bunded haus distance of filling in the ultralimit}
         \{n\in\mathbb{N}\mid 
         d_n^{\Haus}(\Fc_i^n,\wt{\Fc}_i^n)< C_n(\varepsilon)\}\in\omega.
    \end{equation*}
    Since $\seq{\ctrlf_n}$ is non-decreasing and $\omega$-bounded, the same holds for the seqeunce of functions $\seq{C_n}$. Then, by \cref{lem: omega-limit of sequences of functions}, we conclude that $\seq{C_n}$ pointwise converges to a non-decreasing function $C:\mathbb{R}_{\ge0}\to\mathbb{R}_{\ge0}$.
   Then, since $\omega$ is a  non-principal ultrafilter, 
    \begin{equation}\label{eqn: bunded haus distance of filling in the ultralimit 2}
         \{n\in\mathbb{N}\mid 
         d_n^{\Haus}(\Fc_i^n,\wt{\Fc}_i^n)< C(\varepsilon)+\varepsilon\}\in\omega,
    \end{equation}
    for any $\varepsilon>0$. 
    
    Consider the subsets of $X_{\omega}$ defined by $\Fc^{\omega}_i:=\{[(p_n)_n] \in X_\omega| ~p_n\in \Fc^n_i \}$ and $\wt{\Fc}^{\omega}_i:=\{[(\wt{p}_n)_n] \in X_\omega| ~\wt{p}_n\in \wt{\Fc}^n_i \}$. Then, \cref{eqn: bunded haus distance of filling in the ultralimit 2} implies that 
    \begin{equation*}
        d_{\omega}^H(\Fc^{\omega}_i, \wt{\Fc}^{\omega}_i) \leq C(\varepsilon)+\varepsilon.
    \end{equation*}
    Since $C$ is non-decreasing, $C_0\coloneqq\lim\limits_{\varepsilon\to0^+}C(\varepsilon)$ exists. Moreover, as $\varepsilon$ is arbitrary in the above inequality, we have  
    \begin{equation*}
    \label{eqn:F_infty_coarsely_F_hat_infty}
         d_{\omega}^H(\Fc^{\omega}_i, \wt{\Fc}^{\omega}_i) \leq C_0. \qedhere
    \end{equation*}
    \end{proof}

\begin{proposition}\label{prop: filling for ultralimit}
     The pair $(F^{\omega}, \ctrlf_{\omega})$ (see \cref{def: ultralimit of fillings}) is a coarse $r$-filling on the ultralimit $(X_\omega,d_\omega)$. Moreover, if $\ctrlf_\omega(0)=0$, then $(F^{\omega}, \ctrlf_{\omega})$ is an $r$-filling.
\end{proposition}
\begin{proof}
  We claim that since $(F^n,\ctrlf_n)$ satisfies (F0)--(F4), then $(F^{\omega},\ctrlf_\omega)$ also satisfies (F0)--(F4), and hence it is a coarse $r$-filling. We will only prove the case of (F0); the other cases follow by similar reasoning using the definition of ultralimit.

    Let $\mathbf{a_0}\in X_\omega$ and let \Seq{a_n^0}\ represent $\mathbf{a_0}$. Since $(F^n,\ctrlf_n)$ satisfies (F0), for any $n\in\mathbb{N}$, we have
    $$
    d^{\Haus}_n(a_n^0,F_0^n(a_n^0))\le\ctrlf_n(0).
    $$
    Then, passing to the ultralimits
    $$
    d^{\Haus}_{\omega}(\mathbf{a_0}, F^{\omega}_0(\mathbf{a_0}))=\omegalim d^{\Haus}_n(a^0_n, F^n_0(a^0_n)) \le\ctrlf_\omega(0)=\omegalim \ctrlf_n(0).
    $$
    Thus (F0) is satisfied. 

    ``Moreover part": Since $\ctrlf_{\omega}(0)=0$, then by definition $(F^{\omega},\ctrlf_\omega)$ is an $r$-filling.
\end{proof}

With a well-defined coarse $r$-filling on the ultralimit now in hand, we must analyze the behavior of sequences of tuples approaching generic tuples.

\begin{lemma}\label{lem: generic points ultralimit}    
    Consider the \hyperlink{link:setup}{Setup}. Let $\varepsilon>0$ 
    and fix $\delta,D\ge0$. Let $\mathbf{a_0},\dots,\mathbf{a_{r+1}}\in X_\omega$ be $(\delta,D)$-generic \wrt\ the coarse $r$-filling $(F^{\omega},\ctrlf_\omega)$ and let $\seq{a^i_n}_{n\in \Nb}$ represent $\mathbf{a_i}$ for $i=0,\dots,r$. Then, 
    $$\{n\in\mathbb{N}\mid (a^0_n,\dots,a^{r+1}_n) \text{ is a  $(\max\{\delta-\varepsilon,0\},D+\varepsilon)$-generic $(r+2)$-tuple}\}\in \omega.$$        
\end{lemma}
\begin{proof} Take $\delta'=\max\{\delta-\varepsilon,0\}$.
    Assume by contradiction that
    $$
    \{n\in\mathbb{N}\mid (a^0_n,\dots,a^{r+1}_n) \text{ is a $(\delta',D+\varepsilon)$-generic tuple}\}\not\in \omega.
    $$
    Let us denote by $E_{\varepsilon}$ the complement of this set, i.e. $E_\varepsilon$ consists of all $n\in\mathbb{N}$ such that $(a^0_n,\dots,a^{r+1}_n)$ is not a $(\delta',D+\varepsilon)$-generic tuple. 
    Since $\omega$ is an ultrafilter, we have
    $$
    E_{\varepsilon} \in \omega.
    $$
    \begin{claim}\label{claim: ultralimit generic 0}
        There exist $k\in\mathbb{N}$ and $I_{1},\dots,I_{k}\subset\{0,\dots,r+1\}$ non-empty so that
        \begin{itemize}
            \item $\bigcap_{i=1}^kI_i=\emptyset$,
            \item there exists $j\in\{1,\dots,k\}$ such that $\abs{I_j}\le r$, and
            \item if $V_{1,n}\coloneqq (a^i_n\mid i\in I_1),\dots,V_{k,n}\coloneqq(a^i_n\mid i\in I_k)$, then 
        \begin{equation}\label{eqn: claim ultralimit generic 0}
            \left\{n\in\mathbb{N}\biggm| {\diam}_{d_n}\left(\bigcap_{j=1}^k \ngh{F^n_{\abs{I_j}-1}(V_{j,n})}{\delta'}\right)>D+\varepsilon\right\}\in \omega.
        \end{equation}
        \end{itemize}
    \end{claim}
    \begin{proof}[Proof of the claim]
        Let $\Sc$ be the set of all families of non-empty subsets of $\{0,\dots,r+1\}$ with global empty intersection and at least one element of the family having cardinality at most $r$. 
        For any $\Ic=\{I_1,\dots,I_k\}\in\Sc$, denote by $V_{\Ic}^j(n)\coloneqq(a_n^i \mid i\in I_j)$ for any $j=0,\dots,k$ and $n\in\Nb$.
        Then, by definition of $(\delta',D+\varepsilon)$-genericity, the set $E_\varepsilon$ is given by
        $$
        E_\varepsilon=\bigcup_{\Ic=\{I_1,\dots,I_k\}\in\Sc}\left\{
        n\in\Nb \Biggm| {\diam}_{d_n}\left(\bigcap_{j=1}^k \ngh{F^n_{\abs{I_j}-1}(V_{\Ic}^j(n))}{\delta'}\right)>D+\varepsilon
        \right\}.
        $$
        Since this union is finite and $E_\varepsilon\in\omega$, at least one set in this union -- say $E''$ -- must belong to $\omega$. Let $\Ic \in \Sc$ be the element that corresponds to $E''$.  This $\Ic$ satisfies the claim.
    \end{proof}  

    Let $k\in\mathbb{N}$ and $I_{1},\dots,I_{k}\subset\{0,\dots,r+1\}$ be as in the statement of \cref{claim: ultralimit generic 0}. Define $V_i\coloneqq(\mathbf{a_j}\mid j\in I_i)$. Then, passing to the ultralimit, we get from \cref{eqn: claim ultralimit generic 0} that
    $$
    {\diam}_{d_\omega}\left(
    \bigcap_{i=1}^k 
    \left\{
    \mathbf{x}\in X_\omega\mid d_\omega\left(
    \mathbf{x},\left\{\left[(p_n)_n\right]\mid p_n\in F^n_{\abs{I_i}-1}(V_{i,n})\right\}\le \delta'\right)
    \right\}
    \right)\ge D+\varepsilon,
    $$
    and hence
    $$
    {\diam}_{d_\omega}\left(\bigcap_{i=1}^k \ngh{F^\omega_{{\abs{I_{i}}-1}}(V_{i})}{\delta}\right)> D.
    $$
    This contradicts the fact that $(\mathbf{a_0},\dots,\mathbf{a_{r+1}})$ is $(\delta,D)$-generic.
\end{proof}

We are now ready to assemble these tools. By combining the ultralimit filling from Proposition \ref{prop: filling for ultralimit} with the stability of generic points from Lemma \ref{lem: generic points ultralimit}, we establish our main result for this section. 

\begin{proposition}\label{prop: median for ultralimit}
   Consider the \hyperlink{link:setup}{Setup}. 
    Assume that each $(X_n,d_n)$ admits a coarse $r$-median structure with $\omega$-bounded parameters $(F^n,\ctrlf_n,\delta_0^n,\lambda_n,\Psi_n)$ (see \cref{notation:bounded_parameters}). Then $(X_\omega,d_\omega)$ admits a coarse $r$-median structure associated to $(F^\omega,\ctrlf_\omega)$ (see \cref{prop: filling for ultralimit}).

    Moreover, if $\ctrlf_\omega(0)=\lim_{\varepsilon\to0^+}\ctrlf_\omega(\varepsilon)=\omegalim\delta^n_0=\omegalim\lambda_n(0)=\lim_{\varepsilon\to0^+}\omegalim\allowbreak\Psi_n(\varepsilon,\varepsilon)=0$, then $(X_\omega,d_\omega)$ admits an $r$-median.
\end{proposition}
\begin{proof}    
    We will show that $(X_\omega,d_\omega)$, equipped with the coarse $r$-filling $(F^{\omega},\ctrlf_{\omega})$ (see \cref{def: ultralimit of fillings}), has a coarse $r$-median structure. Recall, we assume that any $\mathbf{a}\in X_\omega$ comes with a fixed sequence $(a_n)\in\prod_nX_n$ representing it, i.e., $\mathbf{a}=\left[(a_n)_n\right]$.

    \noindent \hypertarget{link:params}{\textbf{Parameters:}} 
    
    \begin{itemize}
        \item \emph{$\delta_0^\omega\coloneqq\omegalim\delta^n_0+C_0$, where $C_0=\lim_{\varepsilon\to0^+}(r+1)(\ctrlf_\omega(\varepsilon+\ctrlf_\omega(0)))$},
        \item \emph{$\lambda_\omega:[\delta_0^\omega,\infty)\to[0,\infty)$ defined by $\lambda_\omega(\delta)=\omegalim\lambda_n(\delta)$}, and
        \item \emph{$\Psi_\omega:[\delta_0^\omega,\infty)\times[\delta_0^\omega,\infty)\to[0,\infty)$ defined by $$\Psi_\omega(R,D)=\lim_{\varepsilon\to0^+}\left(\omegalim\Psi_n(\max\{\delta_0^\omega,R-\varepsilon\},D+\varepsilon)\right).$$} 
    \end{itemize}
     Note that all the limits considered in \hyperlink{link:params}{Parameters} above exist. Indeed, by \cref{lem: omega-limit of sequences of functions}, $\lambda_\omega$ exists and is affine and non-constant. Moreover, $\Psi_\omega$ is non-decreasing in both entries and hence $\Psi_{\omega}(\cdot,\cdot)$ exists.
    
    We prove that (CM1) and (CM2) are satisfied by $(F^{\omega},\ctrlf_\omega,\delta_0^\omega,\lambda_\omega,\Psi_\omega)$.

    \noindent(CM1):  Let $\mathbf{a_0},\cdots, \mathbf{a_{r+1}}\in X_\omega$. We will show that 
    \begin{equation*}
        \bigcap_{j=0}^{r+1} \ngh{F_r^\omega(\mathbf{a_0},\dots,\mathbf{\wh{a_j}},\dots, \mathbf{a_{r+1}})}{\lambda_\omega(\delta_0^\omega)}\neq \emptyset.
    \end{equation*}
    By applying property (CM1) to $X_n$, we can pick for each $n\in\Nb$ a point
    $$
    p_n\in\bigcap_{j=0}^{r+1} \ngh{F_r^n(a_n^0,\dots,\wh{a}_n^j,\dots, a_n^{r+1})}{\lambda_n(\delta_0^n)}.
    $$
    Pick any $j \in \{0,\dots,r+1\}$. Then, for each $n\in \mathbb{N}$, we can choose a point $q_n^j\in F^n_r(a_n^0,\dots,\wh{a_n^j},\dots, a_n^{r+1})$ such that 
    \begin{equation*}
        d_n(p_n,q_n^j)\le \lambda_n(\delta_0^n).
    \end{equation*}
    Consider the points $\mathbf{p}=[(p_n)]$ and $\mathbf{q^j}=[(q^j_n)]$ in $X_\omega$ (\cref{prop: filling coarsely well defined}($i$) implies that these points are indeed in $X_\omega$). Then 
    \begin{equation*}
    \mathbf{q^j}\in F_r^\omega(\mathbf{a_0},\dots,\mathbf{\wh{a_j}},\dots, \mathbf{a_{r+1}}) \text{  and  } d_{\omega}(\mathbf{p},\mathbf{q^j}) \leq \omegalim\lambda_n(\delta^n_0).
    \end{equation*}
    Now note that since all $\lambda_n$ are affine and the sequence $\seq{\lambda_n}$ is $\omega$-bounded, the $\omega$-limit passes through the evaluation of the bounded sequence $\seq{\delta_0^n}$. Hence, we have $\omegalim\lambda_n(\delta^n_0)\le\lambda_\omega(\delta_0^\omega)$.
    This implies
    \begin{equation*}
        \mathbf{p}\in\bigcap_{j=0}^{r+1} \ngh{F_r^\omega(\mathbf{a_0},\dots,\mathbf{\wh{a_j}},\dots, \mathbf{a_{r+1}})}{\lambda_\omega(\delta_0^\omega)}.
    \end{equation*}

    \noindent(CM2): Fix any $\delta,D\ge\delta_0^\omega$. We claim that for any $(\delta,D)$-generic $(r+2)$-tuple  $(\mathbf{a_0},\cdots, \mathbf{a_{r+1}})\in X_\omega^{r+2}$, 
    \begin{equation*}
        \diam_{d_\omega}\Bigg(\bigcap_{j=0}^{r+1} \ngh{F^{\omega}_{r}(\mathbf{a_0},\dots,\mathbf{\wh{a_j}},\dots, \mathbf{a_{r+1}})}{\lambda_\omega(\delta)}\Bigg)\le \Psi_\omega(\delta,D).
    \end{equation*}

    We now prove this claim.
    Let $(a_n^0),\dots, (a_n^{r+1})$ be the representatives of $\mathbf{a_0},\dots, \mathbf{a_{r+1}}$, respectively.
    Let $\mathbf{x}, \mathbf{y}\in \bigcap_{j=0}^{r+1} \ngh{F^{\omega}_{r}(\mathbf{a_0},\dots,\mathbf{\wh{a_j}},\dots, \mathbf{a_{r+1}})}{\lambda_\omega(\delta)}$. Let the sequences $(x_n)_{n\in\mathbb{N}}$ and $(y_n)_{n\in\mathbb{N}}$ represent $\mathbf{x}$ and $\mathbf{y}$ respectively.
    There are two cases, either $\lambda_\omega(\delta)>0$ or $\lambda_\omega(\delta)=0$.
    
    \noindent \underline{First case.} Assume that $\lambda_\omega(\delta)>0$.
    Then, by definition of $F^\omega$, we have
    that
    for any $j=0,\dots,r+1$ there exists $\mathbf{p^j}, \mathbf{q^j}\in X_\omega$ represented respectively by sequences of $p_n^j,q_n^j\in F^n_r(a_n^0,\dots,\wh{a}_n^j,\dots, a_n^{r+1})$
    such that $d_\omega(\mathbf{x},\mathbf{p^j})<\lambda_\omega(\delta)$ and $d_\omega(\mathbf{y},\mathbf{q^j})<\lambda_\omega(\delta)$. In particular, there exists some $\mu>0$ such that 
    $$
    d_\omega(\mathbf{x},\mathbf{p^j})<\lambda_\omega(\delta)-\mu\quad\text{ and }\quad d_\omega(\mathbf{y},\mathbf{q^j})<\lambda_\omega(\delta)-\mu.
    $$
    We will now consider $\varepsilon>0$ sufficiently small so that $(\lambda_{\omega}(\delta)-\mu-\varepsilon)>0.$  Since $\lambda_\omega(\delta)=\omegalim\lambda_n(\delta)$, $d_\omega(\mathbf{x},\mathbf{p^j})=\omegalim d_n(x_n,p_n^j)$, $d_\omega(\mathbf{y},\mathbf{p^j})=\omegalim d_n(y_n,p_n^j)$, and $\omega$ is closed under taking finite intersections, the following hold for any $\varepsilon>0$ sufficiently small:
    \begin{align*}
        \{n\in\mathbb{N}\mid \lambda_n(\delta)>\lambda_\omega(\delta)-\varepsilon\}&\in\omega,\\
        \{n\in\mathbb{N}\mid d_n(x_n,p_n^j)<\lambda_\omega(\delta)-\mu+\varepsilon \ \forall\ j\}&\in\omega\text{, and}\\
        \{n\in\mathbb{N}\mid d_n(y_n,q_n^j)<\lambda_\omega(\delta)-\mu+\varepsilon \ \forall\ j\}&\in\omega.
    \end{align*}
    Since $\omega$ is closed under taking finite intersections, we have 
    \begin{align*}
        \{n\in\mathbb{N}\mid d_n(x_n,p_n^j)<\lambda_n(\delta)-\mu+2\varepsilon \ \forall\ j\}&\in\omega\text{, and}\\
        \{n\in\mathbb{N}\mid d_n(y_n,q_n^j)<\lambda_n(\delta)-\mu+2\varepsilon \ \forall\ j\}&\in\omega.
    \end{align*} 
    By \cref{lem: omega-limit of sequences of functions}, $\lambda_{\omega}$ is a non-constant affine function on $\Rb_{\ge 0}$. So for any $\nu>0$ sufficiently small, and such that $\delta-\nu\ge0$, we can find $\varepsilon>0$ such that 
    $\{n\in\mathbb{N}\mid \lambda_n(\delta)-\mu+2\varepsilon\le\lambda_n(\delta-\nu)\}\in\omega$.
   
    As $\omega$ is closed under taking intersections and supersets,
    for any such $\nu>0$  the following hold:
    \begin{align*}
        E_\nu^\mathbf{x}=\{n\in\mathbb{N}\mid d_n(x_n,p_n^j)<\lambda_n(\delta-\nu) \ \forall\ j\}&\in\omega\text{, and}\\
        E_\nu^\mathbf{y}=\{n\in\mathbb{N}\mid d_n(y_n,q_n^j)<\lambda_n(\delta-\nu) \ \forall\ j\}&\in\omega.
    \end{align*}
    Moreover, by \cref{lem: generic points ultralimit} we have that for any such $\nu>0$
    $$
    E_\nu^{gen}=\{n\in\mathbb{N}\mid (a^0_n,\dots,a^{r+1}_n) \text{ is a  $(\delta-\nu,D+\nu)$-generic $(r+2)$-tuple}\}\in \omega.
    $$
    Indeed, $\max\{0,\delta-\nu\}=\delta-\nu$ because of our choice of $\nu$ as above.
    
    Therefore, for any $n\in E_\nu'\coloneqq E_\nu^\mathbf{x}\cap E_\nu^\mathbf{y}\cap E_\nu^{gen}$, we can apply (CM2)(b) and get that
    $$d_n(x_n,y_n)\le\Psi_n(\delta-\nu,D+\nu),$$
    for any sufficiently small $\nu>0$.

    Consequently, we have 
    $$d_\omega(\mathbf{x},\mathbf{y})\le\omegalim\Psi_n(\delta-\nu,D+\nu)
        \implies
        d_\omega(\mathbf{x},\mathbf{y})\le\lim_{\nu\to0^+}\omegalim\Psi_n(\delta-\nu,D+\nu)=\Psi_\omega(\delta,D).$$
    \noindent \underline{Second case.}
    Assume that $\lambda_\omega(\delta)=0$. Since $\lambda_{\omega}$ is an increasing affine function and $\delta \geq 0$, this implies that $\lambda_\omega$ must be linear and $\delta=0$. 
    This implies that $C_0=\omegalim\delta_0^n=0$, and hence
    $\delta_0^\omega=0$.
    Fix any $\varepsilon>0$.
    \begin{claim}\label{claim:ultralimits_second_case}
        There exists $\eta>0$ such that 
        $$
            E_\eta=\{n\in\mathbb{N}\mid (a_n^0,\dots,a_n^{r+1}) \text{ is a  $(\eta,D+\varepsilon)$-generic $(r+2)$-tuple}\}\in \omega.
        $$
    \end{claim}
    \begin{proof}[Proof of \cref{claim:ultralimits_second_case}]
        Assume this is not true. Then, by \cref{obs: generic_increase_parameters} we have that for any $m\in\mathbb{N}$ it holds that
        $$
            E_{\frac{1}{m}}=\{n\in\mathbb{N}\mid (a_n^0,\dots,a_n^{r+1}) \text{ is a  $(\frac{1}{m},D+\varepsilon)$-generic $(r+2)$-tuple}\}\not\in \omega.
        $$
        Then, as in \cref{claim: ultralimit generic 0}, for any $m\in\mathbb{N}$ there would  exist $k_m\in\mathbb{N}$ and $I^m_{1},\dots,I^m_{k_m}\subset\{0,\dots,r+1\}$ non-empty, such that:
        \begin{itemize}
            \item $\bigcap_{i=1}^{k_m}I^m_i=\emptyset$,
            \item there is a $j\in\{1,\dots,k_m\}$ such that $\abs{I^m_j}\le r$, and
            \item if $V_{1,n}\coloneqq (a^i_n\mid i\in I^m_1),\dots,V_{k_m,n}\coloneqq(a^i_n\mid i\in I^m_{k_m})$, then 
        \begin{equation}
            \left\{n\in\mathbb{N}\biggm| {\diam}_{d_n}\left(\bigcap_{j=1}^{k_m} \ngh{F^n_{\abs{I^m_j}-1}(V_{j,n})}{\frac{1}{m}}\right)>D+\varepsilon\right\}\in \omega.
        \end{equation}
        \end{itemize}
        Up to passing to a subsequence of $m$, we can assume that $k_m=k$ is constant. Then, by the pigeonhole principle, we can also assume that the sets $I^m_1,\dots,I^m_k$ are constant $I_1,\dots,I_k$. This would imply that
        $$
        {\diam}_{d_\omega}\left(\bigcap_{j=1}^{k} F^\omega_{{\abs{I_{j}}-1}}(V_{j})\right)\ge D+\varepsilon>D,
        $$
        where $V_j\coloneqq(\mathbf{a_i}\mid i\in I_{j})$. This would contradict the $(0,D)$-genericity of $(\mathbf{a_0},\dots,\mathbf{a_{r+1}})$.
    \end{proof}
    Let $\eta>0$ be as in \cref{claim:ultralimits_second_case}. By \cref{obs: generic_increase_parameters}, for any $0 < \eta' \le \eta$, the tuple $(a_n^0,\dots,a_n^{r+1})$ is $(\eta',D+\varepsilon)$-generic for all $n \in E_\eta$. 
    Furthermore, since $\omegalim\delta_0^n=0$, the set $F_{\eta'} = \{n\in\mathbb{N}\mid \delta_0^n \le \eta'\}$ is in $\omega$. 
    Therefore, for any $n \in E_\eta \cap F_{\eta'}\in\omega$, we can apply (CM2)(b) to $X_n$ to get that
    $$
        \diam_{d_n}\left(
        \bigcap_{j=0}^{r+1} \ngh{F^n_{r}(a_n^0,\dots,\wh{a}_n^j,\dots, a_n^{r+1})}{\eta'}\right) \leq \Psi_n(\eta',D+\varepsilon).
    $$
    Taking the $\omega$-limit on both sides, we obtain that for any $\eta'>0$ sufficiently small, it holds that
    $$
        \diam_{d_\omega}\left(
        \bigcap_{j=0}^{r+1} \ngh{F^\omega_{r}(\mathbf{a_0},\dots,\wh{\mathbf{a_j}},\dots, \mathbf{a_{r+1}})}{\eta'}\right) \leq \omegalim \Psi_n(\eta',D+\varepsilon).
    $$
    Since $\varepsilon>0$ was chosen arbitrarily, taking $\varepsilon \to 0^+$ and taking $\eta'\to0^+$, we get
    $$
        \diam_{d_\omega}\left(
        \bigcap_{j=0}^{r+1} F^\omega_{r}(\mathbf{a_0},\dots,\wh{\mathbf{a_j}},\dots, \mathbf{a_{r+1}})\right)\le \Psi_\omega(0,D).
    $$

    This finishes the proof of the second case.

    The ``moreover'' statement follows directly from the definition of the $r$-median (\cref{defn:coarse_r_median}). Indeed, since $\Psi_n$ is non-decreasing, we observe that 
    \begin{equation*} \Psi_\omega(0,0) = \lim_{\varepsilon\to0^+}\omegalim\Psi_n(0,\varepsilon) \le \lim_{\varepsilon\to0^+}\omegalim\Psi_n(\varepsilon,\varepsilon) = 0. \qedhere \end{equation*}
\end{proof}

\subsection{Asymptotic cones}
    
An asymptotic cone is a particular type of ultralimit that we are interested in. Let us fix a metric space $(X,d)$, a base point $o\in X$, and a non-principal ultrafilter $\omega$ on $\Nb$. Let us consider a sequence $\nu=(\nu_n)_{n\in\mathbb{N}}\subseteq\mathbb{R}_{\ge0}$ such that $\lim_{n\to\infty}\nu_n=+\infty$.
Consider the rescaled distances in $X$ given by $d_n\coloneqq \frac{d}{\nu_n}$ and the base point $\mathbf{o}=(o_n)_n$, with $o_n:=o$ for all $n$. Then, the ultralimit of the sequence \Seq{X,d_n} is called the \textit{asymptotic cone} of $(X,d)$ \wrt\ $\omega$, $\nu$ and $\mathbf{o}$. We denote it by $\left({\rm Cone}_{\omega}^\nu(X),d_\omega\right)$.
Indeed, the (isometry class of the) asymptotic cone depends on the choice of the ultrafilter $\omega$ and on the choice of the sequence $\nu\in\mathbb{R}^\mathbb{N}$, but not on the choice of the basepoint (\cref{rmk: ultralimit dependence on the paramenters}).
One advantage of asymptotic cones is that it transfers the study of properties of $X$ preserved under quasi-isometries to the study of topological properties of ${\rm Cone}_{\omega}^\nu(X)$.

In the classical coarse median case, the asymptotic cone of a coarse median space is a median space. In our framework, we show that a coarse $r$-median on $(X,d)$ induces an $r$-median on any $\left({\rm Cone}_{\omega}^\nu(X),d_{\omega}\right)$ under a certain technical hypothesis for the coarse $r$-median on $X$.

The result below follows from Proposition \ref{prop: filling for ultralimit} and Proposition \ref{prop: median for ultralimit}.
\begin{corollary}
\label{cor:coarse_median_on_asymp_cone}
Suppose that $(X,d)$ is a metric that admits a coarse $r$-median with parameters $(F,\ctrlf,\delta_0,\lambda,\Psi)$. Assume that
\begin{enumerate}
    \item either the coarse $r$-median on $X$ is $(\ctrlf,\Psi)$-affine, or
    \item $(X,d)$ is a geodesic metric space and the coarse $r$-median on $X$ is $\Psi$-affine.
\end{enumerate}
Then all asymptotic cones of $(X,d)$ admit an $r$-median structure. 
\end{corollary}
\begin{proof} 
    We first observe that in case (2) as well, we can assume that the coarse $r$-median on $X$ is $(\ctrlf,\Psi)$-affine. Indeed, since $(X,d)$ is geodesic, \cref{prop: affine_control_function weak} implies that we can replace $\ctrlf$ by an affine function. This does not impact any of the other parameters of the coarse $r$-median.  

    So, from now on, we may assume that the coarse $r$-median on $X$, with parameters $(F,\ctrlf,\delta_0,\lambda,\Psi)$, is $(\ctrlf,\Psi)$-affine. By definition, $(F,\ctrlf)$ is a coarse $r$-filling on $X$, and let $F=\{F_i:X^{i+1}\to\Pc(X)\}_{i=0,\dots,r}$.
    Consequently, we can define the coarse $r$-filling $(F^n,\ctrlf_n)$ on $(X,d_n)$ given by  
    \begin{equation*}
        F^n_i(a_0,\dots,a_i)=F_i(a_0,\dots,a_i)\quad\text{and}\quad \ctrlf_n(\delta)=\dfrac{\ctrlf(\nu_n \delta)}{\nu_n}, 
    \end{equation*}
     for any $i\in\{0,\dots,r\}$, $(a_0,\dots,a_i)\in X^{i+1}$, and $\delta\ge0$.

    Define the constant $ \delta_{0,n}\coloneqq\frac{\delta_0}{\nu_n}$. Note that since $\nu_n\to \infty$ as $n\to \infty$, $\delta_{0,n}$ is bounded above and converges to 0. Next, consider the two functions 
    \begin{align*}
        [\delta_{0,n},\infty) \ni \delta \mapsto \lambda_n(\delta)\coloneqq \dfrac{\lambda(\nu_n\delta)}{\nu_n} \in \Rb_{\ge0}, \text{ and }\\
        [\delta_{0,n},\infty)\times [\delta_{0,n},\infty) \ni (\delta,D) \mapsto \Psi_n(\delta,D)\coloneqq \dfrac{\Psi(\nu_n\delta,\nu_nD)}{\nu_n} \in\Rb_{\ge 0}.
    \end{align*}
   
    Since $\lambda$ is non-constant affine, there exists $a>0$ such that for all $\delta\ge0$
    $$
    \omegalim\lambda_n(\delta)=a\cdot\delta.
    $$
    Since we also assumed that $\ctrlf$ and $\Psi$ are affine, they can be extended to $\ctrlf:\mathbb{R}\to\mathbb{R}$ and $\Psi:\mathbb{R}\times\mathbb{R}\to\mathbb{R}$, then there exist $A,B,C>0$ such that for any $\delta,D\ge0$ 
    $$
    \omegalim\ctrlf_n(\delta)=A\cdot \delta\quad\text{and}\quad\omegalim\Psi_n(\delta,D)=B\cdot \delta+C\cdot D.
    $$

    This shows that the coarse $r$-median structures on the sequence $\seq{X,d_n}$ have $\omega$-bounded parameters. We will now apply \cref{prop: median for ultralimit}.
    Fix a non-principal ultrafilter $\omega$, a base point $o\in X$, and a divergent sequence $(\nu_n)_{n\in\mathbb{N}}\subset\mathbb{R}_{>0}$. 
    Consider the associated asymptotic cone $\left({\rm Cone}_{\omega}^\nu(X),d_{\omega}\right)$.
    
    Define $\ctrlf_\omega:\mathbb{R}_{\ge0}\to\mathbb{R}_{\ge0},\ \lambda_\omega:\mathbb{R}_{\ge0}\to\mathbb{R}_{\ge0}$ and $\Psi_\omega:\mathbb{R}_{\ge0}\times\mathbb{R}_{\ge0}\to\mathbb{R}_{\ge0}$ as 
    $$
    \ctrlf_\omega(\delta)\coloneqq A\cdot\delta,\quad 
    \lambda_\omega(\delta)\coloneqq a\cdot\delta\quad\text{and}\quad\Psi_\omega(\delta,D)\coloneqq B\cdot \delta+C\cdot D.$$
    According to Proposition \ref{prop: filling for ultralimit} and Proposition \ref{prop: median for ultralimit}, the pair $(F,\ctrlf_\omega)$ is an $r$-filling on the asymptotic cone $(X_\omega,d_\omega)$ associated to an $r$-median with parameters $(F^\omega,\ctrlf_\omega,0,\lambda_\omega,\Psi_\omega)$.
\end{proof}

\printbibliography

\end{document}